\documentclass{amsart}

\usepackage{amssymb,amsmath,esint}
\usepackage[all]{xy}
\usepackage{graphicx}
\usepackage{enumerate}

\usepackage{colortbl}
\usepackage{cite}
\usepackage{xcolor}

\usepackage{amsrefs}

\usepackage{verbatim}

\usepackage{geometry}
\usepackage[title,titletoc]{appendix}

\usepackage{comment}
\usepackage{graphicx,caption,subcaption}
\usepackage{tikz}
\usepackage{multirow} 

\colorlet{linkequation}{blue}
\usepackage[colorlinks,plainpages=true,pdfpagelabels,hypertexnames=true,colorlinks=true,pdfstartview=FitV,linkcolor=blue,citecolor=red,urlcolor=black]{hyperref}

\PassOptionsToPackage{unicode}{hyperref}
\PassOptionsToPackage{naturalnames}{hyperref}

\definecolor{dgreen}{rgb}{0,0.5,0}
\definecolor{violet}{rgb}{0.5,0,0.5}
\definecolor{dred}{rgb}{0.7,0,0}
\definecolor{ddred}{rgb}{0.5,0,0}
\definecolor{dblue}{rgb}{0,0,0.5}
\definecolor{ddblue}{rgb}{0,0,0.3}
\definecolor{llgray}{rgb}{0.9,0.9,0.9}
\definecolor{lgray}{rgb}{0.7,0.7,0.7}

\newtheorem{defn}{Definition}[section]
\newtheorem{lemma}[defn]{Lemma}
\newtheorem{proposition}[defn]{Proposition}
\newtheorem{theorem}[defn]{Theorem}

\newtheorem{remark}[defn]{Remark}
\numberwithin{equation}{section}

\newcommand{\bq}{\begin{equation}}
\newcommand{\eq}{\end{equation}}

\newcommand{\R}{{ \mathbb{R}  }}

\newcommand{\bbr}{{ \mathbb{R}  }}

\newcommand{\calC}{{ \mathcal C  }}

\newcommand{\bke}[1]{\left( #1 \right)}

\newcommand{\norm}[1]{\left\Vert #1 \right\Vert}
\newcommand{\abs}[1]{\left| #1 \right|}

\DeclareMathOperator*{\esssup}{ess \ sup}

\DeclareMathOperator{\loc}{loc}

\AtBeginDocument{%
   \def\MR#1{}
}

\begin{document}

   
\title[Existence of FDE with drifts]{Existence of weak solutions for fast diffusion equation \\
   with a drift term and its application}

\author[S. Hwang]{Sukjung Hwang}
\address{S. Hwang: Department of Mathematics Education, Chungbuk National University, Cheongju 28644, Republic of Korea}
\email{sukjungh@cbnu.ac.kr}

\author[K. Kang]{Kyungkeun Kang}
\address{K. Kang: Department of Mathematics, Yonsei University,  Seoul 03722, Republic of Korea}
\email{kkang@yonsei.ac.kr}

\author[H.K. Kim]{Hwa Kil Kim}
\address{H.K. Kim: Department of Mathematics Education, Hannam University, Daejeon 34430, Republic of Korea}
\email{hwakil@hnu.kr}

\thanks{
S. Hwang's work is partially supported by the research grant of the Chungbuk National University in 2022 and NRF-2022R1F1A1073199. K. Kang is supported by NRF grant nos. RS-2024-00336346 and RS-2024-00406821. H. Kim's work is supported by NRF-2021R1F1A1048231.}

\date{}

\makeatletter
\@namedef{subjclassname@2020}{%
  \textup{2020} Mathematics Subject Classification}
\makeatother
 



\begin{abstract}
 We construct non-negative weak solutions of fast diffusion equations with a divergence type of drift term satisfying the $L^q$-energy inequality and speed estimate in Wasserstein spaces under some integrability conditions on the drift term. Furthermore, in the case that the drift term has a divergence-free structure, it turns out that its integrability conditions can be relaxed, which is also applicable to porous medium equations, thereby improving previous results. As an application, the existence of weak solutions is also discussed for a viscous Boussinesq system of the fast diffusion type.
\end{abstract}

\maketitle


\hypersetup{linkcolor=black}
\setcounter{tocdepth}{2}
\tableofcontents



\section{Introduction}

In this paper, we consider the fast diffusion equation (FDE) with a drift term in divergence form, given by
\begin{equation}\label{FDE}
  \partial_t \rho \;=\; \nabla \cdot\bigl(\nabla \rho^m \;-\; V\,\rho\bigr),
  \quad 0 < m < 1,
\end{equation}
in the cylinder \(\Omega_T := \Omega \times (0, T)\), where 
\(\Omega \Subset \mathbb{R}^d\) with \(d \ge 2\) is a bounded domain with smooth boundaries, and \(0 < T < \infty\). The vector field \(V : \Omega_T \to \mathbb{R}^d\) is given and will be specified later.
We impose no-flux boundary conditions on the lateral boundary and assume nonnegative initial data:
\begin{equation}\label{FDE_bc_ic}
\begin{cases}
(\nabla \rho^m - V\,\rho)\cdot \mathbf{n} = 0 
  &\text{on } \partial\Omega \times (0, T), \vspace{1mm} \\
\rho(\cdot, 0) = \rho_0(\cdot) \in L^q(\Omega)
  &\text{for } q \ge 1,
\end{cases}
\end{equation}
where \(\mathbf{n}\) denotes the outward normal to \(\partial\Omega\).

The fast diffusion equation (FDE) is a fundamental singular nonlinear diffusion model that arises in a wide range of applications. In particular, for the homogeneous fast diffusion equation, i.e., \eqref{FDE} with \( V=0 \), extensive research has been conducted on the local and global properties of solutions.
For instance, if \( \frac{(d-2)_{+}}{d+2} < m < 1 \), then local solutions remain locally bounded. However, when \( 0 < m \leq \frac{(d-2)_{+}}{d+2} \), the concept of weak solutions with sufficient integrability ensures local boundedness. Furthermore, H\"{o}lder continuity and the Harnack inequality are not guaranteed when \( 0 < m \leq (1 - \frac{2}{d})_{+} \); see \cite{Chen_DiBenedetto, DiBenedetto_Kwong, DiBenedetto_Kwong_Vespri} for further details.
For a systematic study of homogeneous FDE, we refer the reader to comprehensive books and papers such as \cite{Vaz06, Vaz07, DGV12, BonFig2024, Bonforte_Vazquez_2006, Bonforte_Vazquez_2007} and the references therein.

Fast diffusion equations with potential drift or aggregation terms arise naturally in the Wasserstein-gradient-flow framework, in particular when $V=\nabla\Phi$ for some potential function $\Phi$; see, for instance, \cite{Santambrosio15}. Recent works on related potential and aggregation models in the fast-diffusion range have revealed concentration phenomena, such as infinite-time or partial mass concentration and possible Dirac components in the large-time limit \cite{CarGVaz2022, CarDFL2022, CarGG2024}. In contrast, the present paper considers general drift fields $V=V(x,t)$ and constructs global weak solutions with $L^\infty(0,T;L^q(\Omega))$ bounds for $q>1$ and entropy bounds for $q=1$.

Drift terms play a key role in modeling various drift-diffusion phenomena, including the Keller--Segel model, fluid dynamics, and related systems~\cite{SSSZ, Freitag, W18, Z11}. More recently, nonlinear drift-diffusion equations have been studied to understand the well-posedness and regularity of their solutions~\cite{KZ18, HZ21, CCY2019}. 

Motivated by the existence theory for porous medium equations with drift developed in \cite{HKK01,HKK02}, the present paper develops a fast-diffusion analogue for general space-time dependent drift fields. However, the transition from the porous medium range $m>1$ to the fast-diffusion range $0<m<1$ is not straightforward.
In the fast-diffusion case, it provides weaker coercive control than in the porous medium case, and the resulting estimates are not strong enough to handle a drift term in the scaling-invariant class. This leads to the restriction \(1-\frac{1}{d}<m<1\) and to a more restrictive integrability condition on \(V\), which is subcritical relative to the scaling-invariant class. More precisely, the admissible drift class \(\mathsf{S}_{m,q}^{(q_1,q_2)}\) in Definition~\ref{D:Serrin} obtained from the energy estimate, for \(q>1\),
\[
V \in L_{x,t}^{q_1, q_2}\qquad
\frac{d}{q_1} + \frac{2 + d(q+m-2)}{q_2} \leq 1 + d(m-1),
\quad 1-\frac{1}{d}<m<1,
\]
is more restrictive than the scaling-invariant class $\mathcal{S}_{m,q}^{(q_1,q_2)}$ in Remark~\ref{R:Scaling} used in the porous medium case,
\[
V \in L_{x,t}^{q_1, q_2}, \qquad
\frac{d}{q_1} + \frac{2 + q_{m,d}}{q_2} \leq 1 + q_{m,d},
\quad q_{m,d}=\frac{d(m-1)}{q},
\quad  m>1.
\]
The two classes coincide in the \(L^1\)-based case \(q=1\).

In the divergence-free case, the drift term does not contribute to the $L^q$-energy estimate, and the admissible conditions on $V$ are determined instead by the compactness argument used to pass to the limit in the weak formulation and by the speed estimate needed to obtain absolute continuity in the Wasserstein space. This allows us to introduce the classes $\mathcal{D}_{m,q}^{(q_1,q_2)}$ and $\mathcal{D}_{m,q,s}^{(q_1,q_2)}$, which contain the corresponding scaling-type region and extend into supercritical regimes. They also allow an enlarged range of \(m\). Under the condition \(\nabla\cdot V=0\), this is expressed as
\[
V \in L_{x,t}^{q_1,q_2}, \qquad
\frac{d}{q_1}+\frac{2+q_{m,d}}{q_2}
\leq 2+\frac{d(q+m-2)}{q},
\quad
\left(1-\frac{2q}{d}\right)_{+}\leq m<1.
\]
We refer to Definitions~\ref{D:Serrin} and \ref{D:Serrin_divfree}, Remark~\ref{R:Scaling}, and Fig.~\ref{F:D} for the precise admissible conditions and for a comparison with the scaling-invariant class.
The analogous improvement for the porous medium equation is based on the improved compactness argument and is discussed in Remark~\ref{R:PME}; the precise statements are given in Appendix~\ref{Appendix:PME} and illustrated in Fig.~\ref{F:divfree:PME}.

The main existence results are summarized in Table~\ref{Table1}, categorized by the structure of \(V\), the corresponding ranges of \(m\) and \(q\), and the type of weak solution obtained. Throughout the paper, we consider three distinct types of drift fields: either \(V\) belongs to the class \(\mathsf{S}_{m,q}^{(q_1,q_2)}\), or \(\nabla V\) belongs to the class \(\tilde{\mathsf{S}}_{m,q}^{(\tilde q_1,\tilde q_2)}\), or \(V\) satisfies the divergence-free condition. The figures are included to make the comparison between the admissible classes more transparent. In particular, Fig.~\ref{F:S} and Fig.~\ref{F:tilde_S:weak} describe the admissible regions for the first two cases, while Fig.~\ref{F:divfree} shows that, in the divergence-free case, the newly introduced classes \(\mathcal{D}_{m,q}^{(q_1,q_2)}\) and \(\mathcal{D}_{m,q,s}^{(q_1,q_2)}\) contain the corresponding scaling-type region and extend into supercritical regimes.

	\begin{table}[hbt!]
\begin{center}
\caption{\footnotesize Summary of existence results of (AC) $L^q$-weak solutions for $d\geq 2$}
\smallskip
{\scriptsize
\begin{tabular}{| c  || c | c | c |}\hline
\rule[-8pt]{0pt}{22pt}
 \textbf{Range of $0<m<1$ (FDE) and $q \geq 1$} &\textbf{Structure of $V$} & \textbf{Solutions}  
 & \textbf{References}\\ \hline \hline

\rule[-8pt]{0pt}{22pt}
\multirow{3}{*}{$1-\frac{1}{d} < m < 1$ \text{and} $q\geq 1$}
& $V\in \mathsf{S}_{m,q}^{(q_1,q_2)}$ in \eqref{T:ACweakSol:V_q}
& AC $L^q$-weak solution 
& Theorem~\ref{T:ACweakSol}, Fig.~\ref{F:S}
\\
\cline{2-4}

\rule[-8pt]{0pt}{22pt}
& $V\in \tilde{\mathsf{S}}_{m,q}^{(\tilde{q}_1,\tilde{q}_2)}$ in \eqref{T:ACweakSol_tilde:V_q}
& AC $L^q$-weak solution
& Theorem~\ref{T:ACweakSol_tilde}, Fig.~\ref{F:tilde_S:weak}
\\
\cline{1-4}

\rule[-8pt]{0pt}{22pt}
\multirow{3}{*}{$(1-\frac{2q}{d})_{+} \leq m < 1$ \text{and} $q \geq 1$}
&  $V\in \mathcal{D}_{m,q}^{(q_1,q_2)}$ in \eqref{T:weakSol:divfree:Vq}
& $L^q$-weak solution
& Theorem~\ref{T:weakSol:divfree} , Fig.~\ref{F:divfree}
\\
\cline{2-4}

\rule[-8pt]{0pt}{22pt}
& $V\in \mathcal{D}_{m,q,s}^{(q_1,q_2)}$ in \eqref{T:ACweakSol:divfree:Vq}
& AC $L^q$-weak solution
& Theorem~\ref{T:ACweakSol:divfree} , Fig.~\ref{F:divfree}
\\
\hline \hline

\rule[-8pt]{0pt}{22pt}
 \textbf{Range of $m>1$ (PME) and $q\geq 1$} &\textbf{Structure of $V$} & \textbf{Solutions}  
 & \textbf{References}\\ \hline \hline

\rule[-8pt]{0pt}{22pt}
$m>1$ and $q\geq 1$
&  $V\in \mathcal{D}_{m,q}^{(q_1,q_2)}$ in \eqref{T:weakSol:divfree:V:PME}
& $L^q$-weak solution 
& Theorem~\ref{T:weakSol:divfree:PME}, Fig.~\ref{F:divfree:PME}
\\
\cline{1-4}
\rule[-8pt]{0pt}{22pt}
$m>1$ and $q\geq m$ 
& $V\in \mathcal{D}_{m,q,s}^{(q_1,q_2)}$ in \eqref{T:ACweakSol:divfree:V:PME}
& AC $L^q$-weak solution
& Theorem~\ref{T:ACweakSol:divfree:PME}, Fig.~\ref{F:divfree:PME}
\\
\cline{1-4}

\end{tabular}
\label{Table1}
}
\end{center}
\end{table}


We remark a recent work by Park and the same authors~\cite{HKKP}, which constructs non-negative weak solutions of PME and FDE with both drift and measure-type forcing terms. In that case, due to the presence of a measure-type force that does not belong to \( L^1(\Omega_T) \), the condition on \( V \) is even more restrictive than the \( L^1 \)-scaling invariant class, namely strict subscaling class, compared to the case without a forcing term.
In our setting, the drift term \( V \) is allowed to belong to the \( L^1 \)-scaling invariant class when considering \( L^1 \)-initial data. Furthermore, unlike the case with measure-type forcing term, we can also handle \( L^q \)-initial data for \( q > 1 \) when \( V \) belongs to a suitable class \( \mathsf{S}_{m,q}^{(q_1,q_2)} \); see Definition~\ref{D:Serrin}.

Our strategy for constructing weak solutions begins with establishing the existence of regular solutions under sufficiently smooth boundary data and \( V \), using the splitting method. Subsequently, we show that a sequence of regular solutions converges to a weak solution via compactness arguments. The a priori estimates presented in Proposition~\ref{P:energy} play a fundamental role throughout the paper. Additionally, the speed estimates provided in Lemma~\ref{L:speed} are crucial for constructing weak solutions which are absolutely continuous curves in the Wasserstein space. Finally, as an application, we study the existence of solutions for a viscous Boussinesq system of the fast diffusive type, demonstrating the broader applicability of our methods.
   
In Section~\ref{SS:Existence}, we present the main results concerning the existence of weak solutions, along with applications, remarks, and figures. The preliminaries and a priori estimates are outlined in Section~\ref{SS:Preliminaries} and Section~\ref{SS:priori}, respectively. The construction of weak solutions is detailed in Section~\ref{splitting method} and Section~\ref{Exist-weak}. Additionally, the appendices provide supplementary material: Appendix~\ref{A:Exist-regular} contains the detailed proof of the existence of regular solutions. Appendix~\ref{Appendix:PME} presents improved compactness arguments for the porous medium equation (PME) in the case $\nabla \cdot V =0$, derived under weaker conditions on $V$ than those in \cite{HKK02}.



\section{Main results}\label{S: Main}

In this section, we present the existence results for \eqref{FDE} under \eqref{FDE_bc_ic}, depending on the structure of \( V \) and the initial data. 
Before stating the main theorems, we introduce some notations for convenience.

\begin{itemize}
\item Let us denote by $\mathcal{P}(\Omega)$ the set of all Borel probability measures on $\Omega$ and by $\mathcal{P}_p(\Omega)$ the set of all Borel probability measures with finite $p$-th moment. For a bounded domain $\Omega$, we note that $\mathcal{P}(\Omega)= \mathcal{P}_p(\Omega)$ for any $p \geq 1$. We refer to Section~\ref{SS:Wasserstein} for definitions and related properties of the Wasserstein distance, denoted by \( W_p \), as well as the Wasserstein space and absolutely continuous (AC) curves.

\item The letters \( c \) and \( C \) are used to represent generic constants. Additionally, the letter \( \theta \) denotes a generic constant that varies depending on interpolation arguments.
    Throughout the paper, we omit explicit dependence on \( m \), \( q \), \( d \), \( T \), and 
    \( \|\rho_0\|_{L^{1}(\Omega)} \), as we assume that \( 0 < m < 1 \), \( q \geq 1 \), \( d \geq 2 \), and \( T > 0 \) are given constants, and that \( \|\rho_0\|_{L^{1}(\Omega)} = 1 \) in \( \mathcal{P}(\Omega) \).

\item Let us denote $(s)_{+} = \max\{ s, 0\}$.

\end{itemize}

We now introduce classes of $V$ that are suitable for constructing $L^q$-weak solutions.  
  
\begin{defn}\label{D:Serrin}
Let $0<m<1$ and $q\geq 1$. For $q_1, \ q_2, \ \tilde{q}_1, \ \tilde{q}_2 \in (1, \infty]$, we define following spaces.

\begin{itemize}
\item[(i)]Define
\[
\mathcal{L}_{x,t}^{q_1, q_2}:= \overline{\{V : V \in C_c^\infty(\overline{\Omega} \times[0, T)) ~~\mbox{and}~~ V\cdot \textbf{n} = 0 ~ \mbox{on} ~ \partial \Omega\}}^{L_{x,t}^{q_1, q_2}},
\]
where $\textbf{n}$ is the outward unit normal vector to the boundary of $\Omega$.
That is, the vector field $ V \in \mathcal{L}_{x,t}^{q_1, q_2}$ means that there exists a sequence $V_n  \in C_c^\infty(\overline{\Omega} \times[0, T)) $ such that
$ V_n \cdot \textbf{n} = 0$ on $\partial \Omega$ and
\[
\|V_n-V \|_{L_{x,t}^{q_1, q_2}} \rightarrow 0 \quad \mbox{as} ~~ n \rightarrow \infty.
\]

\item[(ii)] Let $1 - \frac{1}{d}  \leq m < 1$. The following classes of $V$ are defined as:
\begin{equation*}\label{Vclass:SS}
 \mathsf{S}_{m,q}^{(q_1, q_2)}:= \left\{ V: \ \|V\|_{\mathcal{L}^{q_1, q_2}_{x,t}} < \infty, \ \text{ where } \ \frac{d}{q_1} + \frac{2 + d(q+m-2)}{q_2} \leq 1 + d(m-1) \right\}.
\end{equation*}

\item[(iii)] Let $\left( 1 - \frac{2}{d} \right)_{+} \leq m < 1$. The following classes of $V$ are defined as:
\begin{equation*}\label{Vclass:tildeSS}
 \tilde{\mathsf{S}}_{m,q}^{(\tilde{q}_1, \tilde{q}_2)}:= \left\{ V: \ \|\nabla V\|_{\mathcal{L}^{\tilde{q}_1, \tilde{q}_2}_{x,t}} < \infty, \ \text{ where } \ \frac{d}{\tilde{q}_1} + \frac{2 + d(q+m-2)}{\tilde{q}_2} \leq 2 + d(m-1) \right\}.
\end{equation*}
\end{itemize}
Also, let us denote $\|V\|_{\mathsf{S}_{m,q}^{(q_1, q_2)}}$ and $\|V\|_{\tilde{\mathsf{S}}_{m,q}^{(\tilde{q}_1, \tilde{q}_2)}}$ as the norms corresponding to each spaces.
\end{defn}

\begin{remark}\label{R:Scaling}
The scaling invariant classes are introduced in \cite{HKK01, HKK02} by the same authors to construct weak solutions of PME $(m>1)$, that is, for $q_{m,d} = \frac{d(m-1)}{q}$,
    \begin{equation}\label{Scaling_invariant_class}
 \mathcal{S}_{m,q}^{(q_1, q_2)}:= \left\{ V: \ \|V\|_{\mathcal{L}^{q_1, q_2}_{x,t}} < \infty, \ \text{ where } \ \frac{d}{q_1} + \frac{2 + q_{m,d}}{q_2} \leq 1 + q_{m,d}\right\}.
\end{equation}  
    The class $\mathsf{S}_{m,q}^{(q_1,q_2)}$ and $\mathcal{S}_{m,q}^{(q_1,q_2)}$ coincides only when $q=1$. For $q>1$, $\mathsf{S}_{m,q}^{(q_1,q_2)}$ lies in sub-critical region of $\mathcal{S}_{m,q}^{(q_1,q_2)}$. For FDE, we are not able to construct $L^q$-weak solutions under $V$ belonging to $\mathcal{S}_{m,q}^{(q_1, q_2)}$ but $\mathsf{S}_{m,q}^{(q_1, q_2)}$ because of technical issues getting a priori estimates, Proposition~\ref{P:energy}.
\end{remark}	

\begin{figure}
\centering

\begin{tikzpicture}[domain=0:16]

\fill[fill= lgray] (0,0) -- (2.6,0) -- (2.6,2.6) -- (0, 2.6) ;

\draw[->] (0,0) -- (6,0) node[right] {\scriptsize $\frac{1}{q_1}$};
\draw[->] (0,0) -- (0,7) node[left] { \scriptsize $\frac{1}{q_2}$};

\draw (0,0)  node[left] {\scriptsize $O$};

\draw (0,2.6) node{\scriptsize $+$} node[left] {\scriptsize $1$} 
-- (1.7,0) node{\scriptsize $+$} node[below] {\scriptsize $\frac{2+d(m-1)}{d}$};
\draw (1.3,0.5) node[right]{\scriptsize $\mathcal{D}_{m,1}^{(q_1,q_2)}$};

\draw[thick] (0,5) node{\scriptsize $+$} node[left] {\scriptsize $\frac{d+2}{2}- \frac{d[2+d(m-1)]}{2[2q+d(m-1)]}$} -- (3.2,0) node{\scriptsize $+$} node[below] {\scriptsize $\frac{d+2}{d} + \frac{m-2}{q}$};
\draw (2.8,0.5) node[right]{\scriptsize $\mathcal{D}_{m,q}^{(q_1,q_2)}$};
\draw (2.6, 0.95) node{\scriptsize $\bullet$} node[right]{\scriptsize $A$} ;
\draw (1.55, 2.6) node{\scriptsize $\bullet$} node[above]{\scriptsize $B$} ;

\draw[gray] (0,6.5) node{\scriptsize $+$} node[left] {\scriptsize $\frac{d+2}{2}$} 
-- (4.35,0) node{\scriptsize $+$} node[below] {\scriptsize $\frac{d+2}{d}$};
\draw[gray] (4,0.5) node[right]{\scriptsize $\mathcal{D}_{m,\infty}^{(q_1,q_2)}$};

\draw (0,1.3) node{\scriptsize $+$} node[left] {\scriptsize $\frac{1}{2}$} 
-- (0.87,0) node{\scriptsize $+$} node[below] {\scriptsize $\frac{1}{d}$};
\draw (0.1,1) node[right] {\scriptsize $\mathcal{S}_{m,q}^{(q_1,q_2)}$};

\draw[thick] (0,0.9) node{\scriptsize $+$} node[left]{\scriptsize $D$} -- (0.5,0) node{\scriptsize $+$}node[below]{\scriptsize $C$};
\draw (0.25,0.3) node[right] {\scriptsize $\mathsf{S}_{m,q}^{(q_1,q_2)}$};

\draw[thick, gray] (2.6, 2.6) circle(0.05);
\draw[very thin] (0,2.6) -- (3.3,2.6) node[right] {\scriptsize $q_2 = 1$};
\draw[very thin] (2.6,0) -- (2.6,3.3) node[above] {\scriptsize $q_1 = 1$};

\draw[thick, ->] (1.7,2) -- (1,1.5);
\draw[thick] (0.9,1.3) node[right]{\scriptsize As $q \to 1$};

\draw[thick, ->] (2,2.2) -- (2.4,2.5);
\draw[thick] (1.9,2) node[right]{\scriptsize As $q \to \infty$};

\draw(6, 5.5) node[right]{\scriptsize $\bullet$ $\mathcal{D}_{m,q}^{(q_1, q_2)}$: $\frac{d}{q_1} + \frac{2 +q_{m,d}}{q_2} \leq  2+ \frac{d(q+m-2)}{q}$};
\draw(6,5) node[right]{\scriptsize $\bullet$ the shaded region: valid range of $q_1, q_2 \geq 1$ satisfying $\mathcal{D}_{m,q}^{(q_1, q_2)}$} ; 
\draw(6, 4.5) node[right]{\scriptsize $\bullet$ $A=(1, 1-\frac{d}{2q+d(m-1)})$, $B=(1-\frac{1}{q}, 1)$}; 
\draw(6, 3.5) node[right]{\scriptsize $\bullet$ $\mathcal{D}_{m,1}^{(q_1, q_2)}$: $\frac{d}{q_1} + \frac{2 + d(m-1)}{q_2} \leq 2 + d(m-1)$};
\draw(6, 3) node[right]{\scriptsize $\bullet$ $\mathcal{D}_{m,\infty}^{(q_1, q_2)}$: $\frac{d}{q_1} + \frac{2}{q_2} \leq 2+d$}; 
\draw(6, 2.5) node[right]{\scriptsize $\bullet$ $\mathcal{S}$ is same as $\mathcal{S}_{m,\infty}^{(q_1,q_2)}$ or $\mathcal{S}_{1,q}^{(q_1,q_2)}$: $\frac{d}{q_1} + \frac{2}{q_2} \leq 1$};
\draw(6, 2) node[right]{\scriptsize $\bullet$ $\overline{CD} = \mathsf{S}$: $\frac{d}{q_1} + \frac{2+d(q+m-2)}{q_2} = 1+d(m-1)$};
\draw(6, 1.5) node[right]{\scriptsize $\bullet$ $C=(\frac{1+d(m-1)}{d}, 0)$, $D=(0,\frac{1+d(m-1)}{2+d(q+m-2)})$};
\end{tikzpicture}
\caption{\footnotesize The class $\mathcal{D}_{m,q}^{(q_1,q_2)}$.}
\label{F:D}
\end{figure}

Furthermore, we introduce two new classes for \( V \) in the divergence-free case. The first class, \( \mathcal{D}_{m,q}^{(q_1,q_2)} \), is derived from the compactness argument in Proposition~\ref{Proposition : compactness}, while the second class, \( \mathcal{D}_{m,q,s}^{(q_1,q_2)} \), is established through the speed estimate in Lemma~\ref{L:speed}.

\begin{defn}\label{D:Serrin_divfree} 
Let $m>0$ and $q\geq 1$. For $q_1, q_2 \in (1, \infty]$, we define following spaces. 
\begin{itemize}

\item[(i)] Define 
\[
\mathfrak{L}_{x,t}^{q_1, q_2}:= \overline{\{V : V \in C_c^\infty(\overline{\Omega} \times[0, T))~~\mbox{and}~~ \nabla\cdot V = 0 ~~\mbox{and}~~ V\cdot \textbf{n} = 0 ~ \mbox{on} ~ \partial \Omega\}}^{L_{x,t}^{q_1, q_2}},
\]
where $\textbf{n}$ is the outward unit normal vector to the boundary of $\Omega$.
That is, the vector field $ V \in \mathfrak{L}_{x,t}^{q_1, q_2}$ means that there exists a sequence $V_n  \in C_c^\infty(\overline{\Omega} \times[0, T)) $ such that
$ V_n\cdot \textbf{n} = 0$ on $\partial \Omega$, $\nabla \cdot V_n =0$, and
\[
\|V_n-V \|_{L_{x,t}^{q_1, q_2}} \rightarrow 0 \quad \mbox{as} ~~ n \rightarrow \infty.
\]

\item[(ii)] Let $(1-\frac{2q}{d})_{+} \leq m < 1$ and $q_{m,d} := \frac{d(m-1)}{q}$. We define the classes of $V$ that 
\begin{equation*}\label{Vclass:compact}
	 \mathcal{D}_{m,q}^{(q_1,q_2)}:= \left\{V: \ \nabla \cdot V = 0 \ \text{ and } \ \|V\|_{\mathfrak{L}_{x,t}^{q_1,q_2}} < \infty, \ \text{ where }\ \frac{d}{q_1} + \frac{2 + q_{m,d}}{q_2} \leq 2+ \frac{d(q+m-2)}{q}\right\},
\end{equation*}
and
\begin{equation*}\label{Vsubclass:compact+}
	 \mathcal{D}_{m,q,+}^{(q_1,q_2)}:= \left\{V: \ \nabla \cdot V = 0 \ \text{ and } \ \|V\|_{\mathfrak{L}_{x,t}^{q_1,q_2}} < \infty, \ \text{ where }\ \frac{d}{q_1} + \frac{2 + q_{m,d}}{q_2} <  2+ \frac{d(q+m-2)}{q}\right\}.
\end{equation*}

\item[(iii)] Let $(1-\frac{2q}{d})_{+} \leq m < 1$ and $q_{m,d} := \frac{d(m-1)}{q}$. We define the classes of $V$ that
\begin{equation*}\label{Vclass:Speed}
	 \mathcal{D}_{m,q,s}^{(q_1,q_2)}:= \left\{V: \ \nabla \cdot V = 0 \ \text{ and } \ \|V\|_{\mathfrak{L}_{x,t}^{q_1,q_2}} < \infty, \ \text{ where }\ \frac{d}{q_1} + \frac{2 + q_{m,d}}{q_2} \leq  1+ \frac{d(q+m-2)}{q}\right\}.
\end{equation*}
\end{itemize}

Also, let us denote  $\|V\|_{\mathcal{D}_{m,q}^{(q_1, q_2)}}$, $\|V\|_{\mathcal{D}_{m,q,+}^{(q_1, q_2)}}$, and $\|V\|_{\mathcal{D}_{m,q, s}^{(q_1, q_2)}}$ as the norms corresponding to each spaces.
\end{defn}

\begin{remark}
When $0<m<1$, both classes $\mathcal{D}_{m,q}^{(q_1,q_2)}$ and $\mathcal{D}_{m,q,s}^{(q_1,q_2)}$ belong to the super-critical regime of $\mathcal{S}_{m,q}^{(q_1,q_2)}$ and $\mathsf{S}_{m,q}^{(q_1,q_2)}$. These classes are important only in case $\nabla \cdot V = 0$ for both FDE and PME. In Fig.~\ref{F:D}, we illustrate the valid range of $(q_1,q_2)$ of $\mathcal{D}_{m,q}^{(q_1,q_2)}$ and see Appendix~\ref{Appendix:PME} for extended existence results for PME in case $\nabla \cdot V = 0$.  
\end{remark}

\subsection{Existence results}\label{SS:Existence}

Here we introduce the notion of weak solutions of \eqref{FDE} under \eqref{FDE_bc_ic}.

\begin{defn}\label{D:weak-sol}
Let $q\geq 1$ and $V$ be a measurable vector field.
We say that a nonnegative measurable function $\rho$ is
a nonnegative \textbf{$L^q$-weak solution} of \eqref{FDE} under \eqref{FDE_bc_ic} with a nonnegative initial data $\rho_0 \in L^{q}(\Omega)$ if the followings are satisfied:
\begin{itemize}
\item[(i)] It holds that
\[
\rho \in L^{\infty}\left(0, T; L^q (\Omega)\right),
\ \nabla \rho^m \in L^{1}(\Omega_T), 
\ \nabla \rho^{\frac{q+m-1}{2}} \in L^{2}(\Omega_T), 
\ \text{ and } \ \rho V \in L^{1}(\Omega_T).
\]
\item[(ii)] For any function $\varphi \in {\mathcal{C}}^\infty_c (\overline{\Omega} \times [0,T))$, it holds that
\begin{equation*}\label{KK-May7-40}
\iint_{\Omega_T} \left\{ \rho \varphi_t - \nabla \rho^m \cdot \nabla \varphi + \rho V \cdot \nabla \varphi \right\} \,dx dt = -\int_{\Omega} \rho_{0} (\cdot) \varphi(\cdot, 0) \,dx.
\end{equation*}
\end{itemize}
\end{defn}

This section is divided into three parts based on the structure of \( V \), which determines the existence results. Specifically, we consider the following cases: \( V \in \mathsf{S}_{m,q}^{(q_1, q_2)} \), where \( V \) satisfies certain integrability conditions; \( V \in \tilde{\mathsf{S}}_{m,q}^{(q_1, q_2)} \), where \( \nabla V \) satisfies certain integrability conditions; and the divergence-free condition \( \nabla \cdot V = 0 \).

\subsubsection{Existence for case: $V \in \mathsf{S}_{m,q}^{(q_1,q_2)}$}\label{SS:energy-sol}

We provide the absolutely continuous $L^q$-weak solutions. 

\begin{theorem}\label{T:ACweakSol}
Let $d\geq 2$ and $ 1-\frac{1}{d} < m <1$. Let $q\geq 1$ and suppose that 
\begin{equation}\label{T:ACweakSol:V_q}
V \in \mathsf{S}_{m,q}^{(q_1,q_2)}  \ \text{ for }  \ 
    0 \leq \frac{1}{q_1} < \frac{1+d(m-1)}{d}, \quad  0\leq \frac{1}{q_2} \leq \frac{1+d(m-1)}{2+d(q+m-2)}.
\end{equation}

\begin{itemize}
	\item[(i)] For $q=1$, assume that $\rho_0 \in  \mathcal{P}_2(\Omega)$ and $\int_{\Omega} \rho_0 \log \rho_0 \,dx < \infty$.  
Then, there exists a nonnegative $L^1$-weak solution of \eqref{FDE}-\eqref{FDE_bc_ic} in Definition~\ref{D:weak-sol} such that  $\rho \in AC(0,T; \mathcal{P}_2 (\Omega))$ with $\rho(\cdot, 0)=\rho_0$.
Furthermore, $\rho$ satisfies
\begin{equation}\label{T:ACweakSol:E_1}
\sup_{0\leq t \leq T} \int_{\Omega} \rho \abs{\log \rho} (\cdot, t)  \,dx
+ \iint_{\Omega_T} \{ \abs{\nabla \rho^{\frac m2}}^2 + \left(\abs{\frac{\nabla
\rho^m}{\rho}}^{2}+|V|^{2}\right ) \rho  \}\,dx\,dt
 \leq  C,
\end{equation}
and
\begin{equation}\label{T:ACweakSol:W_1} 
 W_{2}(\rho(t),\rho(s))\leq C (t-s)^{\frac{1}{2}},\qquad
\forall ~~0\leq s\leq t\leq T,
\end{equation}
where the constant $C= C ( \|V\|_{\mathsf{S}_{m,1}^{(q_1,q_2)}},\,
\int_{\Omega} \rho_0 \log \rho_0 \,dx )$.

\smallskip

\item[(ii)] For $q>1$, assume that $\rho_0 \in  \mathcal{P}_2(\Omega) \cap L^q (\Omega)$.
Then, there exists a nonnegative $L^q$-weak solution of \eqref{FDE}-\eqref{FDE_bc_ic} in Definition~\ref{D:weak-sol} such that  $\rho \in AC(0,T; \mathcal{P}_2 (\Omega))$ with $\rho(\cdot, 0)=\rho_0$.
Furthermore, $\rho$ satisfies
 \begin{equation}\label{T:ACweakSol:E_q}
 \sup_{0\leq t \leq T} \int_{\Omega} \rho^q(\cdot, t) \,dx
 + \iint_{\Omega_T} \{ \abs{\nabla \rho^{\frac{q+m-1}{2}}}^2  + \left(\abs{\frac{\nabla
\rho^m}{\rho}}^{2}+|V|^{2}\right ) \rho  \}\,dx\,dt
\leq C,
\end{equation}
and \eqref{T:ACweakSol:W_1} where the constant
with $C = C ( \|V\|_{\mathsf{S}_{m,q}^{(q_1, q_2)}}, \| \rho_0 \|_{L^{q}(\Omega)} )$.

\end{itemize}
\end{theorem}

\begin{figure}
\centering

\begin{tikzpicture}[domain=0:16]

\fill[fill= lgray] (0,0) -- (0,3.8) -- (3,0);
\fill[fill= gray] (0,0) -- (0,2) -- (3,0);

\draw[->] (0,0) -- (5,0) node[right] {\scriptsize $\frac{1}{q_1}$};
\draw[->] (0,0) -- (0,5) node[left] { \scriptsize $\frac{1}{q_2}$};

\draw (0,0)  node[left] {\scriptsize $O$};

\draw[gray] (0,4.5) node{\scriptsize $+$} node[left] {\scriptsize $\frac{1}{2}$} 
-- (4, 0) node {\scriptsize $+$} node[below] {\scriptsize $\frac{1}{d}$} ;
\draw[gray] (0.5, 4) node[right] {\scriptsize $\frac{d}{q_1}+\frac{2}{q_2}=1$} ;

\draw[thick] (3, 0) circle(0.05) node[below]{\scriptsize $B$};
\draw (0,3.8) node {\scriptsize $\bullet$} node[left] {\scriptsize $A$}  -- (3, 0)  ;
\draw (0.5, 2) node[right]{\scriptsize $\mathsf{S}_{m,1}^{(q_1, q_2)}$};

\draw (0,2) node {\scriptsize $\bullet$} node[left] {\scriptsize $C$}  -- (3, 0)  ;
\draw (0.5, 0.7) node[right]{\scriptsize $\mathsf{S}_{m,q}^{(q_1, q_2)}$};

\draw(6,4) node[right]{\scriptsize $\bullet$ $\mathcal{R}(OAB)$: $(q_1,q_2)$ satisfying  \eqref{T:ACweakSol:V_q} for $q=1$.};
\draw(6,3.5) node[right]{\scriptsize  $\bullet$ $\mathcal{R}(OCB)$: $(q_1,q_2)$ satisfying  \eqref{T:ACweakSol:V_q} for $q>1$.};
\draw(6,3) node[right]{\scriptsize  $\bullet$ The end point $B$ where $q_2 = \infty$ is not included.};

\draw(6, 2) node[right]{\scriptsize $\bullet$ $\overline{AB}$: $\frac{d}{q_1} + \frac{2 + d(m-1)}{q_2} = 1+ d(m-1)$.}; 
\draw(6, 1.5) node[right]{\scriptsize $\bullet$ $A = (0, \frac{1 + d(m-1)}{2 + d(m-1)})$, \ $B=(\frac{1 + d(m-1)}{d},0)$}; 
\draw(6, 1) node[right]{\scriptsize $\bullet$ $\overline{CB}$: $\frac{d}{q_1} + \frac{2+d(q+m-2)}{q_2} = 1+ d(m-1)$.}; 
\draw(6, 0.5) node[right]{\scriptsize $\bullet$ $C= (0, \frac{1+d(m-1)}{2+d(q+m-2)})$ which approach to $A$ as $q\to 1$ and to $O$ as $q\to \infty$ };

\end{tikzpicture}
\caption{\footnotesize AC $L^q$-weak solutions of FDE in case $V \in \mathsf{S}_{m,q}^{(q_1,q_2)}$}
\label{F:S}
\end{figure}

\begin{remark}\label{R:1}
	In Fig.~\ref{F:S}, we plot the range of $(q_1,q_2)$ satisfying \eqref{T:ACweakSol:V_q}. When $q=1$, the region $\mathcal{R}(OAB)$ except the point $B$ is the valid region. When $q>1$, the darkly-shaded region $\mathcal{R}(OCB)$ except the point $B$ is the valid region. The point $B$ when $q_2 = \infty$ is not covered because the absorbing technique does not hold to get an energy estimate in Proposition~\ref{P:energy} in general except the case when the corresponding norm of $V$ is small enough. 
\end{remark}

\subsubsection{Existence for case: $V \in \tilde{\mathsf{S}}_{m,q}^{(q_1,q_2)}$ }\label{SS:tilde Serrin}

Now, we construct weak solutions assuming suitable class for $\nabla V$. 

\begin{theorem}\label{T:ACweakSol_tilde}
Let $d\geq 2$ and $ 1-\frac{1}{d} < m <1$. Let $q\geq 1$ and suppose that 
 \begin{equation}\label{T:ACweakSol_tilde:V_q}
V\in \tilde{\mathsf{S}}_{m,q}^{(\tilde{q}_1, \tilde{q}_2)}\cap L_{x,t}^{1,\tilde{q}_2}    \ \text{ for }  \ 
\frac{1}{d} < \frac{1}{\tilde{q}_1} < \frac{2+d(m-1)}{d}, \quad  0\leq \frac{1}{\tilde{q}_2} < \frac{1+d(m-1)}{2+d(q+m-2)}.
\end{equation}

\begin{itemize}
	\item[(i)]  For $q=1$, suppose that $\rho_0 \in  \mathcal{P}_2(\Omega)$, and $\int_{\Omega} \rho_0 \log \rho_0 \,dx < \infty$. 
Then, there exists a nonnegative $L^1$-weak solution of \eqref{FDE}-\eqref{FDE_bc_ic} in Definition~\ref{D:weak-sol} such that  $\rho \in AC(0,T; \mathcal{P}_2 (\Omega))$ with $\rho(\cdot, 0)=\rho_0$.
Furthermore, $\rho$ satisfies 
\eqref{T:ACweakSol:E_1} and \eqref{T:ACweakSol:W_1} with $C=C
 (\left\| V\right\|_{\tilde{\mathsf{S}}_{m,1}^{(\tilde{q}_1,
\tilde{q}_2)}}, \int_{\Omega} \rho_0 \log \rho_0 \,dx )$.
\smallskip 

   \item[(ii)] For $q>1$, suppose that $\rho_0 \in  \mathcal{P}_2(\Omega) \cap L^{q}(\Omega)$. 
Then, there exists a nonnegative $L^q$-weak solution of \eqref{FDE}-\eqref{FDE_bc_ic} in Definition~\ref{D:weak-sol} such that  $\rho \in AC(0,T; \mathcal{P}_2 (\Omega))$ with $\rho(\cdot, 0)=\rho_0$.
Furthermore, $\rho$ satisfies \eqref{T:ACweakSol:E_q} and \eqref{T:ACweakSol:W_1} 
$C = C (\left\| V\right\|_{\tilde{\mathsf{S}}_{m,q}^{(\tilde{q}_1, \tilde{q}_2)}}, \|\rho_0\|_{L^{q}(\Omega)} )$.
\end{itemize}
\end{theorem}

\begin{figure}
\centering

\begin{tikzpicture}[domain=0:16]

\fill[fill= llgray] (2.3,0) -- (5,0) -- (5, 3) -- (2.3,3) ;
\fill[fill= lgray] (2.3, 1.9) -- (2.3, 0) -- (4,0);
\fill[fill= gray] (2.3, 1.27) -- (2.3, 0) -- (4,0);

\draw[->] (0,0) -- (5.5,0) node[right] {\scriptsize $\frac{1}{\tilde{q}_1}$};
\draw[->] (0,0) -- (0,5.5) node[left] { \scriptsize $\frac{1}{\tilde{q}_2}$};

\draw (0,0)  node[left] {\scriptsize $O$};

\draw[gray] (0,4.5) node{\scriptsize $+$} node[left] {\scriptsize $1$} 
-- (4.8, 0) node {\scriptsize $+$} node[below] {\scriptsize $\frac{2}{d}$} ;
\draw[gray] (4.3, 0.5) node[right]{\scriptsize $\tilde{\mathcal{S}}$} ;
\draw[gray, thick](2.3, 2.35) circle(0.05);
\draw[dotted] (0,2.35) node{\scriptsize $*$} node[left]{\scriptsize $\frac12$} -- (2.3, 2.35);

\draw[thick] (4,0) circle(0.05) node[below]{\scriptsize $B$};
\draw (0,4.5)node{\scriptsize $+$} node[left] {\scriptsize $1$}   -- (4, 0) ; 

\draw[thick] (2.3, 1.9) circle(0.05) node[left]{\scriptsize $A$};
\draw[dotted] (2.3,1.9)-- (0,1.9) node {\scriptsize $*$} node[left]{\scriptsize $\frac{1 + d(m-1)}{2 + d(m-1)}$} ;

\draw (3.3, 0.7) node[right]{\scriptsize $\tilde{\mathsf{S}}_{m,1}^{(\tilde{q}_1, \tilde{q}_2)}$};

\draw (0,3) node{\scriptsize $+$} node[left]{\scriptsize $\frac{2+d(m-1)}{2+d(q+m-2)}$} -- (4,0);
\draw[thick] (2.3, 1.27) circle(0.05) node[left]{\scriptsize $C$};
\draw[dotted] (2.3,1.27)-- (0,1.27) node {\scriptsize $*$} node[left]{\scriptsize $\frac{1 + d(m-1)}{2 + d(q+m-2)}$} ;

\draw (2.3, 0.7) node[right]{\scriptsize $\tilde{\mathsf{S}}_{m,q}^{(\tilde{q}_1, \tilde{q}_2)}$};

\draw[dotted] (2.3,0) -- (2.3,3.3) node[above] {\scriptsize $\tilde{q}_1 = d$};
\draw[dotted] (5,0) -- (5,3.3) node[above] {\scriptsize $\tilde{q}_1 = 1$};
\draw[thick] (2.3, 0) circle(0.05) node[below]{\scriptsize $\tilde{O}$};


\draw(6,5) node[right]{\scriptsize $\bullet$ the lightly-shaded region: embedding region holding \eqref{V_embedding}.};
\draw(6,4.5) node[right]{\scriptsize $\bullet$ $\mathcal{R}(\tilde{O}AB)$: $(\tilde{q}_1, \tilde{q}_2)$ satisfying  \eqref{T:ACweakSol_tilde:V_q} for $q=1$.}; 
\draw(6,4) node[right]{\scriptsize $\bullet$ $\mathcal{R}(\tilde{O}CB)$: $(\tilde{q}_1, \tilde{q}_2)$ satisfying  \eqref{T:ACweakSol_tilde:V_q} for $q>1 $.};

\draw(6,3) node[right]{\scriptsize $\bullet$ Let $d> 2$ and $1-\frac{1}{d} < m < 1$. };
\draw(6,2.5) node[right]{\scriptsize $\bullet$ $\tilde{\mathsf{S}}_{m,1}^{(\tilde{q}_1,\tilde{q}_2)}$: $\frac{d}{\tilde{q}_1} + \frac{2 + d(m-1)}{\tilde{q}_2} \leq 2 + d(m-1)$.} ;
\draw(6,2) node[right]{\scriptsize $\bullet$ $A=(\frac{1}{d}, \frac{1+d(m-1)}{2+d(m-1)})$, $B=(\frac{2+d(m-1)}{d},0)$, $\tilde{O}=(\frac{1}{d},0)$} ;
\draw(6,1.5) node[right]{\scriptsize $\bullet$ $\tilde{\mathsf{S}}_{m,q}^{(\tilde{q}_1,\tilde{q}_2)}$: $\frac{d}{\tilde{q}_1} + \frac{2 + d(q+m-2)}{\tilde{q}_2} \leq 2 + d(m-1)$.} ;
\draw(6,1) node[right]{\scriptsize $\bullet$ $C=(\frac{1}{d}, \frac{1+d(m-1)}{2+d(q+m-2)})$ which approach to $A$ as $q\to 1$ and to $\tilde{O}$ as $q\to \infty$.} ;

\end{tikzpicture}
\caption{\footnotesize  AC $L^q$-weak solutions of FDE in case $V\in \tilde{\mathsf{S}}_{m,q}^{(\tilde{q}_1, \tilde{q}_2)}$}
\label{F:tilde_S:weak}
\end{figure}

We provide a couple of remarks. 
\begin{remark}
	\begin{itemize}
		\item [(i)] The condition $V \in \tilde{\mathsf{S}}_{m,q}^{(\tilde{q}_1, \tilde{q}_2)} \cap L_{x,t}^{1, \tilde{q}_2}$ implies that $V\in \mathsf{S}_{m,q}^{(\tilde{q}_{1}^{\ast}, \tilde{q}_2)}$ because of the following embedding relation between $V$ and $\nabla V$ such that 
			\begin{equation}\label{V_embedding}
			\|V\|_{L_{x,t}^{\tilde{q}_{1}^{\ast},\tilde{q}_2}} \leq c \left(\|\nabla V\|_{L_{x,t}^{\tilde{q}_1,\tilde{q}_2}} + \|V\|_{L_{x,t}^{1, \tilde{q}_2}}\right), \quad \text{ for } \ \tilde{q}_{1}^{\ast} = \frac{d \tilde{q}_1}{ d - \tilde{q}_1} \in (1,d).
			\end{equation}
			The embedding possible region is shaded lightly in Fig.~\ref{F:tilde_S:weak}. 
			Although the class $\tilde{\mathsf{S}}_{m,q}^{(\tilde{q}_1, \tilde{q}_2)}$ allows $(1-\frac{2}{d})_{+} \leq m<1$, the suitable ranges for embedding relation \eqref{V_embedding} allows only $ 1-\frac{1}{d} < m<1$.
			
		\item[(ii)] When $q=1$, the class $\tilde{\mathsf{S}}_{m,1}^{(\tilde{q}_1, \tilde{q}_2)}$ coincides with the scaling invariant class for $\nabla V$ in \cite{HKK01,HKK02}. In Fig.~\ref{F:tilde_S:weak}, we plot admissible region of \eqref{T:ACweakSol_tilde:V_q}. The region $\mathcal{R}(\tilde{O}AB)$ except the line $\overline{\tilde{O}A}$ and the point $B$ is where we can construct AC $L^1$-weak solutions. Moreover, the region $\mathcal{R}(\tilde{O}CB)$ except the line $\overline{\tilde{O} C}$ and the point $B$ is where we can construct AC $L^q$-weak solutions.  
	\end{itemize}
\end{remark}

\subsubsection{Existence for case: $\nabla \cdot V =0$}

In the divergence-free case, we construct weak solutions under weaker restrictions on both $V$ and $m$. The suitable condition on $V$ is determined by the compactness arguments in Proposition~\ref{Proposition : compactness} or the speed estimates in Proposition~\ref{L:speed} because the energy estimates are obtained independently of $V$. Concerning the same arguments for PME in \cite{HKK02}, here we improve them to address certain super-critical regimes of $\mathsf{S}_{m,q}^{(q_1,q_2)}$, which are introduced in Appendix~\ref{Appendix:PME}.
Also, now we construct weak solutions under the range of $m$ such that 
\begin{equation*}\label{m:divfree}
	\begin{cases}
		1-\frac{2q}{d} \leq m < 1, & \text{ for } 1\leq q < \frac{d}{2} \text{ and } d>2,  \vspace{1 mm}\\
		0<m<1, & \text{ for } q\geq 1 \text{ and } d=2, \text{ or, } q\geq \frac{d}{2} \text{ and } d>2. 
	\end{cases}
\end{equation*}

First, we present the \( L^q \)-weak solutions that satisfy an energy estimate. If the critical conditions on \( V \) are not imposed, then we can estimate the \( \delta \)-distance.

\begin{theorem}\label{T:weakSol:divfree}
Let $d\geq 2$ and $q\geq 1$.Also, let $0<m<1$ and $ (1-\frac{2q}{d})_{+} \leq m < 1$. Suppose that 
\begin{equation}\label{T:weakSol:divfree:Vq}
  V\in \mathcal{D}_{m,q}^{(q_1, q_2)} \ \text{ for } \
\begin{cases}
0 \leq \frac{1}{q_1} \leq \frac{q-1}{q} + \frac{2+q_{m,d}}{d}, \ 0 \leq \frac{1}{q_2} \leq 1, &\text{if }  d\geq 2 \text{ and } 1 \leq q < 2-m \vspace{1mm}\\
	0 \leq \frac{1}{q_1} \leq \frac{q-1}{q} + \frac{2+q_{m,d}}{d(m+q-1)}, \ 0 \leq \frac{1}{q_2} \leq 1, &\text{if } d>2 \text{ and } q\geq 2-m ,  \vspace{1 mm} \\
	 0 \leq \frac{1}{q_1} < \frac{q-1}{q} + \frac{2+q_{m,d}}{d(m+q-1)}, \ 0 \leq \frac{1}{q_2} \leq 1, &\text{if }  d=2 \text{ and } q \geq 2-m.
\end{cases}
 \end{equation}
\begin{itemize}
\item[(i)] For $q=1$, assume that $\rho_0 \in  \mathcal{P}_2(\Omega)$ and $\int_{\Omega} \rho_0 \log \rho_0 \,dx < \infty$.
Then, there exists a nonnegative $L^1$-weak solution of \eqref{FDE}-\eqref{FDE_bc_ic} in Definition~\ref{D:weak-sol}
that holds 
 \begin{equation}\label{T:weakSol:E_1}
 \esssup_{0\leq t \leq T} \int_{\Omega} \rho \abs{\log \rho  }(\cdot, t) \,dx
 + \iint_{\Omega_T} \abs{\nabla \rho^{\frac{m}{2}}}^2 \,dx\,dt
\leq C,
\end{equation}
with $C = C ( \int_{\Omega} \rho_0 \log \rho_0 \,dx )$.

\item[(ii)] For $q>1$, assume that $\rho_0 \in  \mathcal{P}_2(\Omega) \cap L^{q}(\Omega)$.
Then, there exists a nonnegative $L^q$-weak solution of \eqref{FDE}-\eqref{FDE_bc_ic} in Definition~\ref{D:weak-sol} that holds
 \begin{equation}\label{T:weakSol:E_q}
 \esssup_{0\leq t \leq T} \int_{\Omega} \rho^q(\cdot, t) \,dx
 + \iint_{\Omega_T} \abs{\nabla \rho^{\frac{q+m-1}{2}}}^2 \,dx\,dt
\leq C,
\end{equation}
 with $C = C (\|\rho_{0}\|_{L^{q} (\Omega)} )$.
 
\item[(iii)] Furthermore, suppose that 
\begin{equation}\label{T:V:divfree_delta}
	V\in \mathcal{D}_{m,q,+}^{(q_1,q_2)}  
	\ \text{ for } \  0 \leq \frac{1}{q_1} < \frac{q-1}{q} + \frac{2+q_{m,d}}{d}, \quad  0 \leq \frac{1}{q_2} < 1.
\end{equation}
Then, we have 
\begin{equation*}\label{T:weakSol:Narrow_2}
\delta(\rho(t),\rho(s))\leq C(t-s)^{a},   \quad  \forall ~ 0\leq s<t\leq T 
\end{equation*}
with $a= \min \left\{ \frac{1}{2}, \frac{2+\frac{d(q+m-2)}{q} - \left(\frac{d}{q_1} + \frac{2+q_{m,d}}{q_2}\right)}{2+q_{m,d}} \right\}$ 
and $C= \begin{cases}
	C ( \|V\|_{\mathcal{D}_{m,1,+}^{(q_1,q_2)}}, \int_{\Omega} \rho_0 \log \rho_0 \,dx ), & \text{if } q=1 , \\
	C (\|V\|_{\mathcal{D}_{m,q,+}^{(q_1,q_2)}}, \|\rho_{0}\|_{L^{q} (\Omega)} ), &\text{if } q>1.
\end{cases}$
Here, the definition of  $\delta$ is given in \eqref{Narrow-distance}.
\end{itemize}
\end{theorem}

Now, using the speed estimates in Lemma~\ref{L:speed}, we construct absolutely continuous \( L^q \)-weak solutions.

\begin{theorem}\label{T:ACweakSol:divfree}
Let $d\geq 2$ and $q\geq 1$.Also, let $0<m<1$ and $ (1-\frac{2q}{d})_{+} \leq m < 1$. Suppose that 
\begin{equation}\label{T:ACweakSol:divfree:Vq}
\begin{aligned}
		 V \in \mathcal{D}_{m,q,s}^{(q_1,q_2)} 
		\  \text{ for } \ 
	\begin{cases}
0 \leq \frac{1}{q_1} \leq \frac{q-1}{2q}+\frac{2+q_{m,d}}{2d},\ 0 \leq \frac{1}{q_2} \leq \frac{1}{2}, &\text{if }  d\geq 2  \text{ and }  1 \leq q<2-m,
\vspace{1mm}\\
	0 \leq \frac{1}{q_1} \leq \frac{q-1}{2q}+\frac{2+q_{m,d}}{2d(m+q-1)}, \ 0 \leq \frac{1}{q_2} \leq \frac{1}{2}, &\text{if } d>2 \text{ and } q \geq 2-m,  \vspace{1 mm} \\
	 0 \leq \frac{1}{q_1} < \frac{q-1}{2q}+\frac{2+q_{m,d}}{2d(m+q-1)}, \ 0 \leq \frac{1}{q_2} \leq \frac{1}{2}, &\text{if }  d=2 \text{ and } q \geq 2-m.  
\end{cases}
\end{aligned}
\end{equation}

\begin{itemize}
\item[(i)] For $q=1$ and $ (1-\frac{2}{d})_{+} \leq m < 1$, assume that $\rho_0 \in  \mathcal{P}_2(\Omega)$ and $\int_{\Omega} \rho_0 \log \rho_0 \,dx < \infty$. 
Then, there exists a nonnegative $L^1$-weak solution of \eqref{FDE}-\eqref{FDE_bc_ic} in Definition~\ref{D:weak-sol} such that  $\rho \in AC(0,T; \mathcal{P}_2 (\Omega))$ with $\rho(\cdot, 0)=\rho_0$.
Furthermore, $\rho$ satisfies \eqref{T:ACweakSol:E_1} and \eqref{T:ACweakSol:W_1}
with $C = C (\|V\|_{\mathcal{D}_{m,1,s}^{(q_1,q_2)}}, \int_{\Omega} \rho_0 \log \rho_0 \,dx )$.
\smallskip 

\item[(ii)]
For $q>1$ and $ (1-\frac{2q}{d})_{+} \leq m < 1$, assume that $\rho_0 \in  \mathcal{P}_2(\Omega) \cap L^{q}(\Omega)$. 
Then, there exists a nonnegative $L^q$-weak solution of \eqref{FDE}-\eqref{FDE_bc_ic} in Definition~\ref{D:weak-sol} such that  $\rho \in AC(0,T; \mathcal{P}_2 (\Omega))$ with $\rho(\cdot, 0)=\rho_0$.
Furthermore, $\rho$ satisfies \eqref{T:ACweakSol:E_q}
and \eqref{T:ACweakSol:W_1}
with $C = C (\|V\|_{\mathcal{D}_{m,q,s}^{(q_1,q_2)}}, \|\rho_{0}\|_{L^{q} (\Omega)})$.
\end{itemize}
\end{theorem}

\begin{figure}
\centering

\begin{tikzpicture}[domain=0:16]

\fill[fill= lgray] (0,0) -- (0,4) -- (1.2,4) -- (4.2,1.5) -- (4.2,0);
\fill[fill= gray] (0, 0) -- (0, 2) --(0.6,2) -- (2.1,0.75)-- (2.1,0);

\draw[->] (0,0) -- (7,0) node[right] {\scriptsize $\frac{1}{q_1}$};
\draw[->] (0,0) -- (0,6) node[left] { \scriptsize $\frac{1}{q_2}$};

\draw (0,0)  node[left] {\scriptsize $O$};

\draw (0,5) node{\scriptsize $+$} node[left] {\scriptsize $1+\frac{d(q-1)}{2q+d(m-1)}$} 
-- (6, 0) node {\scriptsize $+$} node[below] {\scriptsize $\frac{q-1}{q}+\frac{2+q_{m,d}}{d}$};
\draw (2.7, 2.7) node[right]{\scriptsize $\mathcal{D}_{m,q}^{(q_1,q_2)}$};

\draw (1.2,4) node{\scriptsize $\bullet$} node[right]{\scriptsize $B$};
\draw[very thin] (1.2, 4) -- (0, 4)node{\scriptsize $\bullet$}node[left]{\scriptsize $A$};
\draw[dotted] (1.2,4) -- (1.2, 0) node{\scriptsize $*$} node[below]{\scriptsize $\frac{q-1}{q}$};

\draw (4.2,1.5) node{\scriptsize $\bullet$} node[right]{\scriptsize $C$};
\draw[dotted] (4.2,1.5) -- (0, 1.4)node{\scriptsize $*$}node[left]{\scriptsize $\frac{q+m-2}{q+m-1}$};
\draw[very thin] (4.2,1.5) -- (4.2, 0) node{\scriptsize $\bullet$} node[below]{\scriptsize $D$};

\draw (0,2.5) node{\scriptsize $+$} node[left] {\scriptsize $\frac{1}{2}+\frac{d(q-1)}{2[2q+d(m-1)]}$} 
-- (3, 0) node {\scriptsize $+$} node[below] {\scriptsize $g$};
\draw (1.35, 1.35) node[right]{\scriptsize $\mathcal{D}_{m,q,s}^{(q_1,q_2)}$};

\draw (0.6,2) node{\scriptsize $\bullet$} node[right]{\scriptsize $F$};
\draw[very thin] (0.6, 2) -- (0, 2)node{\scriptsize $\bullet$}node[left]{\scriptsize $E$};
\draw[dotted] (0.6,2) -- (0.6, 0) node{\scriptsize $*$} node[below]{\scriptsize $\frac{q-1}{2q}$};

\draw (2.1,0.75) node{\scriptsize $\bullet$} node[right]{\scriptsize $G$};
\draw[dotted] (2.1,0.75) -- (0,0.75)node{\scriptsize $*$}node[left]{\scriptsize $\frac{q+m-2}{2(q+m-1)}$};
\draw[very thin] (2.1,0.75) -- (2.1, 0) node{\scriptsize $\bullet$} node[below]{\scriptsize $H$};

\draw (0,1.1) node{\scriptsize $+$} -- (1.6, 0) node {\scriptsize $+$};
\draw (0.5, 0.75) node[right]{\scriptsize \textbf{$\mathsf{S}_{m,q}^{(q_1,q_2)}$}};

\draw(6, 5.5) node[right]{\scriptsize $\bullet$ the lightly-shaded region $\mathcal{R}(OABCD)$:};
\draw(8, 5) node[right]{\scriptsize $L^q$-weak solutions holding \eqref{T:weakSol:divfree:Vq}};
\draw(6,4.5) node[right]{\scriptsize $\bullet$ on $\mathcal{D}_{m,q}^{(q_1,q_2)}$: $B=(\frac{q-1}{q},1)$, $C=(\frac{q-1}{q}+\frac{2+q_{m,d}}{d(m+q-1)},\frac{q+m-2}{q+m-1})$ };
\draw(6,4) node[right]{\scriptsize $\bullet$ $A=(0,1)$, $D=(\frac{q-1}{q}+\frac{2+q_{m,d}}{d(m+q-1)},0)$ };

\draw(6, 3.5) node[right]{\scriptsize $\bullet$ the darkly-shaded region $\mathcal{R}(OEFGH)$:};
\draw(8, 3) node[right]{\scriptsize AC $L^q$-weak solutions holding \eqref{T:ACweakSol:divfree:Vq}};
\draw(6,2.5) node[right]{\scriptsize $\bullet$ on $\mathcal{D}_{m,q,s}^{(q_1,q_2)}$: $F=(\frac{q-1}{2q},\frac{1}{2})$, $G=(\frac{q-1}{2q}+\frac{2+q_{m,d}}{2d(m+q-1)},\frac{q+m-2}{2(q+m-1)})$ };
\draw(6,2) node[right]{\scriptsize $\bullet$ $E=(0,\frac{1}{2})$, $H=(\frac{q-1}{2q}+\frac{2+q_{m,d}}{2d(m+q-1)},0)$, $g=(\frac{q-1}{2q}+\frac{2+q_{m,d}}{2d},0)$ };

\draw(6,1.5) node[right]{\scriptsize $\bullet$ $\mathsf{S}_{m,q}^{(q_1,q_2)}: \frac{d}{q_1}+\frac{2+d(q+m-2)}{q_2} \leq 1+d(m-1)$ passing through };
\draw(8,1) node[right]{\scriptsize $(\frac{1+d(m-1)}{d},0)$ and $(0,\frac{1+d(m-1)}{2+d(q+m-2)})$. };

\end{tikzpicture}
\caption{\footnotesize (AC) $L^q$-weak solutions of FDE in case $\nabla \cdot V = 0$}
\label{F:divfree}
\end{figure}

\begin{remark}\label{R:divfree}
In Fig.~\ref{F:divfree}, the darkly-shaded region $\mathcal{R}(OEFGH)$ represents the admissible range of $(q_1,q_2)$ satisfying \eqref{T:ACweakSol:divfree:Vq}, while the lightly-shaded region $\mathcal{R}(OABCD)$ corresponds to the range satisfying \eqref{T:weakSol:divfree:Vq}. The figure illustrates that $\mathsf{S}_{m,q}^{(q_1,q_2)}$ for all $q\geq 1$ belongs to both shaded regions. 
When $d=2$ and $q\geq 2-m$, the line segments $\overline{GH}$ and $\overline{CD}$ are excluded from $\mathcal{R}(OEFGH)$ and $\mathcal{R}(OABCD)$, respectively. The restriction on $V$ in \eqref{T:V:divfree_delta} for the $\delta$-distance corresponds to the region $\mathcal{R}(OABCD)$ except for the polygonal line $\overline{ABCD}$.
\end{remark}

\begin{figure}
\centering

\begin{tikzpicture}[domain=0:16]

\fill[fill= lgray] (0,0) -- (0,3) -- (2.5,3) -- (4.8,1.2) -- (4.8,0);
\fill[fill= gray] (0, 0) -- (0, 1.5) --(1.6,1.5) -- (3,0.6)-- (3,0);

\draw[->] (0,0) -- (7,0) node[right] {\scriptsize $\frac{1}{q_1}$};
\draw[->] (0,0) -- (0,6) node[left] { \scriptsize $\frac{1}{q_2}$};

\draw (0,0)  node[left] {\scriptsize $O$};

\draw (0,5) node{\scriptsize $+$} node[left] {\scriptsize $1+\frac{d(q-1)}{2q+d(m-1)}$} 
-- (6.3, 0) node {\scriptsize $+$} node[below] {\scriptsize $\frac{q-1}{q}+\frac{2+q_{m,d}}{d}$};
\draw (2.8, 2.8) node[right]{\scriptsize $\mathcal{D}_{m,q}^{(q_1,q_2)}$};

\draw (2.5,3) node{\scriptsize $\bullet$} node[right]{\scriptsize $B$};
\draw[very thin] (2.5,3) -- (0, 3)node{\scriptsize $\bullet$}node[left]{\scriptsize $A$};
\draw[dotted] (2.5,3) -- (2.5, 0) node{\scriptsize $*$} node[below]{\scriptsize $\frac{q-1}{q}$};

\draw (4.8,1.2) node{\scriptsize $\bullet$} node[right]{\scriptsize $C$};
\draw[dotted] (4.8,1.2) -- (0, 1.2)node{\scriptsize $*$}node[left]{\scriptsize $\frac{q+m-2}{q+m-1}$};
\draw[very thin] (4.8,1.2) -- (4.8, 0) node{\scriptsize $\bullet$} node[below]{\scriptsize $D$};

\draw (0,2.5) node{\scriptsize $+$} node[left] {\scriptsize $\frac{1}{2}+\frac{d(q-1)}{2[2q+d(m-1)]}$} 
-- (4, 0) node {\scriptsize $+$};
\draw (2, 1.3) node[right]{\scriptsize $\mathcal{D}_{m,q,s}^{(q_1,q_2)}$};

\draw (1.6,1.5) node{\scriptsize $\bullet$} node[right]{\scriptsize $F$};
\draw[very thin] (1.6,1.5) -- (0, 1.5)node{\scriptsize $\bullet$}node[left]{\scriptsize $E$};
\draw[dotted] (1.6,1.5) -- (1.6, 0) node{\scriptsize $*$} node[below]{\scriptsize $\frac{q-1}{2q}$};

\draw (3,0.6) node{\scriptsize $\bullet$} node[right]{\scriptsize $G$};
\draw[dotted] (3,0.6) -- (0,0.6)node{\scriptsize $*$}node[left]{\scriptsize $\frac{q+m-2}{2(q+m-1)}$};
\draw[very thin] (3,0.6) -- (3, 0) node{\scriptsize $\bullet$} node[below]{\scriptsize $H$};

\draw (0,1.9) -- (3.5, 0);
\draw (0,1.9) node{\scriptsize $\bullet$} node[left]{\scriptsize $A'$};
\draw (3.5,0) node{\scriptsize $\bullet$} node[below]{\scriptsize $D'$};
\draw (0.8, 0.8) node[right]{\scriptsize \textbf{$\mathcal{S}_{m,q}^{(q_1,q_2)}$}};
\draw (3,0.25) node{\scriptsize $\bullet$} node[right]{\scriptsize $G'$};
\draw (0.7,1.5) node{\scriptsize $\bullet$} node[above]{\scriptsize $F'$};

\draw(5.5, 6) node[right]{\scriptsize $\bullet$ the lightly-shaded region $\mathcal{R}(OABCD)$: $L^q$-weak solutions};
\draw(5.5,5.5) node[right]{\scriptsize $\bullet$ $\mathcal{D}_{m,q}^{(q_1,q_2)}$: $\frac{d}{q_1} + \frac{2+q_{m,d}}{q_2} \leq 2+\frac{d(q+m-2)}{q}$ };
\draw(5.5,5) node[right]{\scriptsize $\bullet$ $B=(\frac{q-1}{q},1)$, $C=(\frac{q-1}{q}+\frac{2+q_{m,d}}{d(m+q-1)},\frac{q+m-2}{q+m-1})$ };
\draw(5.5,4.5) node[right]{\scriptsize $\bullet$ $A=(0,1)$, $D=(\frac{q-1}{q}+\frac{2+q_{m,d}}{d(m+q-1)},0)$ };

\draw(5.5, 3.8) node[right]{\scriptsize $\bullet$ the darkly-shaded region $\mathcal{R}(OEFGH)$: AC $L^q$-weak solutions};
\draw(5.5,3.3) node[right]{\scriptsize $\bullet$ $\mathcal{D}_{m,q,s}^{(q_1,q_2)}$: $\frac{d}{q_1} + \frac{2+q_{m,d}}{q_2} \leq 1+\frac{d(q+m-2)}{2q}$ };
\draw(5.5,2.8) node[right]{\scriptsize $\bullet$ $F=(\frac{q-1}{2q},\frac{1}{2})$, $G=(\frac{q-1}{2q}+\frac{2+q_{m,d}}{2d(m+q-1)},\frac{q+m-2}{2(q+m-1)})$ };
\draw(5.5,2.3) node[right]{\scriptsize $\bullet$ $E=(0,\frac{1}{2})$, $H=(\frac{q-1}{2q}+\frac{2+q_{m,d}}{2d(m+q-1)},0)$ };

\draw(6.5,1.5) node[right]{\scriptsize $\bullet$ $\mathcal{S}_{m,q}^{(q_1,q_2)}$: $\frac{d}{q_1} + \frac{2+q_{m,d}}{q_2} \leq 1+q_{m,d}$ };
\draw(6.5,1) node[right]{\scriptsize $\bullet$ $L^q$-weak: $\mathcal{R}(OA'B')$ in \cite[Theorem~2.8 \& 2.10]{HKK02}};
\draw(6.5,0.5) node[right]{\scriptsize $\bullet$ AC $L^q$-weak: $\mathcal{R}(OEF'G'H)$ in \cite[Theorem~2.12]{HKK02}};

\end{tikzpicture}
\caption{\footnotesize (AC) $L^q$-weak solutions of PME in case $\nabla \cdot V = 0$ and $q>m$ }
\label{F:divfree:PME}
\end{figure}

\begin{remark}[Comparison with the porous medium case]\label{R:PME}
Compared with the previous divergence-free results for the porous medium equation in \cite{HKK02}, the compactness argument and speed estimate developed in the present paper allow admissible drift classes in supercritical regimes relative to the scaling-invariant class.

For $L^q$-weak solutions, the main improvement comes from the compactness argument. In \cite[Proposition~4.19]{HKK01}, the compactness argument for PME was based on estimates for $\partial_t \rho^q$ and $\nabla \rho^q$, and was therefore formulated at the level of $\rho^q$. In Proposition~\ref{Proposition : compactness}, we instead establish compactness directly for $\rho$ by estimating $\partial_t\rho$ and $\nabla\rho$ in suitable spaces, together with a truncation argument. This direct compactness framework leads to the drift class $\mathcal{D}_{m,q}^{(q_1,q_2)}$, which includes supercritical regimes relative to the scaling-invariant class $\mathcal{S}_{m,q}^{(q_1,q_2)}$. In Fig.~\ref{F:divfree:PME}, the previous $L^q$-weak-solution results in \cite[Theorems~2.8 and 2.10]{HKK02} correspond to $\mathcal{R}(OA'D')$, while Theorem~\ref{T:weakSol:divfree:PME} extends the admissible region to $\mathcal{R}(OABCD)$.

For AC $L^q$-weak solutions in the divergence-free PME case with $q>m$, the previous result \cite[Theorem~2.12]{HKK02} was obtained under the scaling-invariant drift condition $V\in\mathcal{S}_{m,q}^{(q_1,q_2)}$ and provides a $W_2$-distance estimate. In the present paper, we obtain the same type of $W_2$-distance estimate under the class $\mathcal{D}_{m,q,s}^{(q_1,q_2)}$, which lies in the supercritical regime of $\mathcal{S}_{m,q}^{(q_1,q_2)}$ when $q>m$. In Fig.~\ref{F:divfree:PME}, the previous admissible region corresponds to $\mathcal{R}(OEF'G'H)$ in \cite[Theorem~2.12]{HKK02}, while Theorem~\ref{T:ACweakSol:divfree:PME} extends it to $\mathcal{R}(OEFGH)$.

The main statements and detailed computations for these PME consequences are given in Appendix~\ref{Appendix:PME}.
\end{remark}

\subsection{Application}

In this section, as an application, we study a viscous Boussinesq system of the fast diffusive type. 
The Boussinesq system models the coupling between incompressible fluid motion and temperature variation through buoyancy forces. It is a basic model for thermal convection and buoyancy-driven flows.

Let \( \mathbf{e}_d \) denote either \( (0,1) \) or \( (0,0,1) \). 
We consider a fast-diffusive version of the viscous Boussinesq system,
\begin{equation}\label{BE-10}
\theta_t-\Delta \theta^m +u\cdot \nabla \theta=0,
\end{equation}
\begin{equation}\label{BE-20}
u_t-\Delta u+(u\cdot \nabla) u+\nabla \pi=-\theta {\bf{e}_d},\qquad {\rm div}\,u=0.
\end{equation}
in $\Omega_T=\Omega\times (0, T)$ where $\Omega\subset\R^d$, $d=2,3$ is a bounded domain with smooth boundary, whereby the following boundary conditions are assigned 
\begin{equation*}\label{BE-30}
\frac{\partial\theta}{\partial {\bf n}}=0,\quad u=0\qquad \mbox{ on }\,\,\partial\Omega\times (0,T).
\end{equation*}
Here $\theta$ and $u$ indicate the temperature and fluid velocity, respectively. 
We suppose that initial conditions satisfy
\begin{equation}\label{BE-40}
\theta_0\in L^1(\Omega),\quad \int_{\Omega} \theta_0 \log \theta_0 dx<\infty,\quad  u_0\in L^2(\Omega).
\end{equation}

The main motivation is that the temperature equation is a fast diffusion equation with drift field \(u\). Since the velocity is divergence-free, the transport term \(u\cdot\nabla\theta\) does not contribute to the energy estimate for \(\theta\). Hence the system provides a natural application of the existence theory for fast diffusion equations with divergence-free drift fields. At the same time, the drift is not prescribed but is determined by the coupled Navier--Stokes equation, which makes the system a nontrivial and physically relevant application of the theory.

Our goal is to construct weak solutions for \eqref{BE-10}--\eqref{BE-20}. 
For convenience, we introduce the notion of a weak solution.

\begin{defn}\label{Weak-ks} Let $0<m<1$ and $\Omega\subset\R^d$, $d=2,3$ be a bounded domain with smooth boundaries. We say that a pair of $(\theta, u)$ is a  weak solution of \eqref{BE-10}-\eqref{BE-20} in $\Omega_{T} := \Omega \times (0, T)$ with zero boundary data if the followings are satisfied:
\begin{itemize}
	\item [(i)] It holds that 
	\[
	\theta , \ \nabla \theta^m, \ \theta u , \ u\otimes u, \ \nabla u \in L_{x,t}^{1}.
	\]
	
	\item [(ii)] For any $\varphi \in \mathcal{C}_{c}^{\infty} \left(\overline{\Omega} \times [0, T);\R\right)$ and $\psi \in \mathcal{C}_{c}^{\infty} \left(\overline{\Omega} \times [0, T);\R^d\right)$ with $\rm{div}\,\psi=0$ both of which vanish on $\partial\Omega$, it holds that 
	\[
	\iint_{\Omega_{T}} \left\{ - \theta \varphi_t + \nabla \theta^m \nabla \varphi - \theta u\cdot \nabla \varphi \right\} \,dxdt = \int_{\Omega}\theta_0 \varphi(\cdot,0) \,dx,
	\]
	\[
	\iint_{\Omega_{T}} \left\{ - u \psi_t + \nabla u \nabla \psi - u\otimes u:\nabla \psi \right\} \,dxdt = \int_{\Omega}u_0 \psi(\cdot,0) \,dx-\iint_{\Omega_{T}}  \theta {\bf e}_d  \psi \,dxdt,
	\]
\end{itemize}	
\end{defn}

It is well known that in the case \( m = 1 \), smooth solutions exist globally in time, provided that the initial data are sufficiently regular in two dimensions; see \cite{Cannon_DiBenedetto}. On the other hand, the question of regularity in three dimensions remains open.

 In contrast, the fast-diffusive case \(0<m<1\) is less standard,
because the diffusion term \(\Delta\theta^m\) is singular near the vacuum
region. The temperature equation can be viewed as a fast diffusion equation
with drift \(u\), and the incompressibility condition
\[
\nabla\cdot u=0
\]
implies that the drift term does not contribute to the energy estimate.
This divergence-free structure allows one to combine the weak existence theory
for fast diffusion equations with divergence-free drifts and the classical
energy estimates for the velocity equation.
Our main objective in this subsection  is to establish weak solutions, satisfying an energy estimate \eqref{BS-energy} when initial data satisfy \eqref{BE-40}, and we establish the existence of weak solutions for some range of $0<m<1$.

We are now ready to state the main result for equations \eqref{BE-10}--\eqref{BE-20}. In the proof, we provide a priori estimates, as such estimates can also be derived without difficulty from the regularized system.

\begin{theorem}
Let  $\Omega\subset\R^d$, $d=2,3$ be a bounded domain with smooth boundaries. 
If $0< m <1$ when $d=2$ and if $\frac{2}{3} \le m<1$ when $d=3$, then,
there exists a pair of weak solution $(\theta, u)$ in Definition \ref{Weak-ks}. Furthermore, for any  $(\theta, u)$ satisfies 
\begin{equation}\label{BS-energy}
\sup_{t\in (0, T)} \int_{\Omega} (\theta(\cdot, t) +\abs{u(\cdot, t)}^2) \,dx
+\iint_{\Omega_T} (\left|\nabla \theta^{\frac{m}{2}}\right|^{2} +\abs{\nabla u}^2) \,dxdt\le C\bke{\int_{\Omega}\theta_0 \log \theta_0 \,dx,  \norm{\theta_0}_{L^1}, \norm{u_0}_{L^{2}_{x}}}.
\end{equation}
\end{theorem}
\begin{proof}
We first note that 
\begin{equation}\label{ks-pf-100}
\frac{d}{dt}\int_{\Omega} \theta(\cdot, t)\log \theta(\cdot, t) \,dx+\int_{\Omega} \left|\nabla \theta^{\frac{m}{2}}\right|^{2} \,dx= 0
\end{equation}

We treat the two and three dimensions separately.
\\
$\bullet$ \underline{(2d case)}\,\,
Let $p\in (1, \infty)$ with $\frac{2}{m}\le p$.
For the fluid equation, energy estimate gives
\[
\frac{1}{2}\frac{d}{dt}\norm{u}_{L^{2}_{x}}^2+\norm{\nabla u}_{L^{2}_{x}}^2\lesssim \int_{\Omega}\theta \abs{u}dx\lesssim \norm{\theta}_{L^{q}_{x}}\norm{u}_{L^{p}_{x}} \lesssim \norm{\theta}_{L^{q}_{x}}\norm{u}^{\frac{2}{p}}_{L^{2}_{x}} \norm{\nabla u}^{\frac{p-2}{p}}_{L^{2}_{x}},
\]
where $\frac{1}{q}+ \frac{1}{p}=1$.
Due to Young's inequality, we have
\[
\frac{d}{dt}\norm{u}_{L^{2}_{x}}^2+\norm{\nabla u}_{L^{2}_{x}}^2\lesssim \norm{\theta}^{\frac{2p}{p+2}}_{L^{q}_{x}}\norm{u}^{\frac{4}{p+2}}_{L^{2}_{x}} =\norm{\theta^{\frac{m}{2}}}^{\frac{4p}{m(p+2)}}_{L^{\frac{2q}{m}}_{x}}\norm{u}^{\frac{4}{p+2}}_{L^{2}_{x}}.
\]
Noting that $\norm{\theta(t)}_{L^1}=\norm{\theta_0}_{L^1}$ and
\[
\norm{\theta^{\frac{m}{2}}}^{\frac{4p}{m(p+2)}}_{L^{\frac{2q}{m}}_{x}}\le \bke{\norm{\theta^{\frac{m}{2}}}^{\frac{1}{q}}_{L^{\frac{2}{m}}_{x}}\norm{\nabla \theta^{\frac{m}{2}}}^{\frac{1}{p}}_{L^{2}_{x}}}^{\frac{4p}{m(p+2)}}\lesssim\norm{\nabla \theta^{\frac{m}{2}}}^{\frac{4}{m(p+2)}}_{L^{2}_{x}},
\]
we obtain
\begin{equation}\label{ks-pf-110}
\frac{d}{dt}\norm{u}_{L^{2}_{x}}^2+\norm{\nabla u}_{L^{2}_{x}}^2\lesssim \norm{\nabla \theta^{\frac{m}{2}}}^{\frac{4}{m(p+2)}}_{L^{2}_{x}}\norm{u}^{\frac{4}{p+2}}_{L^{2}_{x}}
\lesssim \epsilon\norm{\nabla \theta^{\frac{m}{2}}}^{2}_{L^{2}_{x}}+\norm{u}^{\frac{4m}{m(p+2)-2}}_{L^{2}_{x}}.
\end{equation}
Combining \eqref{ks-pf-100} and \eqref{ks-pf-110}, we obtain that 
\begin{equation*}
\frac{d}{dt}\int_{\Omega} (\theta \log \theta +\abs{u}^2)\,dx+\iint_{\Omega_T} (\left|\nabla \theta^{\frac{m}{2}}\right|^{2} +\abs{\nabla u}^2) \,dxdt\lesssim \norm{u}^{\frac{4m}{m(p+2)-2}}_{L^{2}_{x}}.
\end{equation*}
It is required that 
\[
\frac{4m}{m(p+2)-2}\le 2 \quad \Longrightarrow \quad \frac{2}{m} \le p.
\]
Since $u\in L^{a, b}_{x,t}$ where $\frac{2}{a}+\frac{2}{b}=1$ with $2\le a<\infty$, it is automatic that $u\in\mathcal{D}_{m,1}^{(q_1, q_2)}$ with $d=2$.
Indeed, in case $m\le \frac{1}{2}$, it is straightforward. If $m \in (0, \frac{1}{2})$, then we can take $q_2 \le \frac{2(1-m)}{(1-2m)}$ and then choose $q_1$ with $\frac{2}{q_1}+\frac{2}{q_2}=1$, which implies that $u\in\mathcal{D}_{m,1}^{(q_1, q_2)}$, $d=2$.
This completes the proof for 2d case.
\\
\\
$\bullet$ \underline{(3d case)}\,\, 
As similarly, we have 
\begin{equation}\label{ks-pf-120}
\frac{1}{2}\frac{d}{dt}\norm{u}_{L^{2}_{x}}^2+\norm{\nabla u}_{L^{2}_{x}}^2\lesssim \int_{\Omega}\theta \abs{u}dx\lesssim \norm{\theta}_{L^{\frac{6}{5}}_{x}}\norm{u}_{L^{6}_{x}} \lesssim \norm{\theta}^{2}_{L^{\frac{6}{5}}_{x}} +\frac{1}{2}\norm{\nabla u}^2_{L^{2}_{x}}
\end{equation}
Via interpolation, it gives
\begin{equation}\label{ks-pf-130}
\norm{\theta}^{2}_{L^{\frac{6}{5}}_{x}}=\norm{\theta}^{\frac{4}{m}}_{L^{\frac{12}{5m}}_{x}}\le \norm{\theta^{\frac{m}{2}}}^{\frac{10m-4}{m(3m-1)}}_{L^{\frac{2}{m}}_{x}}\norm{\nabla \theta^{\frac{m}{2}}}^{\frac{2}{3m-1}}_{L^{2}_{x}}\lesssim \norm{\nabla \theta^{\frac{m}{2}}}^{\frac{2}{3m-1}}_{L^{2}_{x}}.
\end{equation}
It is required that
\[
\frac{2}{3m-1}\le 2 \quad \Longrightarrow \quad  m \geq \frac{2}{3}.
\]
We note that $u\in L^{6,2}_{x,t}$, which belongs to $u\in\mathcal{D}_{m,1}^{(q_1, q_2)}$, $d=3$. Indeed, since $u\in L^{6,2}_{x,t}$, it is direct that 
\[
\frac{3}{6} + \frac{2+3(m-1)}{2} \leq 2+ 3(m-1)\quad\Longrightarrow\quad \frac{2}{3}\leq m.
\]
Combining \eqref{ks-pf-100}, \eqref{ks-pf-120} and \eqref{ks-pf-130}, we complete the proof.
\end{proof}

\section{Preliminaries}\label{SS:Preliminaries} 

\subsection{Technical lemmas}

In this section, we introduce preliminaries that are used throughout the paper.

For a function $f:\Omega \times [0,T] \to \mathbb{R}$, $\Omega \subset \mathbb{R}^d, d \geq 2$ and constants $q_1, q_2 > 1$, we define
\[
\|f\|_{L^{q_1, q_2}_{x, t}} : = \left(\int_{0}^{T}\left[\int_{\Omega} \abs{f(x,t)}^{q_1} \,dx\right]^{q_2 / q_1} \,dt\right)^{\frac{1}{q_2}}.
\]
For simplicity, let $\|f\|_{L^{q}_{x,t}} = \|f\|_{L^{q, q}_{x,t}}$ for some $q>1$. Also we denote for $\Omega \subseteq \bbr^d$ and $\Omega_T \subseteq \Omega_T$ that
\begin{equation*}
\begin{gathered}
\|f(\cdot,t)\|_{\calC^{\alpha}(\Omega)} := \sup_{x,y \in \Omega, \, x \neq y} \frac{|f(x, t) - f(y,t)|}{|x-y|^\alpha}, \\
\|f\|_{\calC^{\alpha}(\Omega_T)} := \sup_{(x, t), (y,s) \in \Omega_T, \, (x,t) \neq (y,s)} \frac{|f(x, t) - f(y,s)|}{|x-y|^\alpha + |t-s|^{\alpha/2}}.
\end{gathered}
\end{equation*}

Let us introduce a parabolic embedding theory. 
\begin{lemma}\label{T:pSobolev} \cite[Propositions~I.3.1 \& I.3.2]{DB93}
  Let $v \in L^{\infty}(0, T; L^{q}(\Omega)) \cap L^{p}(0, T; W^{1,p}(\Omega))$ for some $1 \leq p < d$, and $0 < q < \frac{dp}{d-p}$. Then there exists a constant $c=c(d,p,q, \abs{\Omega})$ such that
  \[
  \iint_{\Omega_{T}} |v(x,t)|^{\frac{p(d+q)}{d}}\,dx\,dt \leq c \left(\sup_{0\leq t \leq T}\int_{\Omega} |v(x,t)|^{q}\,dx \right)^{\frac{p}{d}} \iint_{\Omega_{T}} |\nabla v(x,t)|^{p}\,dx\,dt +  \frac{1}{|\Omega|^{\frac{p(d+q)}{d}-1}}\int_{0}^{T} \|v(\cdot, t)\|_{L^{1}(\Omega)}^{\frac{p(d+q)}{d}}\,dt.
  \]
\end{lemma}

Now we derive the following interpolation inequality. Depending on the context, we choose $p=1$ or $p=q$.
 \begin{lemma}\label{L:interpolation}\cite[Lemma~3.4]{HKK01}
 Let $0<m<1$, $q \geq 1$, and $1\leq p \leq q$. Suppose that
 \begin{equation*}\label{rho-space}
 \rho \in L^{\infty}(0, T; L^{p}(\Omega)) \quad \text{and} \quad
 \rho^{\frac{m+q-1}{2}} \in L^{2}(0, T; W^{1,2}(\Omega)).
 \end{equation*}
  Then 
 \begin{equation*}\label{L:r1r2}
 \begin{gathered}
\rho\in L^{r_1, r_2}_{x, t} (\Omega_{T}), \ \text{ where } \ \frac{d}{r_1} + \frac{2 + \frac{d}{p}(m+q-1-p)}{r_2} = \frac{d}{p}  \\
 \text{for } \  
 \begin{cases}
  p \leq r_1 \leq \frac{d(m+q-1)}{d-2},  \quad m+q-1 \leq r_2 \leq \infty, & \text{if } \ d > 2, \\
  p \leq r_1 < \infty, \quad m+q-1 < r_2 \leq \infty, &\text{if }\ d = 2, \\
  p \leq r_1 \leq \infty, \quad m+q-1 \leq r_2 \leq \infty, &\text{if } \ d= 1.
  \end{cases}
\end{gathered}
\end{equation*}
Moreover, there exists a constant $c=c(d)$ such that
 \begin{equation*}\label{norm-q-r1r2}
\|\rho\|_{L^{r_1, r_2}_{x, t}} \leq c\left(\sup_{0\leq t \leq T} \int_{\Omega} \rho^{p}(\cdot, t) \,dx\right)^{\gamma} \,
\left\|\nabla \rho^{\frac{q+m-1}{2}} \right\|_{L^{2}_{x,t}}^{\frac{2}{r_2}}
+ \frac{1}{|\Omega|^{1-\frac{1}{r_1}}} \left(\int_{0}^{T} \|\rho(\cdot, t)\|_{L^{1}(\Omega)}^{r_1}\,dt\right)^{\frac{1}{r_2}},
\end{equation*}
where $\gamma = \frac{d p(q+m-1)}{2p+d(m+q-1-p)} \left[\frac{1}{r_1}-\frac{d-2}{d(m+q-1)}\right]$.
 \end{lemma}

The following is the Aubin-Lions lemma.
\begin{lemma}[See \cite{Sim87}]\label{AL}
Let $X_0$, $X$, and $X_1$ be Banach spaces such that $X_0$ is compactly embedded in $X$ and $X$ is continuously embedded in $X_1$.
Let $1\leq p,r \leq \infty$.
For $T>0$, define
\[
W = \left\{ v\in L^{p}(0,T;X_0) \,:\, \partial_t v\in L^{r}(0,T;X_1)\right\},
\]
where $\partial_t v$ is understood in the sense of distributions.
If $p<\infty$, then the embedding of $W$ into $L^{p}(0,T;X)$ is compact.
If $p=\infty$ and $r>1$, then the embedding of $W$ into $C([0,T];X)$ is compact.
\end{lemma}

\subsection{Wasserstein space}\label{SS:Wasserstein}
In this subsection, we introduce the Wasserstein space and its properties. For more detail, we refer \cite{ags:book, Santambrosio15, V}.
 Let us denote by $\mathcal{P}_p (\Omega)$ the set of all Borel probability measures on $\Omega$ with a finite $p$-th moment. That is,
 $\mathcal{P}_p (\Omega):= \{\mu \in \mathcal{P} (\Omega) : \int_\Omega |x|^p \,d \mu(x) < \infty \}$.
 We note that $\mathcal{P}_p (\Omega)=\mathcal{P} (\Omega)$ if $\Omega$ is a bounded set.
For $\mu,\nu\in\mathcal{P}_p (\Omega)$, we consider
\begin{equation}\label{Wasserstein dist}
W_p(\mu,\nu):=\left(\inf_{\gamma\in\Gamma(\mu,\nu)}\int_{\Omega\times \Omega}|x-y|^p\, d\gamma(x,y)\right)^{\frac{1}{p}},
\end{equation}
where $\Gamma(\mu,\nu)$ denotes the set of all Borel probability measures on $\Omega\times \Omega$ which has $\mu$ and
$\nu$ as marginals;
$$\gamma(A\times \Omega)= \mu(A) \quad \text{and} \quad \gamma(\Omega\times A)= \nu(A) $$
for every Borel set $A\subset \Omega.$
Equation (\ref{Wasserstein dist}) defines a distance on $\mathcal{P}_p (\Omega)$ which is called the {\it Wasserstein distance} and denoted by $W_p$.
 Equipped with the Wasserstein distance,  $\mathcal{P}_p (\Omega)$ is called the {\it Wasserstein space}.
We denote by $\Gamma_o(\mu,\nu)$ the set of all $\gamma$ which minimize the expression.

We say that a sequence of Borel probability measures $\{\mu_n\} \subset \mathcal{P}(\Omega)$ is narrowly convergent to $ \mu \in \mathcal{P}(\Omega)$ as
$n \rightarrow \infty$ if
\begin{equation*}\label{D:narrowly convergent}
\lim_{n\rightarrow \infty} \int_{\Omega} \varphi(x) \,d\mu_n(x) =\int_{\Omega} \varphi(x) \,d\mu(x)
\end{equation*}
for every function $ \varphi \in \calC_b (\Omega)$, the space of continuous and bounded real functions defined on $\Omega$. It is well known
 that narrow convergence is induced by a distance on $\mathcal{P}(\Omega)$ 
 \begin{equation}\label{Narrow-distance}
 \delta(\mu, \nu):= \sum_{k=1}^\infty 2^{-k} \left | \int_\Omega f_k \, d\mu -\int_\Omega f_k \, d\nu \right |,
  \end{equation}
where $\{ f_k \}$ is a sequence of functions satisfying $\|f_k \|_{W^{1,\infty}(\Omega)} \leq 1$.  

We recall 
 \begin{equation*}
\lim_{n \rightarrow \infty} W_p(\mu_n, \mu)=0 \quad \Longleftrightarrow \quad
\begin{cases}
\mu_n ~\mbox{narrowly converges to } \mu ~\mbox{ in }  \mathcal{P}(\Omega),\\
\lim_{n\rightarrow \infty} \int_{\Omega} |x|^p\, d\mu_n(x) = \int_{\Omega} |x|^p\,d \mu(x).
\end{cases}
\end{equation*}
Hence, if $\Omega$ is compact then $\mathcal{P}_p(\Omega)$ is also a compact metric space.%


Now, we introduce the notion of absolutely continuous curve and its relation with the continuity equation.
\begin{defn}
Let $\sigma:[0,T]\mapsto \mathcal{P}_p (\Omega)$ be a curve.
We say that $\sigma$ is absolutely continuous and denote it by $\sigma \in AC(0, T;\mathcal{P}_p (\Omega))$, if there exists $l\in
L^1([0,T])$ such that
\begin{equation}\label{AC-curve}
W_p(\sigma(s),\sigma(t))\leq \int_s^t l(r)dr,\qquad \forall ~ 0\leq s\leq t\leq T.
\end{equation}
If $\sigma \in AC(0,T;\mathcal{P}_p (\Omega))$, then the limit
$$|\sigma'|(t):=\lim_{s\rightarrow t}\frac{W_p(\sigma(s),\sigma(t))}{|s-t|} ,$$
exists for $L^1$-a.e $t\in[0,T]$. Moreover, the function $|\sigma'|$ belongs to $L^1(0,T)$ and satisfies
\begin{equation*}
|\sigma'|(t)\leq l(t) \qquad \mbox{for} ~L^1\mbox{-a.e.}~t\in [0,T],
\end{equation*}
for any $l$ satisfying \eqref{AC-curve}.
We call $|\sigma'|$ by the metric derivative of $\sigma$.
\end{defn}

\begin{lemma}\label{representation of AC curves} \cite[Theorem 5.14]{Santambrosio15}
If a narrowly continuous curve $\sigma : [0,T] \mapsto \mathcal{P}_p (\Omega)$ satisfies the continuity equation
$$\partial_t\sigma +\nabla\cdot(V\sigma)=0,  $$
for some Borel vector field $V$ with $\|V(t)\|_{L^p(\sigma(t))}\in L^1(0,T)$, then $\sigma: [0,T]\mapsto \mathcal{P}_p (\Omega)$
is absolutely continuous and
$|\sigma'|(t)\leq \|V(t)\|_{L^p(\sigma(t))}$ for $L^1$-a.e $t\in [0,T)$.
\end{lemma}

\begin{lemma}\label{Lemma : Arzela-Ascoli} \cite[Proposition 3.3.1]{ags:book}
Let $K  \subset \mathcal{P}_p(\Omega)$ be a sequentially compact set w.r.t the narrow topology.
Let $\sigma_n : [0, T] \rightarrow \mathcal{P}_p (\Omega)$ be curves such that 
\begin{equation*}\label{equi-continuity}
\begin{aligned}
&\sigma_n(t) \in K,  &&\forall ~ n \in \mathbb{N}, ~ t \in [0,T],\\
&\delta (\sigma_n(s),\sigma_n(t))\leq \omega(s,t), &&\forall ~
0\leq s\leq t\leq T, ~ n \in \mathbb{N},
\end{aligned}
\end{equation*}
for a continuous function $\omega : [0, T] \times [0, T] \rightarrow [0, \infty)$ such that
$$\omega(t,t)=0, ~~ \forall ~ t \in [0,T].$$
Then there exists a subsequence $\sigma_{n_k}$ and a limit curve $\sigma : [0,T] \rightarrow {P}_p(\Omega)$ such that
\begin{equation*}\label{eq1 : Lemma : Arzela-Ascoli}
\sigma_{n_k}(t) \mbox{ narrowly converges to} ~~\sigma(t), \qquad
\text{for all} ~~t\in[0, T].
\end{equation*}
\end{lemma}

 \subsection{Flows on $\mathcal{P}_p(\Omega)$ generated by vector fields}
For a given $T>0$, let $V\in L^1(0,T; W^{1,\infty}(\Omega;\mathbb{R}^d))$ be such that $V\cdot \mathbf{n}=0$ on $\partial \Omega$ where $\mathbf{n}$ is the outward unit normal to the
boundary of $\Omega$. For any $s, ~t\in [0,T]$, let
$\psi:[0,T]\times [0,T]\times\Omega\mapsto \Omega$ be the flow map
of the vector field $V$.
 More precisely, $\psi$ solves the following ODE
 \begin{equation}\label{ODE}
\begin{cases}
\frac{d}{dt}\psi(t;s,x)= V(\psi(t;s,x),t),  &\text{for } s, t\in[0,T] \vspace{1 mm}\\
\psi(s;s,x)=x,  &\text{for }  x\in\Omega.
\end{cases}
\end{equation}
Using the flow map $\psi$, we define a flow $\Psi: [0,T]\times [0,T]\times \mathcal{P}_p(\Omega)\mapsto \mathcal{P}_p(\Omega)$ through the push forward operation as follows
\begin{equation}\label{Flow on Wasserstein}
\Psi(t;s,\mu):={\psi}(t;s,\cdot)_\# \mu, \qquad \forall ~ \mu \in \mathcal{P}_p(\Omega).
\end{equation}

In this subsection, we remind two basic results on the flow map $\psi$.

\begin{lemma}\label{Lemma : Lipschitz of Jacobian}\cite[Lemma~3.2]{KK-SIMA}
Let $s\in [0,T]$ and $\psi$ be defined as in \eqref{ODE}. For any $t\in [s,T]$, let $J_{s,t}$ be the Jacobian corresponding to the map
 $\psi(t;s,\cdot):\Omega\mapsto \Omega$ . That is,
\begin{equation}\label{Jacobian}
\int_{\Omega} \zeta(y) dy:= \int_{\Omega} \zeta(\psi(t;s,x))
J_{s,t}(x)\,dx,\qquad \forall ~\zeta\in C(\Omega) .
\end{equation}
Then, the Jacobian  $ J_{s,t}$ is given as
\begin{equation*}\label{Jacobian - formular}
J_{s,t}(x)=e^{\int_s^t \nabla \cdot V (\psi(\tau;s,x),\tau) \,d\tau}, \qquad \forall ~ x\in \Omega.
\end{equation*}
\end{lemma}

\begin{lemma}\label{Lemma : density relation on the flow}\cite[Lemma~3.4]{KK-SIMA}
Let $\psi$ and $\Psi$ be defined as in \eqref{ODE} and \eqref{Flow on Wasserstein}, respectively. If $\mu \in \mathcal{P}_p^{ac}(\Omega)$ then
$\Psi(t;s,\mu)\in \mathcal{P}_p^{ac}(\Omega)$. Here, $ \mathcal{P}_p^{ac}(\Omega)$ is the set of all probability measures in $\mathcal{P}_p(\Omega)$ which are absolutely
continuous with respect to the {\it Lebesgue measure} in $\Omega$. Furthermore, suppose $\mu=\varrho \,dx$ and $\Psi(t;s,\mu)=\rho \,dx$ then
\begin{equation*}\label{Density relation}
\rho(\psi(t;s,x))J_{s,t}(x)=\varrho(x), \qquad a.e \quad x\in \Omega,
\end{equation*}
where $J_{s,t}$ is the Jacobian of the map $\psi(t;s,\cdot)$ as in \eqref{Jacobian}.
We also have
\begin{equation*}\label{Entropy relation} \int_{\Omega} \rho \log
\rho \,dx = \int_{\Omega}\varrho \log \varrho \,dx -
\int_{\Omega}\varrho \log J_{s,t}  \,dx.
\end{equation*}
Moreover, if $\varrho \in L^{q}(\Omega)$ for $q\in[1,\infty]$, then
${\rho}\in L^{q}(\Omega)$ and we have
\begin{equation*}\label{L^p relation}
\|{\rho}\|_{L^{q}(\Omega)} \leq
\|\varrho\|_{L^{q}(\Omega)}e^{\frac{q-1}{q}\int_s^t\|\nabla\cdot
V\|_{L^\infty_x} \,d\tau},
\end{equation*}
where $\frac{q-1}{q}=1$ if $q=\infty$.
\end{lemma}

\begin{lemma}\label{Lemma : Holder regularity on the flow}\cite[Lemma~3.19]{HKK01}
Let $\psi$ and $\Psi$ be defined as in \eqref{ODE} and \eqref{Flow on Wasserstein}, respectively. 
\begin{itemize}
\item[(i)] If $V \in L^1(0,T;  \calC^{1, \alpha}(\Omega)) $ and $\varrho\in
\mathcal{P}_p(\Omega) \cap \calC^\alpha (\Omega)$ for some $\alpha \in (0,1)$,
then $\rho:=\Psi(t;s,\varrho)$ is also H\"{o}lder continuous. More
precisely, we have $\| \rho\|_{\calC^\alpha (\Omega)} \leq C $ where $C=C(\| \varrho\|_{\calC^\alpha (\Omega)},\, \int_s^t \|\nabla V \|_{\calC^\alpha (\Omega)} d\tau )$.
\item[(ii)] If $V \in L^1(0,T;  W^{2, \infty}(\Omega)) $, then, for any  $a>0,~ q\geq 1$, we have
\begin{equation*}\label{eq5 : Sobolev}
\begin{aligned}
 \|\nabla \rho^a\|_{L^q (\Omega)}&\leq e^{(a+2)\int_s^t \|\nabla V\|_{L^\infty_x}\, d\tau}\left \{ \|\nabla \varrho^a\|_{L^q (\Omega)} + \| \varrho^{a}\|_{L^q (\Omega)}
  \left ( a\int_s^t \| \nabla^2 V\|_{L^\infty(\Omega)} \,d\tau\right ) \right \}.
\end{aligned}
\end{equation*}
\end{itemize}
\end{lemma}

\section{A priori estimates}\label{SS:priori}

In this section, we provide estimates of energy and speed, and energy-speed combined estimates, and compactness arguments of \eqref{FDE}-\eqref{FDE_bc_ic} given as following notion of solution.

\begin{defn}\label{D:regular-sol}
Let $0< m < 1$, $q\geq 1$ and $V$ be a measurable vector field.
We say that a nonnegative measurable function $\rho$ is
a \textbf{regular solution} of \eqref{FDE}-\eqref{FDE_bc_ic} with nonnegative initial data $\rho_0 \in \calC^{\alpha}(\overline{\Omega})$ for some $\alpha \in (0,1)$ if the followings are satisfied:
\begin{itemize}
\item[(i)] It holds that
\[
\rho \in L^{\infty}(0,T;\mathcal{C}^{\alpha}(\overline{\Omega})), 
\ \nabla \rho^m \in L^1(\Omega_T), 
\ \nabla \rho^{\frac{q+m-1}{2}} \in L^{2} (\Omega_T),
\ \text{ and } \ V \in L^{1}_{\loc}(\Omega_{T}).
\]
\item[(ii)] For any function $\varphi \in {\mathcal{C}}^\infty_c (\overline{\Omega} \times [0,T))$, it holds that 
\[
\iint_{\Omega_{T}} \left\{  \rho \varphi_t - \nabla \rho^m \cdot  \nabla \varphi + \rho V\cdot \nabla \varphi \right\} \,dx\,dt = -\int_{\Omega} \rho_{0} \varphi(\cdot, 0) \,dx.
\]
\end{itemize}
\end{defn}

\subsection{Energy estimates}

Here, we deliver a priori estimates of \eqref{FDE}-\eqref{FDE_bc_ic} categorized by $V$. 

\begin{proposition}\label{P:energy}
 Suppose that $\rho$ is a regular solution of \eqref{FDE}-\eqref{FDE_bc_ic} for $0<m<1$ and $d\geq 2$ with nonnegative initial data $\rho_{0}\in  \mathcal{P}_2(\Omega) \cap \calC^{\alpha}(\overline{\Omega})$. 

\begin{itemize}
\item[(i)] Let $ 1-\frac{1}{d} < m <1$. Suppose that 
\begin{equation}\label{V:Lq-energy}
V \in \mathsf{S}_{m,q}^{(q_1,q_2)}\cap \calC^{\infty} (\overline{\Omega}_{T})  \ \text{ for }  \   0 \leq \frac{1}{q_1} < \frac{1+d(m-1)}{d}, \quad 0 \leq \frac{1}{q_2}\leq \frac{1+d(m-1)}{2+d(q+m-2)}.
\end{equation}
 
\begin{itemize}
\item[$\bullet$] For $q=1$, assume $\int_{\Omega} \rho_0 \log \rho_0 \,dx < \infty$.
Then the following estimate holds:
\begin{equation}\label{L1-energy}
 \sup_{0\leq t \leq T}  \int_{\Omega \times \{t\}} \rho \abs{\log \rho} (\cdot, t)\,dx
 +  \frac{2}{m}\iint_{\Omega_T} \left|\nabla \rho^{\frac{m}{2}}\right|^2 \,dx\,dt \le C ,
\end{equation}
where  $C = C (\|V\|_{\mathsf{S}_{m,1}^{(q_1, q_2)}}, \, \int_{\Omega} \rho_0 \log \rho_0 \,dx)$. 

\item[$\bullet$] For $q>1$, assume $\rho_0 \in L^q (\Omega)$.
Then the following estimate holds:
\begin{equation}\label{Lq-energy}
 \sup_{0\leq t \leq T}  \int_{\Omega \times \{t\}} \rho^q (\cdot, t) \,dx
 +  \frac{2mq(q-1)}{(q+m-1)^2}\iint_{\Omega_T} \left|\nabla \rho^{\frac{q+m-1}{2}}\right|^2 \,dx\,dt \le C ,
\end{equation}
where  $C = C (\|V\|_{\mathsf{S}_{m,q}^{(q_1, q_2)}}, \, \|\rho_{0}\|_{L^{q} (\Omega)})$.
\end{itemize}

\item[(ii)]  Let $(1-\frac{2}{d})_{+} < m <1$. Suppose that 
\begin{equation}\label{V:tilde:Lq-energy}
V\in \tilde{\mathsf{S}}_{m,q}^{(\tilde{q}_1, \tilde{q}_2)} \cap \calC^{\infty} (\overline{\Omega}_{T}) \quad  \text{for} \quad  0 \leq \frac{1}{\tilde{q}_1} < \frac{2+ d(m-1)}{d}, \quad  0 \leq \frac{1}{\tilde{q}_2} \leq 1. 	 
\end{equation}

\begin{itemize}
\item[$\bullet$] For $q=1$, assume $\int_{\Omega} \rho_0 \log \rho_0 \,dx < \infty$. 
Then \eqref{L1-energy} holds with $C = C (\|V\|_{\tilde{\mathsf{S}}_{m,1}^{(\tilde{q}_1, \tilde{q}_2)}}, \, \int_{\Omega} \rho_0 \log \rho_0 \,dx )$.

\item[$\bullet$] For $q>1$, assume $\rho_0 \in L^q (\Omega)$.
Then \eqref{Lq-energy} holds with 
$C = C (\|V\|_{\tilde{\mathsf{S}}_{m,q}^{(\tilde{q}_1, \tilde{q}_2)}}, \, \|\rho_{0}\|_{L^{q} (\Omega)} )$.
\end{itemize}

\item[(iii)] Assume $\nabla\cdot V = 0$ for $V\in\calC^{\infty} (\overline{\Omega}_{T})$ with $V\cdot \textbf{n}=0$ on $\partial \Omega$.

\begin{itemize}
\item[$\bullet$] For $q=1$, assume $\int_{\Omega} \rho_0 \log \rho_0 \,dx < \infty$. Then \eqref{L1-energy} holds with $C = C (\int_{\Omega} \rho_0 \log \rho_0 \,dx )$.

\item[$\bullet$] For $q>1$, assume $\rho_{0}\in  L^{q}(\Omega)$. Then \eqref{Lq-energy}
holds with $C = C ( \|\rho_{0}\|_{L^{q} (\Omega)} )$.
\end{itemize}
\end{itemize}
\end{proposition}

\begin{remark}
	\begin{itemize}
	    \item[(i)] For FDE, we are not able to close the energy estimate using the Gr\"{o}nwall's inequality with $L^{q,\infty}_{x,t}(\Omega_T)$-norm, see \cite[Proposition~4.2]{HKK02}. Here we use the $L^1$-mass conservation property which corresponds to the suitable class of $V$ as $\mathsf{S}_{m,q}^{(q_1,q_2)}$.
	    
	    \item[(ii)] The critical case of $V\in \mathsf{S}_{m,q}^{(q_1,q_2)}$ where $q_1 = \frac{d}{1+d(m-1)}$ and $q_2 = \infty$ , \eqref{Energy:V01} becomes 
	    \[
	    \mathcal{J}_1 \leq c(\|\rho_0\|_{L^{1}_{x}}) \left\|V\right\|_{L_{x}^{q_1}} \left\|\nabla \rho^{\frac{q+m-1}{2}}\right\|_{L_x^2}^{2}  + c	    
	    \]
	    that the Young's inequality does not work. The estimate \eqref{L1-energy} is obtained as long as assuming smallness condition on $\left\|V\right\|_{L_{x,t}^{q_{1}, q_{2}}}$.
	    Also, the critical case of $V\in \tilde{\mathsf{S}}_{m,q}^{(\tilde{q}_1,\tilde{q}_2)}$ where $\tilde{q}_1 = \frac{2}{2+d(m-1)}$ and $\tilde{q}_2 = \infty$ corresponds to 
	    \[
	    \mathcal{J}_1 \leq c(\|\rho_0\|_{L^{1}_{x}}) \left\|\nabla V\right\|_{L_{x}^{\tilde{q}_1}} \left\|\nabla \rho^{\frac{q+m-1}{2}}\right\|_{L_x^2}^{2}  + c.	    
	    \]  
	    Hence, the energy estimate is obtained as long as the smallness is assumed for $\left\|\nabla V\right\|_{L_{x,t}^{\tilde{q}_{1}, \tilde{q}_{2}}}$. 
		\end{itemize}
\end{remark}

\begin{proof}
To the equation \eqref{FDE}, it follows by testing either $\log \rho$ for $q=1$ or $q\rho^{q-1}$ for $q>1$: 
	\[
		\begin{cases}
			\frac{d}{dt}\int_{\Omega} \rho \log \rho \,dx + \frac{4}{m}\int_{\Omega} |\nabla \rho^{\frac{m}{2}}|^2 \,dx = \int_{\Omega} V \cdot \nabla \rho \,dx, &\text{ when } \ q=1, \vspace{2 mm}\\
			\frac{d}{dt}\int_{\Omega} \rho^q \,dx + \frac{4mq(q-1)}{(q+m-1)^2}\int_{\Omega} |\nabla \rho^{\frac{q+m-1}{2}}|^2 \,dx = (q-1)\int_{\Omega} V \cdot \nabla \rho^q \,dx, &\text{ when } \ q>1.
		\end{cases}
	\]
	Let 
	\[
	\mathcal{J}_{q} := \int_{\Omega} V \cdot \nabla \rho^q \,dx, \quad q\geq 1.
	\] 
   	First, assume $V \in \mathsf{S}_{m,q}^{(q_1,q_2)}$. Then 
   	\begin{equation}\label{Energy:V01}
	\begin{aligned} 
	\mathcal{J}_{q} &= \frac{2q}{q+m-1} \int_{\Omega} \rho^{\frac{q-m+1}{2}} V \cdot \nabla \rho^{\frac{q+m-1}{2}} \,dx \\
	&\leq c \left\|V\right\|_{L_{x}^{q_1}} \left\|\nabla \rho^{\frac{q+m-1}{2}}\right\|_{L_x^2} \left\|\rho\right\|_{L_{x}^{\frac{q_1 (q-m+1)}{q_1 - 2}}}^{\frac{q-m+1}{2}} \\
	&\leq c \left\|V\right\|_{L_{x}^{q_1}} \left\|\nabla \rho^{\frac{q+m-1}{2}}\right\|_{L_x^2}  \left(\left\|\nabla \rho^{\frac{q+m-1}{2}}\right\|_{L_x^2}^{\frac{(1-\theta)(q-m+1)}{q+m-1}} \left\|\rho\right\|_{L_{x}^{1}}^{\frac{\theta(q-m+1)}{2}} + \left\|\rho \right\|_{L_{x}^{1}}^{\frac{q-m+1}{2}} \right)
	\end{aligned}
	\end{equation}
	by the H\"{o}lder inequality with $\frac{1}{q_1} + \frac{1}{2}+ \frac{q_1 - 2}{2q_1}=1$ for $q_1 >2$ and by applying interpolation inequalities in Lemma~\ref{L:interpolation} with $p=1$ for $\frac{1}{r_1}= \theta + \frac{(1-\theta)(d-2)}{d(q+m-1)}$. 
	Note that $\left\|\rho \right\|_{L_{x}^1} = \left\|\rho_0 \right\|_{L_{x}^1}$ by the mass conservation property. Provided \eqref{V:Lq-energy} for $q\geq 1$ with $\frac{(1-\theta)(q-m+1)}{q+m-1} < 1$ (corresponding to $q_2 < \infty$), it follows that  
	\begin{equation*}\label{Energy:V02}
	\mathcal{J}_{q} \leq \varepsilon \left\|\nabla \rho^{\frac{q+m-1}{2}}\right\|_{L_x^2}^{2} + c(\varepsilon, \left\|\rho_0 \right\|_{L_{x}^1}) \left\|V\right\|_{L_{x}^{q_1}}^{q_2} + c(\varepsilon)
	\end{equation*}
	for any $\varepsilon>0$. Finally, by selecting $\varepsilon = \frac{q}{m+q-1}$, the estimate \eqref{L1-energy} and \eqref{Lq-energy} follow.

	Now assume the case (ii) or (iii). Because $V\cdot \textbf{n}=0$ on $\partial{\Omega}$, it follows by the integration by parts
	\[
	\mathcal{J}_{q} = - \int_{\Omega} \rho^q \nabla \cdot V \,dx.
	\]
	In case (iii), the integration parts yields that $\mathcal{J}_{q} = 0$ because $\nabla \cdot V = 0$. Therefore the energy estimates \eqref{L1-energy} and \eqref{Lq-energy} are obtained independent of $V$.

	Now in case (ii), we apply the H\"{o}lder inequality for $\frac{1}{\tilde{q}_1} + \frac{\tilde{q}_1 - 1}{\tilde{q}_1}=1$ that 
	\[\begin{aligned}
		\mathcal{J}_{q} &\leq \|\nabla V\|_{L_{x}^{\tilde{q}_1}} \|\rho\|_{L^{\frac{q \tilde{q}_1}{\tilde{q}_1 - 1}}_{x}}^{q} 
		\leq \|\nabla V\|_{L_{x}^{\tilde{q}_1}} \left( \left\| \nabla \rho^{\frac{q+m-1}{2}}\right\|_{L^2_x}^{\frac{2q(1-\theta)}{q+m-1}} \left\|\rho\right\|_{L^{1}_{x}}^{q\theta} + \left\|\rho\right\|_{L^{1}_{x}}^{q}\right)
	\end{aligned}\] 
	by the interpolation inequality for $\tilde{r}_1 = \frac{q\tilde{q}_1}{\tilde{q}_1 -1}$ and $\frac{1}{\tilde{r}_1} = \theta + \frac{(1-\theta)(d-2)}{d(q+m-1)}$.
	Provided \eqref{V:tilde:Lq-energy} for $q\geq 1$ with $\frac{q(1-\theta)}{q+m-1} < 1$ (equivalently, $\tilde{q}_2 < \infty$), the Young's inequality yields that 
	\[
	\mathcal{J}_1 \leq \varepsilon \left\| \nabla \rho^{\frac{q+m-1}{2}}\right\|_{L^2_x}^{2} + c(\varepsilon, \|\rho_0\|_{L_{x}^1}) \left\|V\right\|_{L^{\tilde{q}_1}_{x}}^{\tilde{q}_2} + c(\varepsilon) 
	\]
	for any $\varepsilon>0$. Therefore, $\varepsilon = \frac{q}{q+m-1}$ completes the estimates \eqref{L1-energy} and \eqref{Lq-energy}.
\end{proof}

\subsection{Estimates of speed}\label{SS:speed}
 
First, we remark that the equation \eqref{FDE} is also written as 
\[
\partial_t \rho + \nabla \cdot (w \rho)=0, \quad \text{ where } \quad w:= - \frac{\nabla \rho^m}{\rho} + V. 
\]
From Lemma~\ref{representation of AC curves}, the above equation can be regarded as a curve in $\mathcal{P}_{p}(\Omega)$ whose speed at time $t$ is limited by $\|w(\cdot, t)\|_{L^{p}(\rho(t))}$. In the following lemma, we carry out the estimate of $\|w\|_{L^{p}(0,T; L^p(\rho(t)))}$.

\begin{lemma}\label{L:speed}
	Let $d\geq 2$, $0<m<1$ and $q\geq 1$. Also, let $(1-\frac{2q}{d})_{+} \leq m<1$. Suppose that $\rho:\Omega_{T} \mapsto \mathbb{R}$ is a regular solution of \eqref{FDE}-\eqref{FDE_bc_ic} satisfying 
	 \begin{equation}\label{FS}
 \rho \in L^{\infty}(0, T; L^{q}(\Omega)) \quad \text{and} \quad
 \rho^{\frac{q+m-1}{2}} \in L^{2}(0, T; W^{1,2}(\Omega)).
 \end{equation}
 Furthermore, assume that 
  \begin{equation}\label{V:speed}
  \begin{gathered}
  V\in L_{x,t}^{q_1, q_2}  \ \text{ where } \
  \frac{d}{q_1} + \frac{2 + q_{m,d}}{q_2} = 1+\frac{d(q+m-2)}{2q} \vspace{1 mm}\\
\text{ for } \ 
\begin{cases}
0 \leq \frac{1}{q_1} \leq \frac{q-1}{2q}+\frac{2+q_{m,d}}{2d}, \ \frac{m+q-2}{2(m+q-1)} \leq \frac{1}{q_2} \leq \frac{1}{2}, &\text{if } \  d\geq 2  \text{ and }  1\leq q<2-m,
\vspace{1mm}\\
0 \leq \frac{1}{q_1} \leq \frac{q-1}{2q}+\frac{2+q_{m,d}}{2d(m+q-1)},\ \frac{m+q-2}{2(m+q-1)} \leq \frac{1}{q_2} \leq \frac{1}{2}, &\text{if }\  d>2 \text{ and } q \geq 2-m,  \vspace{1 mm} \\
0 \leq \frac{1}{q_1} < \frac{q-1}{2q}+\frac{2+q_{m,d}}{2d(m+q-1)}, \ \frac{m+q-2}{2(m+q-1)} < \frac{1}{q_2} \leq \frac{1}{2}, &\text{if } \  d=2 \text{ and } q \geq 2-m.  
\end{cases}
 \end{gathered}
 \end{equation}
 
 Then the following holds
 \begin{equation*}
 	\iint_{\Omega_{T}} \{ \abs{\frac{\nabla \rho^m}{\rho}}^2 \rho + \abs{V}^2\rho \} \,dxdt \leq C
\end{equation*}
where the constant $C$ depends on $\Omega$, $\rho_0$, and $V$. 
\end{lemma}

\begin{proof}
By testing $-m(\rho+\epsilon)^{m-1}$ for $\epsilon>0$ to a regular solution of \eqref{FDE} as in Definition~\ref{D:regular-sol}, then it yields 
\[
 (1-m)\int_{\Omega} \frac{\nabla \rho^{m}}{\rho} \cdot \frac{\nabla \left(\rho + \epsilon\right)^{m}}{\rho + \epsilon} \cdot \rho \, dx 
 = \frac{d}{dt}\int_{\Omega} \left(\rho+\epsilon\right)^m \,dx + 
(1-m) \int_{\Omega}  \rho V \cdot \frac{\nabla \left(\rho + \epsilon \right)^m}{\rho + \epsilon} \,dx .
\]
Now the integration in terms of $t$ yields the following
\begin{equation}\label{speed01}
\begin{aligned}
 (1-m)\iint_{\Omega_{T}} \frac{\nabla \rho^{m}}{\rho} \cdot \frac{\nabla \left(\rho + \epsilon\right)^{m}}{\rho + \epsilon} \cdot \rho \, dx dt
 &= \int_{\Omega} \left(\rho + \epsilon\right)^m (\cdot, T) \,dx - \int_{\Omega} \left(\rho_{0}+\epsilon\right)^{m} \,dx \\ 
 &\quad + 
(1-m) \iint_{\Omega_T} \rho V \cdot \frac{\nabla \left(\rho + \epsilon \right)^m}{\rho + \epsilon}  \,dxdt. 
\end{aligned}\end{equation}
Moreover, the Young's inequality provides that 
\begin{equation}\label{speed02}
\begin{aligned}
II &:= (1-m)\iint_{\Omega_T}  \rho V \cdot \frac{\nabla \left(\rho + \epsilon \right)^m}{\rho + \epsilon} \,dxdt \\
&\leq  \frac{1-m}{2}\iint_{\Omega_T} \abs{\frac{\nabla \left(\rho + \epsilon\right)^m}{\rho + \epsilon}}^2 \rho \,dxdt + 2(1-m)\iint_{\Omega_T} |V|^2 \rho \,dxdt. 
\end{aligned}\end{equation}
After combining \eqref{speed01} and \eqref{speed02}, let us take $\epsilon \to 0$ which results the following
\begin{equation*}\label{speed03}
 \frac{1-m}{2}\iint_{\Omega_{T}} \abs{\frac{\nabla \rho^{m}}{\rho}}^2  \rho \, dx dt
 \leq  \int_{\Omega} \rho^m (\cdot, T) \,dx - \int_{\Omega} \rho_{0}^{m} \,dx  + 
2(1-m) \iint_{\Omega_T} \abs{V}^2 \rho \,dxdt. 
\end{equation*}

For any $0<m<1$, the H\"{o}lder inequality for $m+(1-m)=1$ provides that 
\[\begin{aligned}
\int_{\Omega} \rho^m (\cdot, T) \,dx - \int_{\Omega} \rho_{0}^{m} \,dx 
  \leq |\Omega|^{1-m}\left\{ \|\rho(\cdot, T)\|_{L^{1}_{x}(\Omega)} + \|\rho_0\|_{L^{1}_{x}(\Omega)}\right\} \leq C(m, |\Omega|, \|\rho_0\|_{L_{x}^{1}(\Omega)})
\end{aligned}\]
because of the mass conservation property, and which leads to
\[
 \iint_{\Omega_{T}} \left|\frac{\nabla \rho^{m}}{\rho}\right|^2 \rho \, dx dt
 \leq C(m,|\Omega|, \|\rho_0\|_{L^{1}_{x}(\Omega)}) + 
4 \iint_{\Omega_T} \abs{V}^2 \rho  \,dxdt.
\]
Now, based on the hypotheses \eqref{FS}, we apply Lemma~\ref{L:interpolation} with $p=q$ that  
\begin{equation}\label{speed04}
\iint_{\Omega_{T}} |V|^2 \rho \,dxdt \leq \|V\|^{2}_{L_{x,t}^{q_1, q_2}} \|\rho\|_{L_{x,t}^{r_1, r_2}}, \quad q_1 = \frac{2r_1}{r_1 -1}, \ q_2 = \frac{2 r_2}{r_2 - 1}
\end{equation}
by taking the H\"{o}lder inequalities $\frac{r_1 -1}{r_1} + \frac{1}{r_1} = 1$ in the spatial variable and $\frac{r_2 -1}{r_2} + \frac{1}{r_2} = 1$ in the temporal variable, respectively, for $r_1$ and $r_2$ satisfying \eqref{L:r1r2} in Lemma~\ref{L:interpolation} with $p=q$. 
 \end{proof}

Now we combine the energy estimates in Proposition~\ref{P:energy} and the speed estimates in \eqref{L:speed}.

\begin{proposition}\label{P:Energy-speed} 
Let $0<m<1$, $d\geq 2$, and $q\geq 1$. Suppose that $\rho$ is a regular solution of \eqref{FDE}-\eqref{FDE_bc_ic} with $\rho_{0}\in  \mathcal{P}_2(\Omega) \cap \calC^{\alpha}(\overline{\Omega})$. 

\begin{itemize}

\item[(i)] Let $1-\frac{1}{d} < m < 1$. Suppose that $V$ satisfies \eqref{V:Lq-energy}.
\begin{itemize}
	\item[$\bullet$] For $q=1$, assume $\int_{\Omega} \rho_0 \log \rho_0 \,dx < \infty$. 
Then the following estimate holds
\begin{equation}\label{L1-energy-speed}
\sup_{0\leq t \leq T} \int_{\Omega} \rho \abs{\log \rho} (\cdot, t)  \,dx
+ \iint_{\Omega_T} \{ \abs{\nabla \rho^{\frac m2}}^2 + \left(\abs{\frac{\nabla
\rho^m}{\rho}}^{2}+|V|^{2}\right ) \rho  \}\,dx\,dt
 \leq  C,
\end{equation}
where $C=C (\|V\|_{\mathsf{S}_{m,1}^{(q_1, q_2)}}, \, \int_{\Omega} \rho_0 \log \rho_0 \,dx )$.

\item[$\bullet$] For $q>1$, assume $\rho_0 \in L^{q}(\Omega)$.
Then the following estimate holds%
 \begin{equation}\label{Lq-energy-speed}
 \sup_{0\leq t \leq T}  \int_{\Omega} \rho^q(\cdot, t) \,dx
 + \iint_{\Omega_T} \{ \abs{\nabla \rho^{\frac{q+m-1}{2}}}^2  + \left(\abs{\frac{\nabla
\rho^m}{\rho}}^{2}+|V|^{2}\right ) \rho  \}\,dx\,dt
\leq C,
\end{equation}
where $C=C (\|V\|_{\mathsf{S}_{m,q}^{(q_1, q_2)}}, \, \|\rho_0 \|_{L^{q}(\Omega)})$.
\end{itemize}

\item[(ii)] Let $ 1-\frac{1}{d} < m<1$. Suppose that 
\begin{equation}\label{Vq:tilde:energy-speed}
V\in \tilde{\mathsf{S}}_{m,q}^{(q_1, q_2)}\cap L_{x,t}^{1,\tilde{q}_2} \cap \calC^\infty(\overline{\Omega}_T) \ \text{ for } \
\frac{1}{d} < \frac{1}{\tilde{q}_1} < \frac{2 + d(m-1)}{d}, \quad 0 \leq \frac{1}{\tilde{q}_2} < \frac{1+d(m-1)}{2+d(q+m-2)}. 
\end{equation}

\begin{itemize}
	\item[$\bullet$] For $q=1$, assume $\int_{\Omega} \rho_0 \log \rho_0 \,dx < \infty$. 
Then \eqref{L1-energy-speed} holds with $ C =C  (\|V\|_{\tilde{\mathsf{S}}_{m,1}^{(\tilde{q}_1, \tilde{q}_2)}}, \, \int_{\Omega} \rho_0 \log \rho_0 \,dx ).$

\item[$\bullet$] For $q>1$, assume $\rho_0 \in L^{q}(\Omega)$. Then \eqref{Lq-energy-speed} holds with $ C =C (\|V\|_{\tilde{\mathsf{S}}_{m,q}^{(\tilde{q}_1, \tilde{q}_2)}}, \, \|\rho_0 \|_{L^{q}(\Omega)}).$
\end{itemize}

\item[(iii)] Let $ (1-\frac{2q}{d})_{+} \leq m<1$.    Suppose that 
 \begin{equation}\label{V:divfree:energy-speed}
  V\in \mathcal{D}_{m,q,s}^{(q_1, q_2)}\cap \calC^\infty(\overline{\Omega}_T)  \text{ for } 
\begin{cases}
0 \leq \frac{1}{q_1} \leq \frac{q-1}{2q} + \frac{2+q_{m,d}}{2d}, \ 0 \leq \frac{1}{q_2} \leq \frac{1}{2},  &  \text{if }  d\geq 2 \text{ and } 1 \leq q < 2-m, \vspace{1mm}\\
	0 \leq \frac{1}{q_1} \leq \frac{q-1}{2q} + \frac{2+q_{m,d}}{2d(m+q-1)}, \ 0 \leq \frac{1}{q_2} \leq \frac{1}{2}, &\text{if } d>2 \text{ and } q\geq 2-m ,  \vspace{1 mm} \\
	0 \leq \frac{1}{q_1} < \frac{q-1}{2q} + \frac{2+q_{m,d}}{2d(m+q-1)}, \ 0 \leq \frac{1}{q_2} \leq \frac{1}{2}, &\text{if }  d=2 \text{ and } q\geq 2-m.
\end{cases}
 \end{equation}

\begin{itemize}
\item[$\bullet$] For $q=1$, assume $\int_{\Omega} \rho_0 \log \rho_0 \,dx < \infty$. Then \eqref{L1-energy-speed} holds with $ C =C ( \|V\|_{\mathcal{D}_{m,1,s}^{(q_1,q_2)}}, \, \int_{\Omega} \rho_0 \log \rho_0 \,dx).$ 

\item[$\bullet$] For $q>1$, assume $\rho_0 \in L^{q}(\Omega)$. Then \eqref{Lq-energy-speed} holds
	 	holds with $ C =C (\|V\|_{\mathcal{D}_{m,q,s}^{(q_1,q_2)}}, \, \|\rho_0\|_{L^q(\Omega)} ).$
\end{itemize}

\end{itemize}
\end{proposition}

\begin{proof}
	Here we combine the energy estimates in Proposition~\ref{P:energy} and the speed estimates in Lemma~\ref{L:speed}. Hence the corresponding structures of $V$ is the intersection of conditions from two estimates. 
	In case (i), \eqref{V:Lq-energy} and \eqref{V:Lq-energy} coincide because \eqref{V:Lq-energy} is more restrictive than \eqref{V:speed}. In case (ii), \eqref{Vq:tilde:energy-speed} represents the intersection of \eqref{V:tilde:Lq-energy} and the embedding region where \eqref{V_embedding} holds. In case (iii), since the energy estimate is obtained independently of $V$, the conditions from the speed estimate \eqref{V:speed} determine \eqref{V:divfree:energy-speed}.
\end{proof}

\subsection{Estimates of temporal and spatial derivatives} 

In this section, we study further properties of temporal and spatial derivatives of $\rho$, a weak solution of \eqref{FDE}-\eqref{FDE_bc_ic} with the initial data in $L^q$ for $q\geq 1$ based on the functional spaces from the energy estimates in Proposition~\ref{P:energy}. 

The following proposition is key for ensuring the convergence of approximate solutions to weak solutions in appropriate functional classes.
Moreover, it provides a comprehensive version of earlier compactness results by the same authors (see \cite[Section~4.4]{HKK01} and \cite[Section~4.3]{HKK02}). In the case \(\nabla\cdot V=0\), the condition on \(V\) in \eqref{V:divfree:compact} is precisely the condition under which we construct \(L^q\)-weak solutions of \eqref{FDE}--\eqref{FDE_bc_ic}. Furthermore, the same extended compactness argument also applies to the PME, and the corresponding consequences are explained in Appendix~\ref{Appendix:PME}.

\begin{proposition}\label{Proposition : compactness}
	Suppose that $\{\rho_n\}$ is a sequence of regular solutions of \eqref{FDE}-\eqref{FDE_bc_ic}  with a corresponding sequence of vector fields $\{ V_n \}$ under the same hypotheses as in Proposition~\ref{P:energy}.
	\begin{itemize}
		\item [(i)] Let $1-\frac{1}{d}<m<1$. For $q\geq 1$, suppose that $V_n$ satisfies \eqref{V:Lq-energy} and $\|V_n\|_{\mathsf{S}_{m,q}^{(q_1, q_2)}}$ is bounded. Then the following estimate holds: 
		\begin{equation}\label{E:compact}
 \sup_{0\leq t \leq T}  \int_{\Omega \times \{t\}} \rho_n^q (\cdot, t) \,dx
 +  \iint_{\Omega_T} \left|\nabla \rho_n^{\frac{q+m-1}{2}}\right|^2 \,dx\,dt < \infty.
\end{equation}
Furthermore, $\{\rho_n\}$ is relatively compact in the following space:
\begin{equation}\label{Space:compact}
\begin{gathered}
L^{\alpha, \beta}_{x,t} \quad \text{ where } \ \alpha < r_1 \ \text{ and } \ \beta <r_2, \\
 \frac{d}{r_1} + \frac{2+q_{m,d}}{r_2} = \frac{d}{q}
\ \text{ for } \
\begin{cases}
  q \leq r_1 \leq \frac{d(m+q-1)}{d-2}, \ m+q-1 \leq r_2 \leq \infty, & \text{if } \ d > 2, \\
  q \leq r_1 < \infty, \ m+q-1 < r_2 \leq \infty, &\text{if }\ d = 2.
  \end{cases}
\end{gathered}
\end{equation}

\item[(ii)] Let $1-\frac{1}{d}<m<1$. For $q\geq 1$, suppose that $V_n$ satisfies \eqref{Vq:tilde:energy-speed} and $\|V_n\|_{\tilde{\mathsf{S}}_{m,q}^{(\tilde{q}_1, \tilde{q}_2)}\cap L_{x,t}^{1,\tilde{q}_2}}$ is bounded. 
Then \eqref{E:compact} holds, and that $\{\rho_n\}$ is relatively compact in \eqref{Space:compact}.
\vskip 1 mm

\item[(iii)]  Let $q\geq 1$ and $(1-\frac{2q}{d})_{+} \leq m <1$. Assume 
    \begin{equation}\label{V:divfree:compact}
  V_n\in \mathcal{D}_{m,q}^{(q_1, q_2)}\cap \calC^{\infty} (\overline{\Omega}_{T})  \text{ for } 
\begin{cases}
0 \leq \frac{1}{q_1} \leq \frac{q-1}{q}+\frac{2+q_{m,d}}{d}, \ 0 \leq \frac{1}{q_2} \leq 1, &  \text{if }  d\geq 2  \text{ and }  1\leq q<2-m,
\vspace{1mm}\\
	0 \leq \frac{1}{q_1} \leq \frac{q-1}{q}+\frac{2+q_{m,d}}{d(m+q-1)}, \ 0 \leq \frac{1}{q_2} \leq 1, &\text{if } d>2 \text{ and } q \geq 2-m,  \vspace{1 mm} \\
	 0 \leq \frac{1}{q_1} < \frac{q-1}{q}+\frac{2+q_{m,d}}{d(m+q-1)}, \ 0 \leq  \frac{1}{q_2} \leq 1, &\text{if }  d=2 \text{ and } q \geq 2-m, 
\end{cases}
\end{equation}
and  $\|V_n\|_{\mathcal{D}_{m,q}^{(q_1, q_2)}}$  is bounded. Then \eqref{E:compact} holds, and that $\{\rho_n\}$ is relatively compact in \eqref{Space:compact}.
	\end{itemize}
\end{proposition}

\begin{remark}
The assumptions on $V$ are determined so that both the energy estimates in Proposition~\ref{P:energy} and the integrability condition \eqref{Vrho:L1}, needed for the compactness argument, hold. In cases (i) and (ii), the assumptions from Proposition~\ref{P:energy} are sufficient. In the divergence-free case (iii), the additional condition \eqref{Vrho:L1} is needed because the energy estimate is obtained independently of $V$.
\end{remark}

\begin{proof}
For simplicity, we write $\rho$ in place of $\rho_n$.
Under the same hypotheses as in Proposition~\ref{P:energy}, we have that 
\begin{equation}\label{FS001}
 \rho \in L^{\infty}(0, T; L^{q}(\Omega)) \quad \text{and} \quad
 \rho^{\frac{q+m-1}{2}} \in L^{2}(0, T; W^{1,2}(\Omega)).
\end{equation}
Then Lemma~\ref{T:pSobolev} yields 
\begin{equation}\label{FS_parabolic}
\rho \in L^{m+q-1 + \frac{2q}{d}}_{x,t}(\Omega_T).
\end{equation}

\medskip 
\noindent

\underline{\emph{Case 1: $1-\frac{1}{d}<m<1$.}} First, we prove that 
\begin{equation}\label{nabla_rho}
\nabla\rho \in L^{\gamma}_{x,t}(\Omega_T), \quad \text{where} \quad  
\gamma = 
\begin{cases}
    2, & \text{if } q \geq 3-m, \vspace{1mm}\\
	\frac{q(d+2) + d(m-1)}{d+q},  & \text{if } 1\leq q <3-m.
\end{cases}
\end{equation}
Note that $\gamma > 1$ for $1-\frac{1}{d}<m<1$. 

In case $1\leq q < 3-m$, it follows that, by H\"{o}lder inequality,  
\[\begin{aligned}
 \iint_{\Omega_T} |\nabla \rho|^{\gamma} \,dxdt 
 & = \iint_{\Omega_{T}} \rho^{\gamma \left( 1 - \frac{m+q-1}{2}\right)} \left(\rho^{\frac{m+q-1}{2} -1} \nabla \rho \right)^{\gamma}   \,dxdt \\
    &= C \left(\iint_{\Omega_{T}} \rho^{m+q-1+\frac{2q}{d}}   \,dxdt \right)^{\frac{2-\gamma}{2}}	
     \left(\iint_{\Omega_{T}}  \left|\nabla \rho^{\frac{m+q-1}{2}} \right|^{2}\,dxdt \right)^{\frac{\gamma}{2}}.
\end{aligned}\]

In case $q=3-m$, it follows that $\nabla \rho \in L^{2}_{x,t}(\Omega_T)$ directly from a priori estimates in Proposition~\ref{P:energy}. 
 
For $q>3-m$, since $\Omega$ is bounded, $\rho_0\in L^q(\Omega)\cap L^1(\Omega)$ implies $\rho_0\in L^{3-m}(\Omega)$. Moreover, the admissible condition on $V$ for $q$ implies the corresponding condition for $q=3-m$. Therefore, the estimates in Proposition~\ref{P:energy} with $q=3-m$ hold.
   
Second, we prove that 
\begin{equation}\label{rho_t_space}
\rho_t \in W^{-2, \frac{d+1}{d}}_{x} L^{1}_{t}. 
\end{equation}
Because $\rho_0 \in \left(L^{1}\cap L^{q}\right)(\Omega)$,   
it follows by \eqref{FS_parabolic} with $q=1$ that $\rho^m \in L_{x,t}^{\frac{2+md}{md}} (\Omega_{T})$. Note that $\frac{2+md}{md} > \frac{d+1}{d}$ for all $m\in (0,2)$. Therefore, it holds that 
\begin{equation}\label{Delta_rho}
\Delta \rho^m \in W^{-2, \frac{d+1}{d}}_{x} L^{\frac{d+1}{d}}_{t},
\end{equation}
where $W^{-2, \frac{d+1}{d}} (\Omega)$ is the dual of $W^{2, d+1}_{o}(\Omega)$.

Moreover, by applying the H\"{o}lder inequalities, it is clear that 
\begin{equation}\label{L1-space-time}
\|V\rho\|_{L^{1}_{x,t}} \leq \|V\|_{L_{x,t}^{q_1,q_2}}\|\rho\|_{L_{x,t}^{r_1,r_2}}, \ \text{ where }  r_1 = \frac{q_1}{q_1 - 1}, \  r_2 = \frac{q_2}{q_2 - 1}. 
\end{equation}
Then Lemma~\ref{L:interpolation} with $p=q$ corresponds to pairs $(q_1, q_2)$ such that 
\begin{equation}\label{Vrho:L1}
\begin{gathered}
\frac{d}{q_1} + \frac{2+q_{m,d}}{q_2} = \frac{d(q-1)}{q} + 2+ q_{m,d} \\ \ \text{ for } \
\begin{cases}
0 \leq \frac{1}{q_1} \leq \frac{q-1}{q}+\frac{2+q_{m,d}}{d}, \ \frac{m+q-2}{m+q-1} \leq \frac{1}{q_2} \leq 1, &  \text{if }  d\geq 2  \text{ and }  1\leq q<2-m,
\vspace{1mm}\\
	0 \leq \frac{1}{q_1} \leq \frac{q-1}{q}+\frac{2+q_{m,d}}{d(m+q-1)}, \ \frac{m+q-2}{m+q-1} \leq \frac{1}{q_2} \leq 1, &\text{if } d>2 \text{ and } q \geq 2-m,  \vspace{1 mm} \\
	 0 \leq \frac{1}{q_1} < \frac{q-1}{q}+\frac{2+q_{m,d}}{d(m+q-1)}, \ \frac{m+q-2}{m+q-1} < \frac{1}{q_2} \leq 1, &\text{if }  d=2 \text{ and } q \geq 2-m.  
	 \end{cases}
\end{gathered}\end{equation}
Because $V\rho \in L^{1}_{x,t} (\Omega_{T})$, now 
it leads that $\nabla \cdot (V\rho) \in W^{-2, \frac{d+1}{d}}_{x}L_{t}^{1}$ by the following:
\[
\int_{\Omega} \nabla \cdot (V\rho) \, \varphi \,dx = - \int_{\Omega} (V\rho) \, \nabla \varphi \,dx \leq \|V \rho\|_{L^1_{x}(\Omega)} \|\nabla \varphi\|_{L^{\infty}_{x}(\Omega)}< \infty
\]
for any $\varphi \in W_{0}^{2,d+1}(\Omega)$ because $\nabla \varphi \in L^{\infty}_{x}(\Omega)$ by the Sobolev embedding. 
Therefore, \eqref{rho_t_space} follows by using \eqref{FDE} with \eqref{Delta_rho} and \eqref{L1-space-time}.

Note that 
\[
W^{1,\gamma}(\Omega) \hookrightarrow \hookrightarrow L^{\sigma}(\Omega) \hookrightarrow W^{-2, \frac{d+1}{d}}(\Omega) , \quad 
\text{for} \quad 
\begin{cases}
	1<\sigma < \frac{d\gamma}{d-\gamma} , \text{ if } \gamma<d, \vspace{1mm}\\
	1\leq \sigma < \infty, \text{ if } \gamma \geq d. 
\end{cases}
\]
We obtain strong convergence in $L^{\sigma, \gamma}_{x,t}(\Omega_T)$ by applying the Aubin-Lions Lemma~\ref{AL} with $X_0 = W^{1,\gamma}(\Omega)$, $X=L^{\sigma}(\Omega)$ and $X_1= W^{-2, \frac{d+1}{d}}(\Omega)$.
Finally, we complete the proof of relative compactness in \eqref{Space:compact} by applying the uniform boundedness of \eqref{E:compact}. 

\medskip 
\noindent
\underline{\emph{Case 2: $(1-\frac{2q}{d})_{+} \leq m \leq 1-\frac{1}{d}$.}}
This case concerns only the divergence-free case for $0<m<1$. If $q\geq 3-m$, then $\nabla\rho\in L^2_{x,t}(\Omega_T)$, as in Case 1, and the same argument applies.

If $1\le q<3-m$, we use a truncation argument inspired by \cites{BDG15, HKKP}, which also covers the range where the preceding Aubin--Lions argument cannot be applied directly. For $k\in\mathbb N$, set
\[
w_n^{(k)}:=T_k(\rho_n):=\min\{\rho_n,k\}.
\]
In the divergence-free case, these truncations satisfy the same a priori estimates as in Proposition~\ref{P:energy}.

First, we observe that 
\[
|\nabla w_n^{(k)}|
=
\frac{2}{m+q-1}\left(w_n^{(k)}\right)^{\frac{3-m-q}{2}}
\left|\nabla \left(w_n^{(k)}\right)^{\frac{m+q-1}{2}}\right|
\le
C(k,m,q)\left|\nabla \rho_n^{\frac{m+q-1}{2}}\right|,
\]
because $3-m-q>0$ and $w_n^{(k)}\le k$. Hence, for every $U\Subset\Omega$ and every $0<t_1<t_2<T$,
\begin{equation}\label{divfree-trunc-grad}
\sup_n \|w_n^{(k)}\|_{L^{2}(t_1,t_2;W^{1,2}(U))}\le C(k,U,t_1,t_2).
\end{equation}
Using \eqref{E:compact} and adapting Lemmas~3.1--3.4 and Section~4.3.1--4.3.2 in \cite{BDG15} to $w_n^{(k)}$, with no measure-data term present here, we obtain 
\begin{equation}\label{divfree-trunc-time}
\sup_n
\|\partial_t w_n^{(k)}\|_{L^1(t_1,t_2;W^{-2,\frac{d+1}{d}}(U))}
\le
C(k,U,W,t_1,t_2, \|V\|_{L^{1}(\Omega_T)})
\end{equation}
for every $U\Subset W\Subset\Omega$ and every $0<t_1<t_2<T$. 
Consequently, for every fixed $k\in\mathbb N$, Lemma~\ref{AL} implies that, after passing to a subsequence,
\[
T_k(\rho_n)\to T_k(\rho)
\quad \text{strongly in }L^{\sigma,2}_{\rm loc}(\Omega_T), \quad \text{for}\quad
\begin{cases}
	1<\sigma < \frac{2d}{d-2} , \text{ if } 2<d, \vspace{1mm}\\
	1\leq \sigma < \infty, \text{ if } 2 \geq d. 
\end{cases}
\]
Repeating the diagonal construction exactly as in Section~4.3.2 of \cite{BDG15}, we conclude that
\[
\rho_n\to \rho
\quad \text{strongly in }L^{\sigma,2}_{\rm loc}(\Omega_T)
\]
and almost everywhere in $\Omega_T$. With the uniform bounds in \eqref{E:compact}, we complete the proof. 
\end{proof}

\section{Existence of regular solutions }\label{splitting method}

In this section, we construct a regular solution of \eqref{FDE}-\eqref{FDE_bc_ic} when the initial data $\rho_0$ and the vector field $V$ are smooth enough.
For this, we exploit a splitting method in the Wasserstein space $\mathcal{P}_2(\Omega)$, and it turns out that our solution is in the class of absolutely continuous curves in $\mathcal{P}_2(\Omega)$.
For carrying the splitting method, it requires a priori estimates, propagation of compact support, and H\"{o}lder continuity of the following form of homogeneous FDE:
\begin{equation}\label{H-FDE}
  \left\{
  \begin{array}{ll}
  \partial_t \varrho =   \Delta (\epsilon +\varrho)^m,  \quad \text{ in } \ \Omega_{T} := \Omega \times (0, T),  \vspace{2mm}\\
  \frac{\partial \varrho}{\partial \textbf{n}} = 0,  \text{ on } \partial \Omega \times (0, T), 
  \quad \text{and} \quad 
  \varrho(\cdot,0)=\varrho_0,  \text{ on } \ \Omega,
  \end{array}
  \right.
\end{equation}
for $\Omega \Subset \bbr^d$, $d\geq 2$, and $0 <T < \infty$, where $\varrho:\Omega_T\mapsto \bbr$ and the normal $\textbf{n}$ to $\partial \Omega$.

First, we deliver a priori estimates of \eqref{H-FDE} in the following lemma.
\begin{lemma}\label{P:H-PME:e}\cite[Lemma~5.1]{HKK01}
Suppose that $\varrho$ is a nonnegative $L^q$-weak solution of \eqref{H-FDE} in Definition~\ref{D:weak-sol}.
\begin{itemize}
  \item[(i)] Let $\int_{\Omega} \varrho_0 (1+\log \varrho_0 ) \,dx < \infty$. Then there exists a positive constant $c=c(m,d,|\Omega|)$ such that
      \begin{equation}\label{P:H-PME:e1}
      \begin{aligned}
      &\int_{\Omega} (\epsilon+\varrho(\cdot, T)) \log  (\epsilon+\varrho(\cdot, T))  \,dx + \iint_{\Omega_{T}} \left|\frac{\nabla (\epsilon + \varrho)^m}{\epsilon + \varrho}\right|^2 (\epsilon + \varrho) \,dxdt \\
      & \qquad \leq \int_{\Omega} (\epsilon+\varrho_0) \log  (\epsilon+\varrho_0) \,dx + c\int_{\Omega} (\epsilon + \varrho_0) \,dx.
        \end{aligned} \end{equation}
      Moreover, there exists $T^\ast = T^\ast (m,d,p)$ such that, for any positive integer $l=1,2,3,\cdots$,
        \begin{equation*}\label{C:H-PME:e1}\begin{aligned}
      &\int_{\Omega}  (\epsilon+\varrho (\cdot, lT^\ast)) \log (\epsilon + \varrho(\cdot, l T^\ast))  \,dx
      + \int_{(l-1)T^\ast}^{lT^\ast}\int_{\Omega} \abs{\frac{\nabla (\epsilon +\varrho)^m}{\epsilon +\varrho}}^ {2} \varrho \,dx\,dt \\
      &\qquad \leq \int_{\Omega} \left[(\epsilon+\varrho_0) \log  (\epsilon+\varrho_0) + (\epsilon + \varrho_0)\right] \,dx.
      \end{aligned}\end{equation*}
  \item[(ii)]Let $\varrho_0 \in L^{q}(\Omega)$ for $q>1$. Then, we have
  \begin{equation*}\label{P:H-PME:e2}
      \int_{\Omega} (\epsilon+\varrho)^q(\cdot, T) \,dx + \frac{4mq(q-1)}{(m+q-1)^2} \iint_{\Omega_T} \left| \nabla (\epsilon+\varrho)^{\frac{q+m-1}{2}} \right|^2  \,dx\,dt = \int_{\Omega} (\epsilon+\varrho_0)^q \,dx.
          \end{equation*}
\end{itemize}
\end{lemma}

\begin{proof}
	Due to the equation \eqref{H-FDE}, we replace $\varrho$ by $\epsilon + \varrho$ in the similar estimates from \cite[Lemma~5.1]{HKK01}. We point out that \eqref{P:H-PME:e1} is a combination of the energy estimate of \eqref{H-FDE} by testing $\log (\epsilon + \varrho)$ and the following estimate
	\[
       \iint_{\Omega_{T}} \left|\frac{\nabla (\epsilon + \varrho)^m}{\epsilon + \varrho}\right|^2 (\epsilon + \varrho) \,dxdt \leq c\int_{\Omega} (\epsilon + \varrho_0) \,dx
       \]  
       which is obtained by testing $-m(\epsilon + \varrho)^{m-1}$ to \eqref{H-FDE} as in the proof of Lemma~\ref{L:speed}.
\end{proof}

The main result in this section is stated as follows, with its proof provided in Appendix~\ref{A:Exist-regular}.
\begin{proposition}\label{proposition : regular existence}
Let $d\geq 2$ and $q\geq 1$. Suppose that 
$$
 {\rho}_0\in\mathcal{P}_2(\Omega)\cap \mathcal{C}^\infty(\overline{\Omega}) \quad \mbox{and} \quad
V \in  \calC^\infty (\overline{\Omega}_{T}),
$$
such that $V\cdot\mathbf{n}=0$ on $\partial \Omega$. 

\begin{itemize}
	\item [(i)] Let $1-\frac{1}{d} < m < 1$. Suppose that
	\begin{equation}\label{V:Lq-regular}
V \in \mathsf{S}_{m,q}^{(q_1,q_2)}  \ \text{ for }  \   0 \leq \frac{1}{q_1} < \frac{1+d(m-1)}{d}, \quad 0<\frac{1}{q_2}\leq \frac{1+d(m-1)}{2+d(q+m-2)}.
\end{equation} 
      
      \begin{itemize}
      \item[$\bullet$] For $q=1$, assume $\int_{\Omega} \rho_0 \log \rho_0 \,dx < \infty$. Then there exists an absolutely continuous curve ${\rho}\in
AC(0,T;\mathcal{P}_2(\Omega))$ which is a regular solution of \eqref{FDE}-\eqref{FDE_bc_ic} such that	
      \begin{equation}\label{regular:L1-energy}
 \sup_{0\leq t \leq T}  \int_{\Omega \times \{t\}} \rho \abs{\log \rho} (\cdot, t)\,dx
 +  \frac{2}{m}\iint_{\Omega_T} \left|\nabla \rho^{\frac{m}{2}}\right|^2 \,dx\,dt \le C ,
\end{equation}
and 
\begin{equation}\label{regular:narrow-distance}
\delta(\rho(s),\rho(t))\leq C(t-s)^{\frac{1}{2}}, \quad \forall \ 0\leq s < t \leq T
\end{equation}
with $C=C (\|V\|_{\mathsf{S}_{m,1}^{(q_1, q_2)}}, \, \int_{\Omega} \rho_0 \log \rho_0 \,dx )$.

\item[$\bullet$] For $q>1$, assume that $\rho_0 \in L^q (\Omega)$. Then there exists an absolutely continuous curve ${\rho}\in
AC(0,T;\mathcal{P}_2(\Omega))$ which is a regular solution of \eqref{FDE}-\eqref{FDE_bc_ic} such that	
 \begin{equation}\label{regular:Lq-energy}
 \sup_{0\leq t \leq T}  \int_{\Omega \times \{t\}} \rho^q (\cdot, t) \,dx
 +  \frac{2mq(q-1)}{(q+m-1)^2}\iint_{\Omega_T} \left|\nabla \rho^{\frac{q+m-1}{2}}\right|^2 \,dx\,dt \le C ,
\end{equation}
and \eqref{regular:narrow-distance} with $C = C (\|V\|_{\mathsf{S}_{m,q}^{(q_1, q_2)}}, \, \|\rho_{0}\|_{L^{q} (\Omega)})$.

      \end{itemize}
	
	\item [(ii)] Let  $1-\frac{1}{d} < m < 1$. Suppose that
	\begin{equation}\label{V:tilde:Lq-regular}
V\in \tilde{\mathsf{S}}_{m,q}^{(\tilde{q}_1, \tilde{q}_2)} \quad  \text{for} \quad  \frac{1}{d} < \frac{1}{\tilde{q}_1} < \frac{2+ d(m-1)}{d}, \quad  0 \leq \frac{1}{\tilde{q}_2} < \frac{1+d(m-1)}{2+d(q+m-2)}. 	 
\end{equation} 
    \begin{itemize}
      \item[$\bullet$] For $q=1$, assume $\int_{\Omega} \rho_0 \log \rho_0 \,dx < \infty$. Then there exists an absolutely continuous curve ${\rho}\in
AC(0,T;\mathcal{P}_2(\Omega))$ which is a regular solution of \eqref{FDE}-\eqref{FDE_bc_ic} and satisfies \eqref{regular:L1-energy} 	
      and \eqref{regular:narrow-distance} with $C=C (\|V\|_{\tilde{\mathsf{S}}_{m,1}^{(q_1, q_2)}}, \, \int_{\Omega} \rho_0 \log \rho_0 \,dx )$.

\item[$\bullet$] For $q>1$, assume that $\rho_0 \in L^q (\Omega)$. Then there exists an absolutely continuous curve ${\rho}\in
AC(0,T;\mathcal{P}_2(\Omega))$ which is a regular solution of \eqref{FDE}-\eqref{FDE_bc_ic} and satisfies \eqref{regular:Lq-energy} 	
      and \eqref{regular:narrow-distance} with $C=C (\|V\|_{\tilde{\mathsf{S}}_{m,q}^{(q_1, q_2)}}, \, \, \|\rho_{0}\|_{L^{q} (\Omega)} )$.
    \end{itemize}

	\item [(iii)] Let $0<m<1$. Suppose that $\nabla \cdot V =0$. 
 \begin{itemize}
      \item[$\bullet$] For $q=1$, assume $\int_{\Omega} \rho_0 \log \rho_0 \,dx < \infty$. Then there exists an absolutely continuous curve ${\rho}\in
AC(0,T;\mathcal{P}_2(\Omega))$ which is a regular solution of \eqref{FDE}-\eqref{FDE_bc_ic} and satisfies \eqref{regular:L1-energy} with  	$C=C(\int_{\Omega} \rho_0 \log \rho_0 \,dx )$.

\item[$\bullet$] For $q>1$, assume that $\rho_0 \in L^q (\Omega)$. Then there exists an absolutely continuous curve ${\rho}\in
AC(0,T;\mathcal{P}_2(\Omega))$ which is a regular solution of \eqref{FDE}-\eqref{FDE_bc_ic} and satisfies \eqref{regular:Lq-energy} 	
     with $C=C (\|\rho_{0}\|_{L^{q} (\Omega)} )$.

\item[$\bullet$] Let $(1-\frac{2q}{d})_{+} \leq m <1$. Furthermore, suppose that 
\begin{equation*}\label{V:divfree:regular:delta}
\begin{gathered}
	V\in \mathcal{D}_{m,q,+}^{(q_1,q_2)} 
	\ \text{ for } \  0 \leq \frac{1}{q_1} < \frac{q-1}{q} + \frac{2+q_{m,d}}{d}, \quad  0 \leq \frac{1}{q_2} < 1.
\end{gathered}
\end{equation*}
Then, $\rho$  satisfies
    \begin{equation}\label{regular:narrow-distance:2}
	\delta(\rho(s),\rho(t))\leq C(t-s)^{a},  \qquad \forall~ 0\leq s<t\leq T
	\end{equation}
   with $a= \min \left\{ \frac{1}{2}, \frac{2+\frac{d(q+m-2)}{q} - \left(\frac{d}{q_1} + \frac{2+q_{m,d}}{q_2}\right)}{2+q_{m,d}} \right\}$ 
and $C= \begin{cases}
	C ( \|V\|_{\mathcal{D}_{m,1,+}^{(q_1,q_2)}}, \int_{\Omega} \rho_0 \log \rho_0 \,dx ), & \text{if } q=1 , \\
	C (\|V\|_{\mathcal{D}_{m,q,+}^{(q_1,q_2)}}, \|\rho_{0}\|_{L^{q} (\Omega)} ), &\text{if } q>1.
\end{cases}$
    \end{itemize}

\end{itemize}

\end{proposition}

\section{Existence of weak solutions}\label{Exist-weak}
In this section, we show the existence of $L^q$-weak solutions in Definition~\ref{D:weak-sol}. 
 Since the critical case matters most and the strict-sub-critical case is simpler, we only give proof for the critical cases.
 
\subsection{Proof of Theorem \ref{T:ACweakSol}}
{\it Proof of (i).} Let $\rho_0\in  \mathcal{P}_2(\Omega) $ with $\int_{\Omega}\rho_0 \log \rho_0 \,dx < \infty$ and 
$V \in  \mathsf{S}_{m,1}^{(q_1,q_2)} $ be satisfying \eqref{T:ACweakSol:V_q}.
Then there exists a sequence of vector fields $V_n\in C^\infty_c(\overline{\Omega}\times[0,T)) $ with $V_n(t)\cdot \mathbf{n}=0$ on $\partial\Omega$ for all $t\in[0,T)$
such that
\begin{equation}\label{eq7 : Theorem-1}
\lim_{n \rightarrow \infty}\|V_n - V \|_{\mathsf{S}_{m,1}^{(q_1,q_2)} } =0.
\end{equation}
Using truncation, mollification and normalization, we have a sequence of functions
$\rho_{0,n}\in \calC^\infty(\overline{\Omega})\cap \mathcal{P}_2(\Omega)$ satisfying
\begin{equation}\label{eq9 : Theorem-1}
\lim_{n \rightarrow \infty} W_2(\rho_{0,n} , \rho_0) =0,
\quad \text{and} \quad  \lim_{n \rightarrow \infty} \int_{\Omega} \rho_{0,n} \log \rho_{0,n} \,dx = \int_{\Omega} \rho_0\log \rho_{0} \,dx.
\end{equation}
Exploiting Proposition \ref{proposition : regular existence}, we have $\rho_n\in AC(0,T; \mathcal{P}_2(\Omega))$ which is a regular solution of
\begin{equation}\label{eq5 : Theorem-1}
\begin{cases}
   \partial_t \rho_n =   \nabla \cdot( \nabla (\rho_n)^m  - V_n\rho_n),\\
   \rho_n(0)=\rho_{0,n}.
 \end{cases}
\end{equation}
Thus, $\rho_n$ satisfies the energy estimate \eqref{regular:L1-energy} and 
\begin{equation}\label{eq21 : Theorem-1}
\begin{aligned}
\iint_{\Omega_T}  \left [\partial_t\varphi + (V_n \cdot \nabla \varphi )\right ]\rho_n
 - \nabla(\rho_n)^m \cdot \nabla \varphi  \,dx\,dt
 =  - \int_{\Omega}\varphi(0) \rho_{0,n}\,dx,
\end{aligned}
\end{equation}
for any $\varphi\in C_c^\infty( \overline{\Omega} \times [0,T))$. Furthermore, due to Proposition \ref{P:Energy-speed}, $\rho_n$  satisfies not only \eqref{regular:L1-energy}  but also 
\begin{equation}\label{eq11 : Theorem-2-a}
\sup_{0\leq t \leq T}\int_{\Omega \times \{ t\}}\rho_n |\log \rho_n|  \,dx
+ \iint_{\Omega_T} \{ \left|\nabla \rho_{n}^{\frac{m}{2}} \right|^2 + \left (\left | \frac{\nabla \rho_n^m}{\rho_n}\right |^{2} +|V_n|^{2} \right )\rho_n \}\,dx\,dt  \leq C,
\end{equation}
 where the constant $C=C ( \|V_n\|_{\mathsf{S}_{m, 1}^{(q_1,q_2)}}, \int_{\Omega} \rho_{0,n} \log \rho_{0,n} \,dx )$.
 Actually, due to \eqref{eq7 : Theorem-1} and \eqref{eq9 : Theorem-1}, we may choose the constant
$C=C ( \|V\|_{\mathsf{S}_{m, 1}^{(q_1,q_2)}}, \int_{\Omega} \rho_{0} \log \rho_{0} \,dx )$. 

Now, we note that \eqref{eq5 : Theorem-1} can be written as follows
$$\partial_t \rho_n +  \nabla \cdot (w_n \rho_n )=0, \qquad \mbox{where}\quad
w_n:= -\frac{\nabla (\rho_n)^m}{\rho_n}+ V_n.$$
Then, \eqref{eq11 : Theorem-2-a} implies
\begin{equation*}\label{eq25 : Theorem-1}
\int_0^T  |w_{n}(t)|_{L^{2}(\rho_n(t))}^{2} \,dt \leq C,
\end{equation*}
and we exploit Lemma \ref{representation of AC curves} to have $\rho_n \in AC(0,T;\mathcal{P}_{2}(\Omega))$ and
\begin{equation}\label{eq35 : Theorem-a}
\begin{aligned}
W_{2}(\rho_n(s),\rho_n(t)) &\leq \int_s^t  \|w_n(\tau)\|_{L^{2}(\rho_n(\tau))} \,d\tau \leq  C (t-s)^{\frac{1}{2}}, \qquad \forall ~ 0\leq s\leq t\leq T.
\end{aligned}
\end{equation}
%
%
%
%
%
%

Now, we investigate the convergence of $\rho_n$. First of all, from Lemma \ref{Lemma : Arzela-Ascoli} with \eqref{eq35 : Theorem-a}, there exists a curve $\rho : [0, T] \rightarrow \mathcal{P}_2(\Omega)$ such that
\begin{equation}\label{eq38 : Theorem-1}
\rho_n(t) \,\mbox{(up to a subsequence) \,narrowly converges to}\, \rho(t) \,~\mbox{as}~ \, n \rightarrow \infty, \quad \forall \ 0\leq t \leq T.
\end{equation}
Then,
due to the lower semicontinuity of Wasserstein distance $W_{2}(\cdot, \cdot)$ with respect to the narrow convergence, we combine \eqref{eq35 : Theorem-a} and \eqref{eq38 : Theorem-1} to have
\begin{equation*}\label{eq37 : Theorem-1}
 W_{2}(\rho(s),\rho(t)) \leq   C (t-s)^{\frac{1}{2}},
\qquad \forall ~ 0\leq s\leq t\leq T,
\end{equation*}
which concludes \eqref{T:ACweakSol:W_1}.
%
Furthermore, we note that the assumption \eqref{T:ACweakSol:V_q} guarantees Proposition \ref{Proposition : compactness} hold. Hence, we have that $\rho_n$ converges to $\rho$ in $ L_{x,t}^{\alpha, \beta}$ for some $\alpha$ and $\beta$ satisfying \eqref{Space:compact}. 

Let us note that 
\begin{equation}\label{eq40 : Theorem-1}
\|\nabla \rho_n^m \|_{L_{x,t}^{\frac{2}{m+1},2}} = \|\rho_n^{\frac{m}{2}} \nabla \rho_n^{\frac{m}{2}} \|_{L_{x,t}^{\frac{2}{m+1},2}} 
\leq \|\rho_n^{\frac{m}{2}}  \|_{L_{x,t}^{\frac{2}{m},\infty}} \| \nabla \rho_n^{\frac{m}{2}} \|_{L_{x,t}^{2}}  \leq C,
\end{equation}
where we exploit $ \| \nabla \rho_n^{\frac{m}{2}} \|_{L_{x,t}^{2}}  \leq C$  and  $\|\rho_n^{\frac{m}{2}} \|_{L_{x,t}^{\frac{2}{m},\infty}}=1$ from $\|\rho_n \|_{L_{x,t}^{1,\infty}} =1$.  
From \eqref{eq40 : Theorem-1}, we note $ \nabla \rho_n^m$ has a weak limit $\eta$ in $L_{x,t}^{\frac{2}{m+1},2}$. Due to the strong convergence of $\rho_n$ to $\rho$ in $ L_{x,t}^{\alpha, \beta}$ 
 , we have
$\eta= \nabla \rho^m$. We put these together to get
\begin{equation}\label{eq3 : Theorem-1}
\begin{aligned}
\iint_{\Omega_T}  \left[\partial_t\varphi \rho_n  -\nabla \rho_n^m  \cdot \nabla \varphi\right ] \,dx\,dt \longrightarrow
\iint_{\Omega_T}  \left[\partial_t\varphi \rho  - \nabla  \rho^m \cdot \nabla \varphi\right ] \,dx\,dt, ~\mbox{as} \ n\rightarrow \infty.
\end{aligned}
\end{equation}

Next, we investigate the following convergence
\begin{equation}\label{eq32 : Theorem-1}
\begin{aligned}
\iint_{\Omega_T}  (V_n \cdot \nabla \varphi  ) \rho_n \,dx\,dt
\longrightarrow  \iint_{\Omega_T}  (V \cdot \nabla \varphi ) \rho \,dx\,dt, \,~\mbox{as} \ n\rightarrow \infty,
\end{aligned}
\end{equation}
for any $\varphi \in \calC_c^\infty(\overline{\Omega} \times [0,T)) $. Indeed, we have
\begin{equation*}
\begin{aligned}
&\iint_{\Omega_T}  \big[(V_n \cdot \nabla \varphi ) \rho_n
- ( V \cdot \nabla \varphi ) \rho\big ] \,dx\,dt \\
&\quad = \iint_{\Omega_T}  \nabla \varphi \cdot \big (V_n- V\big )\rho_n  \,dx\,dt
+ \iint_{\Omega_T}  \nabla \varphi \cdot  V \bke{\rho_n- \rho}  \,dx\,dt
 = I + II.
\end{aligned}
\end{equation*}
  Let us first estimate $I$.  For doing this, we note that if $\rho_n \in L_{x,t}^{1,\infty}$ and $\nabla \rho_n^{\frac{m+q-1}{2}} \in L_{x,t}^{2}  $ then 
    Lemma \ref{L:interpolation} gives us
  \begin{equation*}
  \rho_n \in L_{x,t}^{r_1,r_2} \ \text{ where } \  \frac{d}{r_1} + \frac{d(m+q-1)}{r_2}=d.
  \end{equation*}
 Therefore, if $V_n \in \mathsf{S}_{m,q}^{(q_1, q_2)}$ that is $V_n \in L_{x,t}^{q_1,q_2}$ with $\frac{d}{q_1} + \frac{2+d(q+m-2)}{q_2}=1+d(m-1)$ then
 \begin{equation}\label{eq41 : Theorem-1}
 \|V_n\rho_n\|_{L_{x,t}^r} \leq \| V_n\|_{L_{x,t}^{q_1,q_2}} \|\rho_n \|_{L_{x,t}^{r_1,r_2}},                                      
 \end{equation}
for $ r= \frac{2+d(m+q-1)}{1+md}$, $r_1=\frac{ r q_1}{q_1-r} $ and $r_2=\frac{ r q_2}{q_2-r} $. 

 We exploit \eqref{eq41 : Theorem-1} with $q=1$ to obtain
 \begin{equation}\label{eq28 : Theorem-1}
\|(V_n-V)\rho_n \|_{L_{x,t}^{\frac{2+md}{1+md}}}
\leq \|V_n-V \|_{{\mathsf{S}}_{m,1}^{({q}_1, {q}_2)}} \| \rho_n\|_{L_{x,t}^{r_1, r_2}}
 \leq C 
\|V_n-V \|_{\mathsf{S}_{m,1}^{(q_1,q_2)}} ,
\end{equation}
where we used the fact that $\| \rho_n\|_{L_{x,t}^{r_1, r_2}} \leq C (\|\rho_n\|_{L_{x,t}^{1,\infty}},~ \|\nabla \rho_n^{\frac{m}{2}}\|_{L_{x,t}^{2}})
=C (\|V\|_{\mathsf{S}_{m, 1}^{(q_1,q_2)}}, \int_{\Omega} \rho_{0} \log \rho_{0} \,dx )$ from \eqref{eq11 : Theorem-2-a}.
Hence, we have
\begin{equation}\label{eq13 : Theorem-1}
| I |\leq \|\nabla \varphi \|_{L_{x,t}^{2+md}} \|(V_n - V)\rho_n \|_{L_{x,t}^{\frac{2+md}{1+md}}}
\leq C \|V_n - V \|_{\mathsf{S}_{m,1}^{(q_1,q_2)}} \longrightarrow 0,  \  \mbox{ as } \   n\rightarrow \infty.
\end{equation}
Also, the weak convergence of $\rho_n$ implies $II$ converges to $0$ as $n \rightarrow \infty$.  Hence, this with \eqref{eq13 : Theorem-1} implies \eqref{eq32 : Theorem-1}.
 We put \eqref{eq3 : Theorem-1}, \eqref{eq32 : Theorem-1} and $W_2(\rho_{0,n}, \rho_0) \rightarrow 0$ into \eqref{eq21 : Theorem-1}, and get
\begin{equation*}\label{eq10 : Theorem-1}
\begin{aligned}
\iint_{\Omega_T}  \left [\partial_t\varphi + (V \cdot \nabla \varphi )\right ]\rho
 - \nabla \rho^m  \cdot \nabla \varphi  \,dx\,dt
 =  - \int_{\Omega}\varphi(0) \rho_{0}\,dx,
\end{aligned}
\end{equation*}
for any $\varphi\in C_c^\infty( \overline{\Omega} \times [0,T))$.

 Next, we prove \eqref{T:weakSol:E_1}.  From \eqref{eq11 : Theorem-2-a}, we have
\begin{equation*}\label{eq24 : Theorem-1}
\| \nabla \rho_n^{\frac{m}{2}}\|_{L_{x,t}^2} \leq C.
\end{equation*}
Using the strong convergence of $\rho_n$ to $\rho$, we have that $\nabla \rho_n^{\frac{m}{2}}$ weakly converges to $\nabla \rho^{\frac{m}{2}}$
 in $L^2(\overline{\Omega}_{T})$. Hence, due to the lower semi-continuity of $L^2-$norm with respect to the weak convergence,  we have
\begin{equation}\label{eq18 : Theorem-1}
\begin{aligned}
\|\nabla \rho^{\frac{m}{2}} \|_{L_{x,t}^{2}} 
& \leq C.
\end{aligned}
\end{equation}
 We remind that the entropy is lower semicontinuous with respect to the narrow convergence (refer \cite{ags:book}). 
  Hence, from \eqref{eq11 : Theorem-2-a}, we have
\begin{equation}\label{eq15 : Theorem-1}
\begin{aligned}
\int_{\Omega\times \{t\}} \rho  \log \rho \,dx
  &\leq \liminf_{n \rightarrow \infty} \int_{\Omega\times \{t\}} \rho_n  \log \rho_n  \,dx
  \leq C \qquad \forall ~ t\in (0,T).
\end{aligned}
\end{equation}
We also remind that there exists a constant $c>0$ (independent of $\rho$) such that
\begin{equation}\label{eq16 : Theorem-1}
\begin{aligned}
\int_{\Omega} | \min \{\rho \log \rho,~0 \} | \,dx \leq c \int_{\Omega}  \rho  \langle x \rangle^2  \, dx \leq c (1 + |\Omega|^2).
\end{aligned}
\end{equation}
Combining \eqref{eq15 : Theorem-1} and \eqref{eq16 : Theorem-1}, we have
\begin{equation}\label{eq20 : Theorem-1}
\begin{aligned}
\sup_{0\leq t \leq T} &\int_{\Omega\times \{t\}} \rho  \abs{\log \rho}  \,dx \leq  C.
\end{aligned}
\end{equation}
Furthermore, from \eqref{eq11 : Theorem-2-a}, we have 
$$ \iint_{\Omega_T}  \left | \frac{\nabla \rho_n^m}{\rho_n}\right |^{2} \rho_n \,dx\,dt =\|\nabla \rho_n^{m-\frac{1}{2}}\|^2_{L^2_{x,t}} \leq C.$$ 
Exploiting the strong convergence of $\rho_n$ to $\rho$, we have 
\begin{equation}\label{eq22 : Theorem-1}
 \iint_{\Omega_T} \left | \frac{\nabla \rho^m}{\rho}\right |^{2} \rho \,dx\,dt  \leq C.
\end{equation}\label{eq23 : Theorem-1}
Next, we note that $V$ satisfies \eqref{T:ACweakSol:V_q} and hence it satisfies \eqref{V:speed}. Hence, as in \eqref{speed04}, we have 
\begin{equation*}
 \iint_{\Omega_T} |V|^{2}\rho \,dx\,dt  \leq C.
\end{equation*}
We combine \eqref{eq18 : Theorem-1}, \eqref{eq20 : Theorem-1}, \eqref{eq22 : Theorem-1} and \eqref{eq23 : Theorem-1} to get \eqref{T:ACweakSol:E_1} which completes the proof for the case {\it (i)}.

{\it Proof of (ii)} :
Suppose $\rho_0\in  \mathcal{P}_2(\Omega) \cap L^{q}(\Omega)$ and $V \in  \mathsf{S}_{m,q}^{(q_1,q_2)} $ satisfies  \eqref{T:ACweakSol:V_q}.
Similar to the proof of (i), there exist  $\rho_{0,n}\in \calC^\infty(\overline{\Omega})\cap \mathcal{P}_2(\Omega)$ and  $V_n\in C^\infty_c(\overline{\Omega}\times[0,T)) $ with $V_n(t)\cdot \mathbf{n}=0$ on $\partial\Omega$ for all $t\in[0,T)$ such that
\begin{equation}\label{eq7 : Theorem-2-a}
\lim_{n \rightarrow \infty}\|\rho_{0,n} - \rho_0 \|_{L^q(\Omega)} =0, \qquad \lim_{n \rightarrow \infty}\|V_n - V \|_{\mathsf{S}_{m,q}^{(q_1,q_2)} } =0.
\end{equation}

Exploiting Proposition \ref{proposition : regular existence},  we have $\rho_n\in AC(0,T; \mathcal{P}_2(\Omega))$  satisfying \eqref{eq21 : Theorem-1} and
\begin{equation*}
\begin{aligned}
\sup_{0\leq t \leq T} &\int_{\Omega\times \{t\}}  (\rho_n)^q   \,dx
 + \frac{2qm(q-1)}{(q+m-1)^2} \iint_{\Omega_T} \left |\nabla \rho_n^{\frac{q+m-1}{2}}\right |^2  \,dx\,dt \leq C,
\end{aligned}
\end{equation*}
where $C= C( \|V_n\|_{\mathsf{S}_{m, q}^{(q_1,q_2)}}, \|\rho_{0,n}\|_{L^q(\Omega)} )$.
Due to \eqref{eq7 : Theorem-2-a}, we may choose
$C= C( \|V\|_{\mathsf{S}_{m, q}^{(q_1,q_2)}}, \|\rho_{0}\|_{L^q(\Omega)} )$.

Furthermore, Proposition \ref{P:Energy-speed} gives us
 \begin{equation*}\label{eq12 : Theorem-2-a}
 \sup_{0\leq t \leq T}  \int_{\Omega} \rho_n^q(\cdot, t) \,dx
 + \iint_{\Omega_T} \{ \abs{\nabla \rho_n^{\frac{q+m-1}{2}}}^2  + \left(\abs{\frac{\nabla
\rho_n^m}{\rho_n}}^{2}+|V|^{2}\right ) \rho_n  \}\,dx\,dt \leq C,
\end{equation*}
and \eqref{eq35 : Theorem-a}. Including the convergence of $\rho_n$, the rest of the proof is exactly same as the proof of (i). 

The only difference is that, instead of \eqref{eq28 : Theorem-1} and \eqref{eq13 : Theorem-1}, we exploit \eqref{eq41 : Theorem-1}
 and have
\begin{equation}\label{eq34 : Theorem-1}
\begin{aligned}
| I | \leq \| \nabla \varphi \|_{L_{x,t}^{\frac{r}{r-1}}} \| (V_n-V)\rho_n\|_{L_{x,t}^r} 
 \leq  \| \nabla \varphi \|_{L_{x,t}^{\frac{r}{r-1}}} \|V_n-V\|_{\mathsf{S}_{m,q}^{(q_1,q_2)}} \|\rho_n\|_{L^{r_1, r_2}_{x,t}}  
 \longrightarrow 0, \ \mbox{ as } \  n\rightarrow \infty,
\end{aligned}
\end{equation}
where $r= \frac{2+d(m+q-1)}{1+md}$, $r_1=\frac{r q_1}{q_1-r}$, and $r_2=\frac{r q_2}{q_2-r} $. 
\qed

\subsection{Proof of Theorem \ref{T:ACweakSol_tilde}}
We note that the assumption \eqref{T:ACweakSol_tilde:V_q} on $V$ guarantees Proposition \ref{P:Energy-speed},  Proposition \ref{Proposition : compactness} and Proposition \ref{proposition : regular existence} hold. Let $V_n\in C^\infty_c(\overline{\Omega}\times[0,T)) $ with $V_n(t)\cdot \mathbf{n}=0$ on $\partial\Omega$ for all $t\in[0,T)$ be
such that
\begin{equation*}\label{eq7 : T:weakSol_tilde}
\lim_{n \rightarrow \infty} \left (\|V_n - V \|_{\tilde{\mathsf{S}}_{m,q}^{(\tilde{q}_1,\tilde{q}_2)}}  +  \|V_n - V \|_{L_{x,t}^{1,\tilde{q}_2} }\right)=0,
\end{equation*}
 where $(\tilde{q}_1, \tilde{q}_2)$ satisfies \eqref{T:ACweakSol_tilde:V_q}. 

{\it For the case (i)} : Let  $\rho_n$ be the same as in the proof of (i) of Theorem \ref{T:ACweakSol}.
 We exploit Proposition \ref{proposition : regular existence} to have regular solutions $\rho_n$ 
satisfying  \eqref{eq21 : Theorem-1}, and \eqref{regular:L1-energy} with the constant $C=C(  \|V\|_{\tilde{\mathsf{S}}_{m, 1}^{(\tilde{q}_1,\tilde{q}_2)}}, \int_{\Omega}\rho_0 \log \rho_0 \,dx )$.
Since $V$ also satisfies \eqref{Vq:tilde:energy-speed} for $q=1$, as in the proof of Theorem \ref{T:ACweakSol} (i), Proposition \ref{P:Energy-speed} gives 
\begin{equation*}\label{eq11 : Theorem-2-c}
\sup_{0\leq t \leq T}\int_{\Omega \times \{ t\}}\rho_n |\log \rho_n|  \,dx
+ \iint_{\Omega_T} \{ \left|\nabla \rho_{n}^{\frac{m}{2}} \right|^2 + \left (\left | \frac{\nabla \rho_n^m}{\rho_n}\right |^{2} +|V_n|^{2} \right )\rho_n \}\,dx\,dt \leq C,
\end{equation*}
where the constant $C=C \big(\|V\|_{\tilde{\mathsf{S}}_{m, 1}^{(q_1,q_2)}}, \int_{\Omega} \rho_0 \log \rho_0 \,dx\big)$, and hence we also have \eqref{eq35 : Theorem-a}.

We follow  same steps as in the proof of Theorem \ref{T:ACweakSol}  to complete the proof. The only difference is that we have
 \begin{equation}\label{eq42 : Theorem-1}
\begin{aligned}
| I | &\leq  \| \nabla \varphi \|_{L_{x,t}^{\frac{r}{r-1}}} \|V_n-V\|_{\tilde{\mathsf{S}}_{m,1}^{(\tilde{q}_1,\tilde{q}_2)}} \|\rho_n\|_{L^{r_1, r_2}_{x,t}} \longrightarrow 0, \ \mbox{ as } \ n\rightarrow \infty,
\end{aligned}
\end{equation}
where   $r_1=\frac{ r q_1}{q_1-r} $,  $r_2=\frac{ r q_2}{q_2-r} $ and $ r= \frac{2+md}{1+md} $.  

Indeed, by embedding as in \eqref{V_embedding} for $1-\frac{1}{d}< m < 1$, we have
\begin{equation}\label{eq43 : Theorem-1}
\begin{aligned}
V \in \tilde{\mathsf{S}}_{m,q}^{(\tilde{q}_1,\tilde{q}_2)} \cap L_{x,t}^{1, \tilde{q}_2}, \quad \frac{1}{d}<\frac{1}{\tilde{q}_1} <\frac{2+d(m-1)}{d}  ~ &\Longrightarrow ~
V \in \mathsf{S}_{m,q}^{(\tilde{q}_1^*, \tilde{q}_2)}, ~~ \tilde{q}_1^*=\frac{d\tilde{q}_1}{d-\tilde{q}_1} \\
~ &\Longrightarrow ~ V \in \mathsf{S}_{m,q}^{(\tilde{q}_1^*, \tilde{q}_2)},  \quad 0<\frac{1}{\tilde{q}_1^*} <\frac{1+d(m-1)}{d} 
\end{aligned}
\end{equation}
 with $\|V \|_{\mathsf{S}_{m,q}^{(\tilde{q}_1^*, \tilde{q}_2)}} \leq \|V \|_{\tilde{\mathsf{S}}_{m,q}^{(\tilde{q}_1, \tilde{q}_2)}}$.  
 Therefore, we combine \eqref{eq41 : Theorem-1} and \eqref{eq43 : Theorem-1} for $q=1$ to conclude \eqref{eq42 : Theorem-1}.

\noindent{\it For the case (ii)} : Let $\rho_{0,n}$ and $\rho_n$ be as in the proof of (ii) of Theorem \ref{T:ACweakSol}. Similar to the proof of the case (i), we can complete the proof.
 Similar to the case (i), we can complete the proof.
The only difference is that we have
 \begin{equation*}\label{eq44 : Theorem-1}
\begin{aligned}
| I | &\leq  \| \nabla \varphi \|_{L_{x,t}^{\frac{r}{r-1}}} \|V_n-V\|_{\tilde{\mathsf{S}}_{m,q}^{(\tilde{q}_1,\tilde{q}_2)}} \|\rho_n\|_{L^{r_1, r_2}_{x,t}} \longrightarrow 0, \ \mbox{ as } \ n\rightarrow \infty,
\end{aligned}
\end{equation*}
where $ r= \frac{2+d(m+q-1)}{1+md}$,  $r_1=\frac{ r q_1}{q_1-r} $,   and $r_2=\frac{ r q_2}{q_2-r}$.  
\qed

\subsection{Proof of Theorem \ref{T:weakSol:divfree}}
We note that the assumption \eqref{T:weakSol:divfree:Vq} on $V$ guarantees Proposition \ref{proposition : regular existence} and Proposition \ref{Proposition : compactness} hold.

{\it Proof of (i).}  As in the proof of (i) of Theorem \ref{T:ACweakSol}, there exist $\rho_{0,n}$ and $V_n\in C^\infty_c(\overline{\Omega}\times[0,T)) $ satisfying  $\nabla \cdot V_n(t)=0$, $V_n(t)\cdot \mathbf{n}=0$ on $\partial\Omega$ for all $t\in[0,T)$ and
\begin{equation*}\label{eq45 : Theorem-1}
\lim_{n \rightarrow \infty}\|V_n - V \|_{\mathcal{D}_{m,1}^{(q_1,q_2)} } =0.
\end{equation*}
Exploiting Proposition \ref{proposition : regular existence}, we have a regular solution $\rho_n$ satisfying \eqref{eq21 : Theorem-1} and   
\begin{equation}\label{eq46 : Theorem-1}
\sup_{0\leq t \leq T}\int_{\Omega \times \{ t\}}\rho_n |\log \rho_n|  \,dx
+ \iint_{\Omega_T} \left|\nabla \rho_{n}^{\frac{m}{2}} \right|^2 \,dx\,dt  \leq C,
\end{equation}
with the constant $C=C ( \int_{\Omega} \rho_{0} \log \rho_{0} \,dx )$. 

Note that \eqref{T:weakSol:divfree:Vq} implies \eqref{V:divfree:compact}. Hence, due to the Proposition \ref{Proposition : compactness},  we have that $\rho_n$ (up to a subsequence) converges to $\rho$ in $ L_{x,t}^{\alpha, \beta}$ for some $\alpha$ and $\beta$ satisfying \eqref{Space:compact}. This gives us 
 \begin{equation*}
\rho_n(t) \,\mbox{(up to a subsequence) \,narrowly converges to}\, \rho(t) \,~\mbox{as}~ \, n \rightarrow \infty, \quad a.e ~ t \in (0,T).
\end{equation*}
Combining this with \eqref{eq46 : Theorem-1}, we have 
 \begin{equation*}
\esssup_{0\leq t \leq T}\int_{\Omega \times \{ t\}}\rho |\log \rho|  \,dx \leq C.
\end{equation*}

The rest of the proof is similar to the proof of Theorem \ref{T:ACweakSol} (i). The only difference is that, instead of  \eqref{eq13 : Theorem-1}, we have
\begin{equation}\label{eq48 : Theorem-1}
| I |\leq \|\nabla \varphi \|_{L_{x,t}^\infty} \|(V_n - V)\rho_n \|_{L_{x,t}^{1}}
\leq C \|V_n - V \|_{\mathcal{D}_{m,1}^{(q_1,q_2)}} \longrightarrow 0,  \  \mbox{ as } \   n\rightarrow \infty.
\end{equation}
 Indeed, we exploit $\rho_n \in L_{x,t}^{1,\infty}$ and $\nabla \rho_n^{\frac{m+q-1}{2}} \in L_{x,t}^{2}$ then 
    Lemma \ref{L:interpolation} gives us
  \begin{equation}
  \rho_n \in L_{x,t}^{r_1,r_2} \ \text{ where } \ \frac{d}{r_1} + \frac{2 + q_{m,d}}{r_2}=\frac{d}{q}.
  \end{equation}
 Therefore, if  $V_n \in L_{x,t}^{q_1,q_2}$ with $\frac{1}{r_1}+\frac{1}{q_1}=1, ~~  \frac{1}{r_2}+\frac{1}{q_2}=1$ that is   
\begin{equation}
\begin{aligned}
\frac{d}{r_1} + \frac{2 + q_{m,d}}{r_2}=\frac{d}{q} ~~ &\Longleftrightarrow  ~~ d\left(1-\frac{1}{q_1} \right) + (2+q_{m,d})\left(1-\frac{1}{q_2} \right)=\frac{d}{q}\\
 &\Longleftrightarrow  ~~ \frac{d}{q_1} + \frac{2+q_{m,d}}{q_2} = \frac{d(q-1)}{q} + 2+q_{m,d},
\end{aligned}
\end{equation}
which means $V_n \in \mathcal{D}_{m,q}^{(q_1, q_2)}$ then
 \begin{equation}\label{eq49 : Theorem-1}
 \|V_n\rho_n\|_{L_{x,t}^1} \leq \| V_n\|_{\mathcal{D}_{m,q}^{(q_1,q_2)}} \|\rho_n \|_{L_{x,t}^{r_1,r_2}}.                                     
 \end{equation}
 We exploit \eqref{eq49 : Theorem-1} with $q=1$ to obtain
 \begin{equation*}\label{eq50 : Theorem-1}
\|(V_n-V)\rho_n \|_{L_{x,t}^{1}}
\leq \|V_n-V \|_{{\mathcal{D}}_{m,1}^{({q}_1, {q}_2)}} \| \rho_n\|_{L_{x,t}^{r_1, r_2}}
 \leq C \|V_n-V \|_{\mathcal{D}_{m,1}^{(q_1,q_2)}} ,
\end{equation*}
where we used the fact that $\| \rho_n\|_{L_{x,t}^{r_1, r_2}} \leq C (\|\rho_n\|_{L_{x,t}^{1,\infty}},~ \|\nabla \rho_n^{\frac{m}{2}}\|_{L_{x,t}^{2}} )
=C ( \int_{\Omega} \rho_{0} \log \rho_{0} \,dx )$
from \eqref{eq46 : Theorem-1}.
Hence, we have \eqref{eq48 : Theorem-1} and complete the proof of (i).

{\it Proof of (iii) for $q=1$.}  Besides the assumption on $V$ for the case (i), if $V$ further satisfies $V\in L_{x,t}^{(q_1, q_2)}$ for $\frac{d}{q_1}+\frac{2+ d(m-1)}{q_2}<  2+ d(m-1) $
 then \eqref{regular:narrow-distance:2} gives us 
\begin{equation}\label{eq47 : Theorem-1}
\delta(\rho_n(s), \rho_n(t)) \leq C(t-s)^{a} , \qquad  \forall ~0 \leq s < t \leq T,
\end{equation}
for some constant $a=a(q_1, q_2)$. 

From Lemma \ref{Lemma : Arzela-Ascoli} with \eqref{eq47 : Theorem-1}, there exists a curve $\rho : [0, T] \rightarrow \mathcal{P}_2(\Omega)$ such that
\begin{equation}\label{eq55 : Theorem-1}
\rho_n(t) \,\mbox{(up to a subsequence) \,narrowly converges to}\, \rho(t) \,~\mbox{as}~ \, n \rightarrow \infty, \quad \forall \ 0\leq t \leq T,
\end{equation}
which gives us 
\begin{equation}\label{eq56 : Theorem-1}
\delta(\rho(s), \rho(t)) \leq C(t-s)^{a} , \qquad  \forall ~0 \leq s < t \leq T,
\end{equation}
and
 \begin{equation*}
\sup_{0\leq t \leq T}\int_{\Omega \times \{ t\}}\rho |\log \rho|  \,dx \leq C.
\end{equation*}
The rest of the proof is same as the proof of (i).

{\it Proof of (ii).}
There exists a sequence of vector fields $V_n\in C^\infty_c(\overline{\Omega}\times[0,T)) $ with $\nabla \cdot V_n(t)=0$ and $V_n(t)\cdot \mathbf{n}=0$ on $\partial\Omega$ for all $t\in[0,T)$
such that
\begin{equation*}\label{eq51 : Theorem-1}
\lim_{n \rightarrow \infty}\|V_n - V \|_{\mathcal{D}_{m,q}^{(q_1,q_2)} } =0.
\end{equation*}
We also have $\rho_{0,n}$ satisfying \eqref{eq7 : Theorem-2-a} as in the proof of (ii) of Theorem \ref{T:ACweakSol}. 

Exploiting Proposition \ref{proposition : regular existence}, we have $\rho_n$ which is a regular solution of
\eqref{eq5 : Theorem-1} satisfying \eqref{eq21 : Theorem-1} and 
\begin{equation*}
\begin{aligned}
\sup_{0\leq t \leq T} &\int_{\Omega\times \{t\}}  (\rho_n)^q   \,dx
 +  \iint_{\Omega_T} \left |\nabla \rho_n^{\frac{q+m-1}{2}}\right |^2  \,dx\,dt \leq C,
\end{aligned}
\end{equation*}
with the constant $C=C (\|\rho_0\|_{L^{q}(\Omega)})$. 

As in the proof of (i),  due to the Proposition \ref{Proposition : compactness},  we have that $\rho_n$ (up to a subsequence) converges to $\rho$ in $ L_{x,t}^{\alpha, \beta}$ for some $\alpha$ and $\beta$, and  
 \begin{equation*}\label{eq52 : Theorem-1}
\rho_n(t) \,\mbox{(up to a subsequence) \,narrowly converges to}\, \rho(t) \,~\mbox{as}~ \, n \rightarrow \infty, \quad a.e ~ t \in (0,T).
\end{equation*}
 This gives us
 \begin{equation*}
\begin{aligned}
\esssup_{0\leq t \leq T} &\int_{\Omega\times \{t\}}  \rho^q   \,dx \leq C.
\end{aligned}
\end{equation*}
As in the proof of (i), the rest of the proof is similar to the proof of Theorem \ref{T:ACweakSol} (ii) except for exploiting  \eqref{eq49 : Theorem-1} to have
\begin{equation*}
| I |\leq \|\nabla \varphi \|_{L_{x,t}^\infty} \|(V_n - V)\rho_n \|_{L_{x,t}^{1}}
\leq C \|V_n - V \|_{\mathcal{D}_{m,q}^{(q_1,q_2)}} \longrightarrow 0,  \  \mbox{ as } \   n\rightarrow \infty,
\end{equation*}
instead of  \eqref{eq34 : Theorem-1}.

{\it Proof of (iii) for $q>1$.}   Besides the assumption on $V$ for the case (ii), if  $V$ further satisfies $V\in L_{x,t}^{(q_1, q_2)}$ for $\frac{d}{q_1}+\frac{2+q_{m,d}}{q_2}< \frac{d(q-1)}{q} + 2+ q_{m,d} $ then we have \eqref{eq47 : Theorem-1}. From this, we have \eqref{eq55 : Theorem-1}, \eqref{eq56 : Theorem-1}, and
\begin{equation*}
\begin{aligned}
\sup_{0\leq t \leq T} &\int_{\Omega\times \{t\}}  \rho^q   \,dx \leq C.
\end{aligned}
\end{equation*}
The rest of the proof is same as the proof of (ii).
\qed

\subsection{Proof of Theorem \ref{T:ACweakSol:divfree}}
We note $V \in \mathcal{D}_{m,q,s}^{(q_1, q_2)}$ implies $V \in \mathcal{D}_{m,q}^{(q_1^*, q_2^*)}$ where $q_i^*=q_i/2$ for $i=1, 2$. We also can check that if 
$V \in \mathcal{D}_{m,q,s}^{(q_1, q_2)}$ with $(q_1, q_2)$ satisfying  \eqref{T:ACweakSol:divfree:Vq} then 
$V \in \mathcal{D}_{m,q}^{(q_1^*, q_2^*)}$ with $(q_1^*, q_2^*)$ satisfying \eqref{T:weakSol:divfree:Vq}.

Hence, as in the proof of Theorem \ref{T:weakSol:divfree}, we have  $\rho_n\in AC(0,T; \mathcal{P}_2(\Omega))$ which is a regular solution of
\eqref{eq5 : Theorem-1}. 
Furthermore, since $V$ satisfies the assumptions in Proposition \ref{P:Energy-speed}, we have
\begin{equation*}\label{eq53 : Theorem-1}
\sup_{0\leq t \leq T}\int_{\Omega \times \{ t\}}\rho_n |\log \rho_n|  \,dx
\iint_{\Omega_T} \{ \abs{\nabla \rho_n^{\frac m2}}^2 + \left(\abs{\frac{\nabla
\rho_n^m}{\rho_n}}^{2}+|V_n|^{2}\right ) \rho_n  \}\,dx\,dt
 \leq  C,
\end{equation*}
which gives \eqref{eq35 : Theorem-a} with the constant $C=C (\|V\|_{\mathcal{D}_{m,1,s}^{(q_1, q_2)}}, \, \int_{\Omega} \rho_0 \log \rho_0 \,dx )$. 
We also have
 \begin{equation*}\label{eq54 : Theorem-1}
\sup_{0\leq t \leq T}\int_{\Omega \times \{ t\}}(\rho_n)^q   \,dx
+ \iint_{\Omega_T} \{ \abs{\nabla \rho_n^{\frac{q+m-1}{2}}}^2 + \left(\abs{\frac{\nabla
\rho_n^m}{\rho_n}}^{2}+|V_n|^{2}\right ) \rho_n  \}\,dx\,dt
 \leq  C,
\end{equation*}
and \eqref{eq35 : Theorem-a} with the constant $C=C (\|V\|_{\mathcal{D}_{m,q,s}^{(q_1, q_2)}}, \, \|\rho_0 \|_{L^{q}(\Omega)} )$ for  $q>1$.
 
Now, including the convergence of $\rho_n$, the rest of the proof can be completed exactly same as the proof of Theorem \ref{T:ACweakSol}.

\qed

\appendix  
\section{Existence of a regular solution}\label{A:Exist-regular}
\subsection{Splitting method}
 In this subsection, we introduce the splitting method and construct two sequences of curves in
$\mathcal{P}_{2}(\Omega)$ which are approximate solutions of \eqref{FDE}-\eqref{FDE_bc_ic} where $\rho_0$ and $V$ satisfy  the assumptions in Proposition \ref{proposition : regular existence}. 

For each $n\in \mathbb{N},$ we define approximated solutions
$\varrho_n,~ \rho_n:[0,T)\mapsto \mathcal{P}_{2}(\Omega)$ as follows;
\begin{itemize}

\item
 For $t\in [0,\frac{T}{n}]$, we define $\varrho_n:[0,\frac{T}{n}]\mapsto \mathcal{P}_{2}(\Omega)$ as the solution of
\begin{equation*}\label{eq 1 : splitting}
  \left\{
  \begin{array}{ll}
  \partial_t \varrho_n =   \Delta (\epsilon +\varrho_n)^m, & \text{ in } \ \Omega_{T},  \vspace{1mm}\\
  \frac{\partial\varrho_n}{\partial\mathbf{n}} = 0, & \text{ on } \partial \Omega \times (0, \frac{T}{n}],  \vspace{1mm}\\
  \varrho_n(\cdot,0)=\varrho_0, & \text{ on } \ \Omega.
  \end{array}
  \right.
\end{equation*}
We also define  $\rho_n:[0,\frac{T}{n}]\mapsto
\mathcal{P}_{2}(\Omega)$ as follows
\begin{equation*}
\rho_n(t)= \Psi(t;0, \varrho_n(t)), \quad \forall ~ t \in [0, \frac
T n].
\end{equation*}
Here, we recall \eqref{Flow on Wasserstein} for the definition of $\Psi$. 


\item In general, for each $i=1, \ldots, n-1$ and $t\in (\frac{iT}{n} ,\frac{(i+1)T}{n}]$, we define $\varrho_n:(\frac{iT}{n} ,\frac{(i+1)T}{n}]\mapsto \mathcal{P}_{2}(\Omega)$ as the solution of
\begin{equation*}\label{eq 3 : splitting}
  \left\{
  \begin{array}{ll}
  \partial_t \varrho_n =   \Delta (\epsilon +\varrho_n)^m, & \text{ in } \ \Omega_{T},  \vspace{1mm}\\
  \frac{\partial \varrho_n}{\partial \mathbf{n}} = 0, & \text{ on } \partial \Omega \times (\frac{iT}{n}, \frac{(i+1)T}{n}],  \vspace{1mm}\\
 \varrho_n(\cdot, \frac{iT}{n})=\rho_n( \frac{iT}{n}), & \text{ on } \ \Omega.
  \end{array}
  \right.
\end{equation*}
We also define  $\rho_n:(\frac{iT}{n} ,\frac{(i+1)T}{n}]\mapsto
\mathcal{P}_{2}(\Omega)$ as follows
\begin{equation*}
\rho_n(t)= \Psi(t; iT/n, \varrho_n(t)), \quad \forall ~ t \in
(\frac{iT}{n} ,\frac{(i+1)T}{n}].
\end{equation*}
\end{itemize}

 Now, we investigate some properties useful for the proof of Proposition \ref{proposition : regular existence}.
\begin{lemma}\label{Lemma:AC-curve}\cite[Lemma 5.4]{HKK01}
Let $V \in \mathcal{C}_c^\infty (\overline{\Omega} \times [0, T))$ with
$V\cdot\mathbf{n}=0$ on $\partial \Omega$  and $\rho_0 \in \mathcal{P}_{2}(\Omega)$. 
Suppose that $\rho_n,\varrho_n:[0,T)\mapsto
\mathcal{P}_{2}(\Omega)$ are curves defined as above with the initial
data $\varrho_n(0),\rho_n(0):=\rho_0 $. Then, these curves satisfy
the following properties;
\begin{itemize}
\item[(i)] For all $t\in [0,T),$ we have
\begin{equation}\label{eq29:Lemma:AC-curve}
\|\varrho_{n}(t)\|_{L^q_x}, \  \|\rho_{n}(t)\|_{L^q_x} \leq
\|\rho_0\|_{L^q(\Omega)} e^{\frac{q-1}{q}\int_0^T \|\nabla \cdot
V\|_{L^\infty_x} \,d\tau}, \qquad \forall ~  q \geq 1.
\end{equation}

\item[(ii)] For all $s,t \in [0,T),$ we have
\begin{equation*}\label{eq20:Lemma:AC-curve}
\begin{aligned}
W_{2}(\rho_{n}(s), \rho_{n}(t)) &\leq C |t-s|^{\frac{1}{2}}+ \int_s^t
\|V(\tau)\|_{L^\infty_x} \,d\tau  ,
\end{aligned}
\end{equation*}
where $C= C(\int_0^T \|\nabla V\|_{L^\infty_x}  \,d\tau , \int_\Omega \rho_0 \log \rho_0 \, dx)$

\item[(iii)] For all $t \in [0,T),$ we have
\begin{equation*}\label{eq21:Lemma:AC-curve}
\begin{aligned}
W_{2}(\rho_{n}(t), {\varrho}_{n}(t)) \leq
\max_{\{i=0,1,\dots,n\}}\int_{\frac{iT}{n}}^{\frac{(i+1)T}{n}}
\|V(\tau)\|_{L^\infty_x} \,d\tau.
\end{aligned}
\end{equation*}

\end{itemize}
\end{lemma}

\begin{lemma}\label{Lemma : equi-continuity}\cite[Lemma 5.7]{HKK02}
There exists $\tilde{\alpha} \in (0,1)$ such that, for any $0<\alpha \leq\tilde{\alpha}$, suppose 
\begin{equation*}
\rho_0 \in \mathcal{P}_{2}(\Omega) \cap \calC^\alpha (\overline{\Omega})  \quad
\mbox{and} \quad V \in  L^1(0,T ; \calC^{1, \alpha}(\Omega)),
\end{equation*}
with  $V\cdot\mathbf{n}=0$ on $\partial \Omega$ and let $\rho_n,\varrho_n:[0,T)\mapsto \mathcal{P}_{2}(\Omega)$ be curves
defined as in Lemma \ref{Lemma:AC-curve}.
Then, we have
 \begin{equation*}\label{eq 2 : Holder of splitting}
|\varrho_n(x,t)-\varrho_n(y,t)| + |\rho_n(x,t)-\rho_n(y,t)|\leq
C|x-y|^\alpha,  \qquad \forall ~ x, y \in \Omega , ~ t \in [0, T],
 \end{equation*}
  where $C(\|\rho_0\|_{L^\infty(\Omega)},  \int_0^T \|\nabla V\|_{\calC^\alpha (\Omega)}  \,d\tau)$.
\end{lemma}

\begin{lemma}\label{Lemma:solving ODE}
Let $\rho_n,\varrho_n:[0,T]\mapsto \mathcal{P}_{2}(\Omega)$ be the curves defined as in Lemma \ref{Lemma:AC-curve}. 
Suppose $\varphi \in \calC_c(\overline{\Omega} \times [0,T))$ is such that   $\varphi \in \calC^{2,1}_{x,t}(\overline{\Omega} \times (iT/n, (i+1)T/n))$ for all $i=0, 1, 2, \dots, n-1$. 
For any $s,~t \in [0,T]$, we define 
\begin{equation}\label{eq 15 : solving ODE}
\begin{aligned}
E_n &:= \int_{\Omega}\varphi(x,t) ~\rho_n(x,t)\,dx - \int_{\Omega}\varphi(x,s) ~\rho_n(x,s)\,dx\\
  &\quad  - \int_{s}^{t} \int_{\Omega}  ( \nabla   \varphi \cdot V) ~\rho_n \,dx \,d\tau
  - \int_{s}^{t} \int_{\Omega}  \{ \partial_\tau \varphi ~ \rho_n - \nabla \varphi
\cdot \nabla (\epsilon +\rho_n)^m \}\,dx \,d\tau.
\end{aligned}
\end{equation}

Then,
\begin{equation}\label{eq 14 : solving ODE}
\begin{aligned}
|E_n| \leq  \frac{C}{n},
\end{aligned}
\end{equation}
for some constant  $C:=C (\|\rho_0\|_{W^{2,\infty}(\Omega)},  ~\|V\|_{L^\infty(0,T: W^{2,\infty}(\Omega))}, ~\max_{0\leq i\leq n-1} \|\varphi\|_{C^{2,1}_{x,t}(\bar{\Omega} \times [\frac{iT}{n}, \frac{(i+1)T}{n}])} )$.
\end{lemma}

\begin{proof}
Let $\varphi \in \calC(\overline{\Omega} \times [0,T))$ be such that   $\varphi \in \calC^{2,1}_{x,t}(\overline{\Omega} \times [\frac{iT}{n}, \frac{(i+1)T}{n}])$ for all $i=0, 1, 2, \dots, n-1$.  
For any $s,~t \in [0,T)$, we define
\begin{equation*}
\begin{aligned}
\tilde{E}_n &:= \int_{\Omega}\varphi(x,t) ~\rho_n(x,t)\,dx - \int_{\Omega}\varphi(x,s) ~\rho_n(x,s)\,dx\\
  &\quad  - \int_{s}^{t} \int_{\Omega}  ( \nabla   \varphi \cdot V) ~\rho_n \,dx \,d\tau
  - \int_{s}^{t}
\int_{\Omega}  \{ \partial_\tau \varphi ~ \varrho_n - \nabla \varphi
\cdot \nabla (\epsilon +\varrho_n)^m \}\,dx \,d\tau.
\end{aligned}
\end{equation*}
Then, revisiting the proof of  \cite[Lemma 4.2]{KK-SIMA} or \cite[Lemma 5.7]{HKK01}, we have   
\begin{equation}\label{eq 12 : solving ODE}
\begin{aligned}
|\tilde{E}_n| \leq  \frac{C}{n}
\end{aligned}
\end{equation}
for some constant $C:=C (\int_\Omega \rho_0 \log \rho_0,  ~\|V\|_{L^1(0,T: W^{1,\infty}(\Omega))}, ~\max_{0\leq i\leq n-1} \|\varphi\|_{C^{2,1}_{x,t}(\bar{\Omega} \times [\frac{iT}{n}, \frac{(i+1)T}{n}])} )$.
 
 Note that 
 \begin{equation*}\label{eq 13 : solving ODE}
 E_n=\tilde{E}_n - \int_{s}^{t} \int_{\Omega}  \{ \partial_\tau \varphi ~ (\rho_n-\varrho_n) - \nabla \varphi
\cdot \left [\nabla (\epsilon +\rho_n)^m  -\nabla (\epsilon +\varrho_n)^m \right ] \} \,dx \,d\tau.
 \end{equation*}
{\bf Claim.} We have
\begin{equation}\label{eq 11 : solving ODE}
\left |\int_{s}^{t} \int_{\Omega}  \{ \partial_\tau \varphi ~ (\rho_n-\varrho_n) - \nabla \varphi
\cdot \left [\nabla (\epsilon +\rho_n)^m  -\nabla (\epsilon +\varrho_n)^m \right ] \} \,dx \,d\tau \right | \leq \frac{C}{n} ,
\end{equation}
for some constant $C:=C (\|\rho_0\|_{W^{2,\infty}(\Omega)},  ~\|V\|_{L^\infty(0,T: W^{2,\infty}(\Omega))}, ~\max_{0\leq i\leq n-1} \|\varphi\|_{C^{2,1}_{x,t}(\bar{\Omega} \times [\frac{iT}{n}, \frac{(i+1)T}{n}])})$.
\begin{proof}
For notational convenience, we use notation $\rho$ and $ \varrho$ instead of $\rho_n$ and $ \varrho_n$. 
 For each $\tau \in [s, t]$, there exists $i$ such that $\tau \in  (\frac{iT}{n}, \frac{(i+1)T}{n}]$ and  we have
$$\varrho(x,\tau)= \rho(\psi(\tau; iT/n, x), \tau)) J_{\frac{iT}{n}, \tau}(x),$$
where $J_{a,b}(x)=e^{\int_a^b\nabla \cdot V(\psi(\theta; a,x), \theta) \,d\theta}$. 
This gives us
\begin{equation}\label{eq 1 : solving ODE}
\begin{aligned}
|\rho(x,\tau)-\varrho(x,\tau)| &\leq  |\rho(x,\tau)- \rho(\psi(\tau; iT/n, x), \tau))| +  |\rho(\psi(\tau; iT/n, x), \tau)) -  \rho(\psi(\tau; iT/n, x), \tau))J_{\frac{iT}{n}, \tau}(x)|\\
 &\leq \|\nabla\rho(\cdot, \tau) \|_{L^{\infty}_{x}} |\psi(\tau; iT/n, x)-x | + \|\rho(\cdot, \tau)\|_{L^{\infty}_{x}} |1- J_{\frac{iT}{n}, \tau}(x) |.
\end{aligned}
\end{equation}
%
Note that 
\begin{equation*}\label{eq 2 : solving ODE}
|\psi(\tau; iT/n, x)-x| \leq  \int_{\frac{iT}{n}}^\tau \|V(\cdot,\theta )\|_{L^{\infty}_{x}}\,d\theta \leq (\tau-iT/n)\| V\|_{L^\infty(0,T; L^\infty(\Omega))}
\end{equation*}
  and
\begin{equation*}\label{eq 3 : solving ODE}
\begin{aligned}
\left|J_{\frac{iT}{n}, \tau}(x) -1 \right|&={\left |e^{\int^\tau_{\frac{iT}{n}} \nabla \cdot V\,d\theta}-1\right |} \\
&\leq {e^{\int^\tau_{\frac{iT}{n}} \|\nabla\cdot V\|_{L^{\infty}_{x}}\, d\theta}} \int^\tau_{\frac{iT}{n}} \|\nabla \cdot V\|_{L^{\infty}_{x}}\,d\theta\\
&\leq C \left(\tau-\frac{iT}{n}\right) \leq \frac{C}{n}
\end{aligned}
\end{equation*}
for some constant $C:=C(\|V\|_{L^\infty(0,T, W^{1,\infty}(\Omega)})$.  From Lemma \ref{Lemma : Holder regularity on the flow} and \ref{Lemma:AC-curve}, we have
\begin{equation}\label{eq 4 : solving ODE}
\|\rho\|_{L^{\infty}_{x}},~\|\nabla \rho\|_{L^{\infty}_{x}} \leq C 
\end{equation}
for some constant $C:=C(\|\rho_0\|_{W^{1,\infty}}, \|V\|_{L^1(0,T; W^{1,\infty}(\Omega)})$.  Hence, we combine \eqref{eq 1 : solving ODE}-\eqref{eq 4 : solving ODE} to conclude
\begin{equation}\label{eq 5 : solving ODE}
\|\rho(\cdot,\tau)-\varrho(\cdot,\tau)\|_{L^{\infty}_{x}} \leq  \frac{C}{n}
\end{equation}
for some constant $C:=C(\|\rho_0\|_{W^{1, \infty}(\Omega)}, \|V\|_{L^\infty (0,T; W^{1,\infty}(\Omega)})$.

Next, we estimate 
\begin{equation}\label{eq 8 : solving ODE}
\begin{aligned}
&\nabla (\epsilon+\rho)^m(x)-\nabla  (\epsilon+\varrho)^m(x)\\
&=m\left(\epsilon+\rho(x) \right)^{m-1}\left(\nabla \rho(x)-\nabla \varrho(x) \right) 
+ m \left((\epsilon +\rho(x))^{m-1} -(\epsilon +\varrho(x))^{m-1} \right) \nabla \varrho (x).
\end{aligned}
\end{equation}
We recall
 $$\varrho(x,\tau)=\rho(\psi(\tau;iT/n,x),\tau) e^{\int_{\frac{iT}{n}}^\tau \nabla \cdot V(\psi(\theta;iT/n,x),\theta)\,d\theta}, $$
and hence
 \begin{equation}
\begin{aligned}
\nabla \varrho(x,\tau) &=\nabla\rho(\psi(\tau;iT/n,x),\tau)\cdot \nabla \psi(\tau;iT/n,x) e^{\int_{\frac{iT}{n}}^\tau \nabla \cdot V(\psi,\theta)\,d\theta} \\
&\quad +\rho(\psi(\tau;iT/n,x),\tau)  e^{\int_{\frac{iT}{n}}^\tau \nabla \cdot V(\psi,\theta)\,d\theta} \left(\int_{\frac{iT}{n}}^\tau \nabla(\nabla \cdot V)(\psi,\theta)\cdot \nabla \psi(\theta;iT/n,x) \,d\theta \right),
\end{aligned}
\end{equation}
which gives us
\begin{equation*}
\begin{aligned}
 \nabla \rho(x,\tau)-\nabla\varrho(x,\tau) &=\nabla \rho(x,\tau) -\nabla \rho(\psi,\tau) + \nabla \rho(\psi,\tau) -\nabla\rho(\psi,\tau)\cdot \nabla \psi(\tau;iT/n,x) \\
 &+ \nabla\rho(\psi,\tau)\cdot \nabla \psi(\tau;iT/n,x) \left(1- e^{\int_{\frac{iT}{n}}^\tau \nabla \cdot V} \right)-\rho(\psi,\tau)e^{\int \nabla \cdot V} \left(\int_{\frac{iT}{n}}^\tau \nabla^2 V\cdot \nabla \psi ~d\tau \right).
 \end{aligned}
\end{equation*}
Hence, we have
\begin{equation*}\label{eq 6 : solving ODE}
\begin{aligned}
| \nabla \rho(x,\tau)-\nabla\varrho(x,\tau)| 
&\leq \|\nabla^2 \rho(\cdot,\tau)\|_{L^{\infty}_{x}} \|\psi(\tau;iT/n,\cdot)- Id\|_{L^{\infty}_{x}} +\|\nabla \rho(\cdot,\tau)\|_{L^{\infty}_{x}} \|\nabla (\psi(\tau;iT/n,\cdot) -Id)\|_{L^{\infty}_{x}}  \\
& + \|\nabla\rho(\cdot,\tau)\|_{L^{\infty}_{x}} \|\nabla \psi(\tau;iT/n,\cdot)\|_{L^{\infty}_{x}}  e^{\int_{\frac{iT}{n}}^\tau \| \nabla \cdot V \|_{L^{\infty}_{x}}}  \|\nabla V\|_{L^\infty (0,T;L^\infty(\Omega))} (\tau-iT/n)   \\
&+\|\rho(\cdot,\tau)\|_{L^{\infty}_{x}} e^{\int_{\frac{iT}{n}}^\tau \|\nabla V\|_{L^{\infty}_{x}} \,d\theta}\|\nabla^2 V\|_{L^\infty(0,T; L^\infty(\Omega))} \|\nabla \psi\|_{L^\infty(iT/n, \tau; L^\infty(\Omega))}(\tau-iT/n).
\end{aligned}
\end{equation*}
We recall
$$\|\nabla \psi\|_{L^\infty(iT/n, \tau ; L^\infty(\Omega))} \leq e^{\int_{\frac{iT}{n}}^\tau \|\nabla V\|_{L^{\infty}_{x}}}, \qquad  \|\psi(\tau; iT/n,\cdot) -Id\|_{L^{\infty}_{x}} \leq \| V\|_{L^\infty(0,T;L^\infty(\Omega))}(\tau-iT/n) ,$$
and
$$\|\nabla(\psi(\tau; iT/n, \cdot)-Id)\|_{L^{\infty}_{x}} \leq \|\nabla V\|_{L^\infty(iT/n, \tau ; L^\infty(\Omega))} \|\nabla \psi\|_{L^\infty(iT/n, \tau ; L^\infty(\Omega))}(\tau-iT/n),$$
to get 
\begin{equation}\label{eq 9 : solving ODE}
\|\nabla \rho(\cdot,\tau) -\nabla\varrho(\cdot,\tau)\|_{L^{\infty}_{x}} \leq \frac{C}{n}, 
\end{equation}
with the constant $C:=C(\|V\|_{L^\infty(0,T ;W_x^{2, \infty}(\Omega)},~ \|\rho_0\|_{W^{2, \infty}(\Omega)})$. 
And, we can check
\begin{equation}\label{eq 7 : solving ODE}
\begin{aligned}
\|(\epsilon +\rho(\cdot,\tau))^{m-1} -(\epsilon +\varrho(\cdot,\tau))^{m-1}\|_{L^{\infty}_{x}} &\leq (1-m)\epsilon^{\alpha-2}\|\rho(\cdot,\tau)-\varrho(\cdot,\tau)\|_{L^{\infty}_{x}}.
\end{aligned}
\end{equation}
Combining \eqref{eq 5 : solving ODE}, \eqref{eq 9 : solving ODE} and \eqref{eq 7 : solving ODE} with \eqref{eq 8 : solving ODE}, we conclude
\begin{equation}\label{eq 10 : solving ODE}
\| \nabla (\epsilon+\rho)^m(\cdot,\tau)-\nabla  (\epsilon+\varrho)^m(\cdot,\tau)\|_{L^{\infty}_{x}} \leq \frac{C}{n},
\end{equation}
for some constant $C:=C(\|V\|_{L^\infty(0,T ;W_x^{2, \infty}(\Omega)}),~ \|\rho_0\|_{W^{2, \infty}(\Omega)})$. Finally, we combine \eqref{eq 5 : solving ODE} and \eqref{eq 10 : solving ODE} to get \eqref{eq 11 : solving ODE} which concludes the proof of the Claim.
\end{proof}
Combining \eqref{eq 12 : solving ODE} and \eqref{eq 11 : solving ODE} to have \eqref{eq 14 : solving ODE} which completes the proof of the Lemma.
\end{proof}

\begin{lemma}[Energy estimate in splitting scheme]\label{Energy:Splitting}
Let $\rho_n :[0,T]\mapsto \mathcal{P}_{2}(\Omega)$ be the curves defined as in Lemma \ref{Lemma:AC-curve}. For $0<m<1$ and $d\geq 2$ with nonnegative initial data $\rho_{0}\in \mathcal{P}_2(\Omega)$,

\begin{itemize}
\item[(i)] Let $ 1-\frac{1}{d} < m <1$. Suppose that $V$ satisfies \eqref{V:Lq-energy}.

\begin{itemize}
\item[$\bullet$] For $q=1$, 
the following estimate holds: For any $0 \leq s< t\leq T$,
\begin{equation}\label{L1-energy:splitting}
\int_{\Omega}\left(\rho_n\log \rho_n\right)(\cdot,t)\,dx
 +  \frac{2}{m}\int_{s}^{t}\int_{\Omega} \left|\nabla (\epsilon+\rho_n)^{\frac{m}{2}}\right|^2 \,dx\,dt \le \int_{\Omega}\left(\rho_n\log \rho_n\right)(\cdot,s)\,dx +C+ E_n,
\end{equation}
where  $C = C ( \|V\|_{\mathsf{S}_{m,1}^{(q_1, q_2)}}, \, \int_{\Omega} \rho_0 \log \rho_0 \,dx )$. 
\smallskip 
\item[$\bullet$] For $q>1$, 
 the following estimate holds: : For any $0 \leq s< t\leq T$,
\begin{equation}\label{Lq-energy:splitting}
  \int_{\Omega } (\epsilon+\rho_n)^q (\cdot, t) \,dx
 +  \frac{2mq(q-1)}{(q+m-1)^2}\int_{s}^{t}\int_{\Omega} \left|\nabla (\epsilon+\rho_n)^{\frac{q+m-1}{2}}\right|^2 \,dx\,dt \le  \int_{\Omega} (\epsilon+\rho_n)^q (\cdot, s) \,dx +C + E_n ,
\end{equation}
where  $C = C (\|V\|_{\mathsf{S}_{m,q}^{(q_1, q_2)}}, \, \|\rho_{0}\|_{L^{q} (\Omega)} )$.
\end{itemize}

\item[(ii)]  Let $(1-\frac{2}{d})_{+} < m <1$. Suppose that $V$ satisfies \eqref{V:tilde:Lq-energy}.

\begin{itemize}
\item[$\bullet$] For $q=1$, 
\eqref{L1-energy:splitting} holds with $C = C (\|V\|_{\tilde{\mathsf{S}}_{m,1}^{(\tilde{q}_1, \tilde{q}_2)}}, \, \int_{\Omega} \rho_0 \log \rho_0 \,dx )$.
\smallskip 
\item[$\bullet$] For $q>1$, 
\eqref{Lq-energy:splitting} holds with 
$C = C (\|V\|_{\tilde{\mathsf{S}}_{m,q}^{(\tilde{q}_1, \tilde{q}_2)}}, \, \|\rho_{0}\|_{L^{q} (\Omega)} )$.
\end{itemize}

\smallskip 

\item[(iii)] Assume $\nabla\cdot V = 0$ for $V\in\calC^{\infty} (\overline{\Omega}_{T})$ with $V\cdot \textbf{n}=0$ on $\partial \Omega$.

\smallskip 
\begin{itemize}
\item[$\bullet$] For $q=1$, 
\eqref{L1-energy:splitting} holds with $C = C (\int_{\Omega} \rho_0 \log \rho_0 \,dx )$.
\smallskip 
\item[$\bullet$] For $q>1$, 
\eqref{Lq-energy:splitting}
holds with $C = C ( \|\rho_{0}\|_{L^{q} (\Omega)} )$.
\end{itemize}
\end{itemize}
\end{lemma}

\begin{proof}
From \eqref{eq 15 : solving ODE}, we have
\begin{equation}\label{eq 1 : Energy : splitting}
\begin{aligned}
&\int_{\Omega}\varphi(x,t) ~\rho_n(x,t)\,dx - \int_{\Omega}\varphi(x,s) ~\rho_n(x,s)\,dx\\
  &=  \int_{s}^{t} \int_{\Omega}  ( \nabla   \varphi \cdot V) ~\rho_n \,dx \,d\tau
 + \int_{s}^{t} \int_{\Omega}  \{ \partial_\tau \varphi ~ \rho_n - \nabla \varphi
\cdot \nabla (\epsilon +\rho_n)^m \}\,dx \,d\tau + E_n.
\end{aligned}
\end{equation}
First, we take $\varphi=\log \rho_n$ in \eqref{eq 1 : Energy : splitting}. Then, 
\begin{equation*}
\int_s^t\int_\Omega \partial_\tau \varphi \rho_n \,dx \,d\tau =\int_s^t\int_\Omega \partial_\tau \rho_n\,dx \,d\tau =0
\end{equation*}
due to the mass conservation. Next, we have
\begin{equation*}
\begin{aligned}
\int_{s}^{t}\int_{\Omega} \nabla \varphi \cdot \nabla (\epsilon+\rho_n)^m &=\int_{s}^{t}\int_{\Omega} \frac{\nabla \rho_n}{\rho_n} \cdot \nabla(\epsilon+\rho_n)^m=\int_{s}^{t}\int_{\Omega} \frac{\nabla(\epsilon+\rho_n)}{\rho_n}\cdot \nabla (\epsilon+\rho_n)^m\\
&= \frac{4}{m}\int_{s}^{t}\int_{\Omega} \frac{\epsilon+\rho_n}{\rho_n}\left | \nabla(\epsilon+\rho_n)^{\frac{m}{2}}\right |^2 \\
&\geq \frac{4}{m}\int_{s}^{t}\int_{\Omega} \left | \nabla(\epsilon+\rho_n)^{\frac{m}{2}}\right |^2.
\end{aligned}
\end{equation*}
We also have
\begin{equation*}
\begin{aligned}
\int_{s}^{t}\int_{\Omega} \left(\nabla \varphi \cdot V \right)\rho_n& =\int_{s}^{t}\int_{\Omega} V \cdot \nabla \rho_n =\int_{s}^{t}\int_{\Omega} V\cdot \nabla (\epsilon+\rho_n).\\
\end{aligned}
\end{equation*}
Hence, 
\begin{equation*}\label{eq 2 : Energy : splitting}
\begin{aligned}
&\int_{\Omega}\left(\rho_n\log \rho_n\right)(x,t)\,dx + \frac{4}{m}\int_{s}^{t}  \int_{\Omega}  \left | \nabla(\epsilon+\rho_n)^{\frac{m}{2}}\right |^2\,dx \,d\tau  \\
& \leq \int_{\Omega}\left(\rho_n\log \rho_n\right)(x,s)\,dx + \int_{s}^{t}  \int_{\Omega}  V\cdot \nabla (\epsilon+\rho_n) \,dx \,d\tau  + E_n.
\end{aligned}
\end{equation*}
Let 
\begin{equation*}
\mathcal{J}_1 :=\int_\Omega V\cdot \nabla(\epsilon +\rho_n) \,dx.
\end{equation*}
Then, by the exactly same argument as in the proof of Proposition \ref{P:energy}, we conclude \eqref{L1-energy:splitting} for cases  (i), (ii) and (iii).

Secondly, we take $\varphi :=q (\epsilon+\rho_n)^{q-1}~(q>1)$  in \eqref{eq 1 : Energy : splitting} . Then
\begin{equation}\label{eq 6 : Energy : splitting}
\begin{aligned}
&\int_{\Omega}\varphi(x,t) ~\rho_n(x,t)\,dx - \int_{\Omega}\varphi(x,s) ~\rho_n(x,s)\,dx\\
  &=  q\int_{\Omega}(\epsilon+ \rho_n)^{q-1}(x,t) ~\rho_n(x,t)\,dx - q\int_{\Omega}(\epsilon+ \rho_n)^{q-1}(x,s) ~\rho_n(x,s)\,dx\\
  &=q \left ( \int_\Omega (\epsilon+ \rho_n)^{q}(x,t) \,dx- \int_\Omega (\epsilon + \rho_n)^{q}(x,s) \,dx\right ) \\
& ~~- \epsilon q \left ( \int_\Omega (\epsilon + \rho_n)^{q-1}(x,t) \,dx- \int_\Omega (\epsilon + \rho_n)^{q-1}(x,s) \,dx\right ),
\end{aligned}
\end{equation}
and
\begin{equation}\label{eq 3 : Energy : splitting}
\begin{aligned}
\int_s^t \int_\Omega \partial_\tau \varphi \rho_n \,dx\,d\tau &=q\int_s^t \int_\Omega [\partial_\tau (\epsilon+\rho_n)^{q-1}] \rho_n \,dx\,d\tau\\
&=q\int_s^t \int_\Omega [\partial_\tau (\epsilon+\rho_n)^{q-1}][ (\epsilon+\rho_n ) -\epsilon ]\,dx\,d\tau\\
&=(q-1) \int_s^t \int_\Omega\partial_\tau\left(\epsilon+ \rho_n \right)^{q}  \,dx\,d\tau -q\epsilon \int_s^t \int_\Omega \partial_\tau\left(\epsilon+ \rho_n \right)^{q-1}  \,dx\,d\tau \\
&=(q-1) \left ( \int_\Omega (\epsilon+ \rho_n)^{q}(x,t) \,dx- \int_\Omega (\epsilon + \rho_n)^{q}(x,s) \,dx\right ) \\
& ~~- \epsilon q\left ( \int_\Omega (\epsilon + \rho_n)^{q-1}(x,t) \,dx- \int_\Omega (\epsilon + \rho_n)^{q-1}(x,s) \,dx\right ).
\end{aligned}
\end{equation}
Next
\begin{equation}\label{eq 4 : Energy : splitting}
\begin{aligned}
\int_s^t \int_\Omega\nabla \varphi \cdot\nabla (\epsilon+\rho_n)^m &= q\int_s^t \int_\Omega \nabla(\epsilon + \rho_n)^{q-1} \cdot \nabla (\epsilon+\rho_n)^m \,dx \,d\tau \\
&=\frac{4mq(q-1)}{(m+q-1)^2}\int_s^t \int_\Omega  \left | \nabla (\epsilon+\rho_n)^{\frac{m+q-1}{2}} \right |^2 \,dx\,d\tau.
\end{aligned}
\end{equation}
Combining \eqref{eq 6 : Energy : splitting}, \eqref{eq 3 : Energy : splitting}, \eqref{eq 4 : Energy : splitting} with \eqref{eq 1 : Energy : splitting}, we have
\begin{equation*}\label{eq 7 : Energy : splitting}
\begin{aligned}
 \int_\Omega (\epsilon+ \rho_n)^{q}(x,t) \,dx + \frac{4mq(q-1)}{(m+q-1)^2}\int_s^t \int_\Omega  \left | \nabla (\epsilon+\rho_n)^{\frac{m+q-1}{2}} \right |^2 \,dx\,d\tau  \\
 = \int_\Omega (\epsilon + \rho_n)^{q}(x,s) \,dx +  q \int_s^t \int_\Omega \nabla (\epsilon+\rho_n)^{q-1} \cdot V \rho_n \,dx \,d\tau + E_n.
\end{aligned}
\end{equation*}
Let 
\begin{equation*}
\mathcal{J}_q :=\int_\Omega \left | V\cdot \nabla(\epsilon +\rho_n)^q \right |\,dx, \qquad q >1.
\end{equation*}
We note
\begin{equation*}
 q \int_\Omega \nabla (\epsilon+\rho_n)^{q-1} \cdot V \rho_n \,dx \leq (q-1) \mathcal{J}_q. 
\end{equation*}
Again, by the exactly same argument as in the proof of Proposition \ref{P:energy}, we conclude \eqref{Lq-energy:splitting} for cases  (i), (ii) and (iii).
\end{proof}

\subsection{Proof of Proposition \ref{proposition : regular existence}}
(1) Assume $\epsilon>0$ and let $\rho_n:[0,T]\mapsto \mathcal{P}_2(\Omega)$ be defined as in the splitting method. Then, from Lemma \ref{Lemma:solving
ODE}, we have
for any $\varphi \in \calC_c^\infty(\overline{\Omega} \times [0,T))$
and $s,~t \in [0,T)$,
\begin{equation}\label{eq3 : Theorem : Toy Fokker-Plank}
\begin{aligned}
& \int_{\Omega}\varphi(x,t) \rho_n(x,t)\,dx - \int_{\Omega}\varphi(x,s) \rho_n(x,s)\,dx\\
&= \int_{s}^{t} \int_{\Omega}   (\nabla   \varphi \cdot V) \rho_n
\,dx  \,d\tau + \int_{s}^{t} \int_{\Omega} \left\{
\partial_\tau\varphi  \rho_n -  \nabla \varphi \cdot
\nabla(\epsilon + \rho_n)^m\right\}\,dx \,d\tau
 + E_n,
\end{aligned}
\end{equation}
with $|E_n| \longrightarrow 0 $ as $n \rightarrow \infty$.

 Recalling Lemma \ref{Lemma:AC-curve}, we know%
\begin{equation}\label{eq5 : Theorem : Toy Fokker-Plank}
W_{2}(\rho_n(s),\rho_n(t))\leq C |t-s|^{\frac{1}{2}} + \int_s^t
\|V(\tau)\|_{L^\infty_x}  \,d\tau,
\end{equation}
where $C=C\left(  \int_\Omega \rho_0 \log \rho_0 \, dx,  \,  \,\|V\|_{L^1(0,T;
W^{1,\infty}(\Omega))} \right) $. Combining the estimates \eqref{eq5
: Theorem : Toy Fokker-Plank} and Lemma \ref{Lemma : Arzela-Ascoli},
there exist a subsequence (by abusing notation) $\rho_n$ and a limit
curve  $\rho_\epsilon :[0,T]\mapsto \mathcal{P}_{2}(\Omega)$ such that
\begin{equation*}\label{eq11 : Theorem : Toy Fokker-Plank}
\rho_{n}(t) ~~ \mbox{ converges to} ~~\rho_\epsilon (t) ~~ \mbox{ in} ~
\mathcal{P}_{2}({\Omega}), \qquad \text{for all} ~~ t\in[0,T].
\end{equation*}
 This implies
\begin{equation}\label{eq 1 : proposition : regular existence}
\begin{aligned}
\int_{\Omega}\varphi(x,t) \rho_n(x,t)\,dx -
\int_{\Omega}\varphi(x,s) \rho_n(x,s)\,dx \longrightarrow
\int_{\Omega}\varphi(x,t) \rho_\epsilon(x,t)\,dx - \int_{\Omega}\varphi(x,s)
\rho_\epsilon (x,s)\,dx,
\end{aligned}
\end{equation}
and
\begin{equation}\label{eq 2 : proposition : regular existence}
\begin{aligned}
\int_{s}^{t} \int_{\Omega}  \left[ (\nabla   \varphi \cdot V) \rho_n
+ \partial_\tau\varphi  \rho_n \right] \,dx  \,d\tau
\longrightarrow \int_{s}^{t} \int_{\Omega}  \left[ (\nabla
\varphi\cdot V) + \partial_\tau\varphi  \right]\rho_\epsilon \,dx  \,d\tau,
\end{aligned}
\end{equation}
as $n \rightarrow \infty$. Furthermore,  from Lemma \ref{Lemma :
equi-continuity}, $\rho_n$ are equi-continuous with
respect to the space variable. We combine this equi-continuity with
the uniform bound \eqref{eq29:Lemma:AC-curve} to claim that $\rho_n$ actually pointwise converge to $\rho_\epsilon$.
Furthermore, \eqref{Lq-energy:splitting} with $q=m+1$ implies that
$\nabla(\epsilon + \rho_n)^{m}$ is bounded in $L^2([0,T]\times\Omega)$ and hence
it (up to a subsequence) has a weak limit in $L^2([0,T]\times\Omega)$.
Combining this with $q=m+1$ and the pointwise convergence of $\rho_n$ to
$\rho_\epsilon$, we have
\begin{equation}\label{eq 3 : proposition : regular existence}
\int_{s}^{t} \int_{\Omega}  \nabla \varphi \cdot \nabla(\epsilon + \rho_n
)^m    \,dx \,d\tau \longrightarrow \int_{s}^{t} \int_{\Omega}
\nabla \varphi \cdot \nabla (\epsilon + \rho_\epsilon)^m  \,dx \,d\tau.
\end{equation}
Finally, we put \eqref{eq 1 : proposition : regular existence},
\eqref{eq 2 : proposition : regular existence} and \eqref{eq 3 :
proposition : regular existence} into \eqref{eq3 : Theorem : Toy
Fokker-Plank} to get
\begin{equation*}\label{eq 4 : proposition : regular existence}
 \int_{\Omega}\varphi(x,t) \rho_\epsilon(x,t) \,dx - \int_{\Omega}\varphi(x,s) \rho_\epsilon(x,s)\,dx
=\int_{s}^{t} \int_{\Omega}  \left [ \left(\partial_\tau\varphi  + V
\cdot \nabla \varphi \right){\rho_\epsilon} -\nabla  \varphi \cdot \nabla (\epsilon + \rho_\epsilon)^m   \right ] \,dx\,d\tau.
\end{equation*}

Next, from Lemma \ref{Lemma : equi-continuity} and the pointwise
convergence of $\rho_n$, we have
\begin{equation}\label{eq18 : Theorem : Toy Fokker-Plank}
\left |\rho_\epsilon (x,t)-\rho_\epsilon(y,t)\right | \leq C |x-y|^\alpha, \quad
\forall ~x,~y \in \Omega, \, t \in [0,T),
\end{equation}
and \eqref{eq5 : Theorem : Toy Fokker-Plank} gives us
\begin{equation}\label{eq20 : Theorem : Toy Fokker-Plank}
W_{2}(\rho_\epsilon(s),\rho_\epsilon(t))\leq C |t-s|^{\frac{1}{2}} + \int_s^t
\|V(\tau)\|_{L^\infty_x}  \,d\tau,
\end{equation}
where $C=C\left(  \int_\Omega \rho_0 \log \rho_0 \, dx,  \,  \,\|V\|_{L^1(0,T; W^{1,\infty}(\Omega))} \right) $.

Furthermore, 
Lemma \ref{Energy:Splitting} gives us
\begin{equation}\label{eq21 : Theorem : Toy Fokker-Plank}
\sup_{0\leq t \leq T}\int_{\Omega \times \{ t\}} \rho_n\log \rho_n\,dx
 +  \frac{2}{m}\iint_{\Omega_T} \left|\nabla (\epsilon+\rho_n)^{\frac{m}{2}}\right|^2 \,dx\,dt \le C+ E_n,
\end{equation}
where  
\begin{equation}\label{eq24 : Theorem : Toy Fokker-Plank}
C= \
 \begin{cases}
  C (\|V\|_{\mathsf{S}_{m,1}^{(q_1, q_2)}}, \, \int_{\Omega} \rho_0 \log \rho_0 \,dx), & \text{if } \ 1-\frac{1}{d}<m<1 \text{ and } 
  (\ref{V:Lq-energy}) \text{ hold},  \vspace{1mm}\\
  C (\|V\|_{\tilde{\mathsf{S}}_{m,1}^{(\tilde{q}_1, \tilde{q}_2)}}, \, \int_{\Omega} \rho_0 \log \rho_0 \,dx), & \text{if } \ \left(1-\frac{2}{d} \right)_+<m<1 \text{ and } 
   (\ref{V:tilde:Lq-energy}) \text{ hold}, \vspace{1mm}\\
  C (\int_{\Omega} \rho_0 \log \rho_0 \,dx), & \text{if } \ \nabla\cdot V = 0.
 \end{cases}
\end{equation} 
 Similarly, 
 from Lemma \ref{Energy:Splitting}, we also have 
 \begin{equation}\label{eq22 : Theorem : Toy Fokker-Plank}
 \sup_{0\leq t \leq T}\int_{\Omega \times \{ t\}} (\epsilon+\rho_n)^q (\cdot, t) \,dx
 +  \frac{2mq(q-1)}{(q+m-1)^2}\iint_{\Omega_T} \left|\nabla (\epsilon+\rho_n)^{\frac{q+m-1}{2}}\right|^2 \,dx\,dt \le  C + E_n ,
\end{equation}
where  
\begin{equation}\label{eq25 : Theorem : Toy Fokker-Plank}
C=  \left\{
  \begin{array}{ll}
  C (\|V\|_{\mathsf{S}_{m,q}^{(q_1, q_2)}}, \, \|\rho_{0}\|_{L^{q} (\Omega)}), & \text{if } \ 1-\frac{1}{d}<m<1 \text{ and } 
 (\ref{V:Lq-energy}) \text{ hold},  \vspace{1mm}\\
  C (\|V\|_{\tilde{\mathsf{S}}_{m,q}^{(\tilde{q}_1, \tilde{q}_2)}}, \, \|\rho_{0}\|_{L^{q} (\Omega)}), & \text{if } \ \left(1-\frac{2}{d} \right)_+<m<1 \text{ and } 
  (\ref{V:tilde:Lq-energy}) \text{ hold}, \vspace{1mm}\\
  C (\|\rho_{0}\|_{L^{q} (\Omega)} ), & \text{if } \ \nabla\cdot V = 0.
  \end{array}
  \right.
\end{equation} 
Now, we exploit the lower semicontinuity of the entropy functional with respect to the narrow convergence and let $n \rightarrow \infty$ in \eqref{eq21 : Theorem : Toy Fokker-Plank} to get
\begin{equation*}\label{eq23 : Theorem : Toy Fokker-Plank}
\sup_{0\leq t \leq T}\int_{\Omega \times \{ t\}} \rho_\epsilon\log \rho_\epsilon\,dx
 +  \frac{2}{m}\iint_{\Omega_T} \left|\nabla (\epsilon+\rho_\epsilon)^{\frac{m}{2}}\right|^2 \,dx\,dt \le C,
\end{equation*}
where $C$ is same as in \eqref{eq24 : Theorem : Toy Fokker-Plank}. 
Similarly, we let $n \rightarrow \infty$ in \eqref{eq22 : Theorem : Toy Fokker-Plank} to get
\begin{equation*}\label{eq27 : Theorem : Toy Fokker-Plank}
 \sup_{0\leq t \leq T}\int_{\Omega \times \{ t\}} (\epsilon+\rho_\epsilon)^q (\cdot, t) \,dx
 +  \frac{2mq(q-1)}{(q+m-1)^2}\iint_{\Omega_T} \left|\nabla (\epsilon+\rho_\epsilon)^{\frac{q+m-1}{2}}\right|^2 \,dx\,dt \le  C ,
\end{equation*}
where $C$ is same as in \eqref{eq25 : Theorem : Toy Fokker-Plank}.

(2) Now we consider the set of curves $\{\rho_\epsilon \}_\epsilon$. As $\epsilon \rightarrow 0$, exactly the same way we did in (1), there exist a limit
curve  $\rho :[0,T]\mapsto \mathcal{P}_{2}(\Omega)$ such that 
\begin{equation*}
\rho_{\epsilon}(t) ~~ \mbox{ (up to a sequence) converges to} ~~\rho (t) ~~ \mbox{ in} ~
\mathcal{P}_{2}({\Omega}), \qquad \text{for all} ~~ t\in[0,T).
\end{equation*}
Furthermore, we have
\begin{equation}\label{eq13 : Theorem : Toy Fokker-Plank}
 \int_{\Omega}\varphi(x,t) \rho(x,t) \,dx - \int_{\Omega}\varphi(x,s) \rho(x,s)\,dx
=\int_{s}^{t} \int_{\Omega}  \left [ \left(\partial_\tau\varphi  + V
\cdot \nabla \varphi \right){\rho} -\nabla  \varphi \cdot \nabla  \rho^m   \right ] \,dx\,d\tau,
\end{equation}

\begin{equation*}\label{eq14 : Theorem : Toy Fokker-Plank}
W_{2}(\rho(s),\rho(t))\leq C |t-s|^{\frac{1}{2}} + \int_s^t
\|V(\tau)\|_{L^\infty_x}  \,d\tau,
\end{equation*}
and
\begin{equation*}
\left |\rho (x,t)-\rho(y,t)\right | \leq C |x-y|^\alpha, \quad
\forall ~x,~y \in \Omega, \, t \in [0,T),
\end{equation*}
from \eqref{eq 4 : Energy : splitting}, \eqref{eq20 : Theorem : Toy Fokker-Plank} and \eqref{eq18 : Theorem : Toy Fokker-Plank}, respectively. We also have
\begin{equation}\label{eq28 : Theorem : Toy Fokker-Plank}
\sup_{0\leq t \leq T}\int_{\Omega} \rho (\cdot, t) \log \rho (\cdot, t)\,dx
 +  \frac{2}{m}\iint_{\Omega_T} \left|\nabla \rho^{\frac{m}{2}}\right|^2 \,dx\,dt \le C,
\end{equation}
where $C$ is same as in \eqref{eq24 : Theorem : Toy Fokker-Plank}, and  
\begin{equation}\label{eq29 : Theorem : Toy Fokker-Plank}
 \sup_{0\leq t \leq T}\int_{\Omega} \rho^q (\cdot, t) \,dx
 +  \frac{2mq(q-1)}{(q+m-1)^2}\iint_{\Omega_T} \left|\nabla \rho^{\frac{q+m-1}{2}}\right|^2 \,dx\,dt \le  C ,
\end{equation}
where $C$ is same as in \eqref{eq25 : Theorem : Toy Fokker-Plank}. This completes the proof of energy estimates \eqref{regular:L1-energy} and \eqref{regular:Lq-energy}.

To prove \eqref{regular:narrow-distance} or \eqref{regular:narrow-distance:2}, we take $\varphi \in W^{1,\infty}(\Omega)$ such that $\|\varphi\|_{W^{1,\infty}(\Omega)} \leq 1$ in \eqref{eq13 : Theorem : Toy Fokker-Plank}. Then, we have
\begin{equation}\label{eq26 : Theorem : Toy Fokker-Plank}
\begin{aligned}
 \left | \int_{\Omega}\varphi(x) \rho(x,t) \,dx - \int_{\Omega}\varphi(x) \rho(x,s)\,dx \right |
&\leq \|\nabla \varphi\|_{L^\infty(\Omega)} \int_{s}^{t} \int_{\Omega}  \left (| V |\rho + | \nabla  \rho^m | \right) \,dx\,d\tau.
\end{aligned}
\end{equation}
We first note that, for any $q\geq 1$, it holds
\begin{equation*}
\|\nabla \rho^{\frac{m}{2}}\|_{L^{2}_{x,t}(\Omega_T)} \leq C
\end{equation*}
where $C$ is the constant from the energy estimates \eqref{eq28 : Theorem : Toy Fokker-Plank} or \eqref{eq29 : Theorem : Toy Fokker-Plank}.
Therefore, for any $0 \leq s<t \leq T$, it follows that 
\begin{equation}\label{eq30 : Theorem : Toy Fokker-Plank}
\begin{aligned}
  \int_{s}^{t} \int_{\Omega}   | \nabla  \rho^m |  \,dx\,d\tau 
&= 2\int_{s}^{t} \int_{\Omega}   \rho^{\frac{m}{2}} \left| \nabla  \rho^{\frac{m}{2}} \right| \,dx\,d\tau\\
&\leq \left( \int_{s}^{t} \int_{\Omega}   \left| \nabla  \rho^{\frac{m}{2}} \right|^2 \,dx\,d\tau \right)^{\frac{1}{2}}\left( \int_{s}^{t} \int_{\Omega}   \rho^m \,dx\,d\tau \right)^{\frac{1}{2}}\\
&\leq \left( \int_{s}^{t} \int_{\Omega}   \left| \nabla  \rho^{\frac{m}{2}} \right|^2 \,dx\,d\tau \right)^{\frac{1}{2}} \left( \int_{s}^{t} \left (\int_{\Omega}   \rho \,dx \right)^m \,d\tau \right)^{\frac{1}{2}}\\
&\leq C |t-s|^{\frac{1}{2}},
\end{aligned}
\end{equation}
because $0<m<1$ and $L^1$-mass conservation property. 

Next, we observe the following, for $\hat{r}_1 = \frac{q_1}{q_1 -1}$ and $\hat{r}_2 = \frac{q_2}{q_2 -1}$,
\[
 \int_{s}^{t} \int_{\Omega}  | V |\rho  \,dx\,d\tau \leq \| V\|_{L_{x,t}^{q_1,q_2}} \|\rho\|_{L_{x,t}^{\hat{r}_1, \hat{r}_2}}.
\]

$\bullet$ Let $V\in \mathsf{S}_{m,q}^{(q_1,q_2)}$ in \eqref{V:Lq-regular}. Then by the energy estimates, 
\[
\rho \in L_{x,t}^{r_1,r_2}  \ \text{ for }  \  \frac{d}{r_1} + \frac{2+d(m+q-2)}{r_2} = d
\]
as in Lemma~\ref{L:interpolation}. For $V\in \mathsf{S}_{m,q}^{(q_1,q_2)}$, the pair $(\hat{r}_1, \hat{r}_2)$ corresponds to 
\[
\frac{d}{\hat{r}_1} + \frac{2+d(q+m-2)}{\hat{r}_2} = 1+dq.
\]
Let  
\[
r_1 = \hat{r}_1 \quad \text{ and } \quad [2+d(q+m-2)] \left( \frac{1}{\hat{r}_2} - \frac{1}{r_2}\right)= 1+d(q-1). 
\]
Then it follows that 
\begin{equation}\label{eq31 : Theorem : Toy Fokker-Plank}
\begin{aligned}
 \|\rho \|_{L_{x,t}^{\hat{r}_1,\hat{r}_2}} 
& \leq \|\rho\|_{L_{x,t}^{r_1,r_2}} |t-s|^{\frac{1}{\hat{r}_2} -\frac{1}{r_2}}\\
&= \|\rho\|_{L_{x,t}^{r_1,r_2}} |t-s|^{\frac{1+d(q-1)}{2+d(q+m-2)}}\\
&\leq C |t-s|^{\frac{1}{2}} \quad \mbox{for} \quad |t-s| <1, 
\end{aligned} 
\end{equation}
where $C$ is the constant from energy estimates \eqref{eq28 : Theorem : Toy Fokker-Plank} or \eqref{eq29 : Theorem : Toy Fokker-Plank}. 
Therefore, it holds that 
\begin{equation}\label{V_rho_est}
\int_{s}^{t} \int_{\Omega}  | V |\rho  \,dx\,d\tau \leq C\|V\|_{\mathsf{S}_{m,q}^{(q_1,q_2)}} \,|t-s|^{\frac{1}{2}} \quad \mbox{for} \quad |t-s| <1,
\end{equation}
for $V$ in \eqref{V:Lq-regular}. We plug \eqref{eq30 : Theorem : Toy Fokker-Plank} and \eqref{V_rho_est} into \eqref{eq26 : Theorem : Toy Fokker-Plank} to get \eqref{regular:narrow-distance}. 

\noindent $\bullet$ Let $V\in \tilde{\mathsf{S}}_{m,q}^{(\tilde{q}_1, \tilde{q}_2)} $ in \eqref{V:tilde:Lq-regular}. 
Then because of the embedding relation in \eqref{V_embedding} for $1-\frac{1}{d}< m<1$, now \eqref{V_rho_est} yields
 \begin{equation*}\label{V_rho_est_tilde}
\int_{s}^{t} \int_{\Omega}  | V |\rho  \,dx\,d\tau \leq C\|V\|_{\tilde{\mathsf{S}}_{m,q}^{(q_1,q_2)}} \,|t-s|^{\frac{1}{2}} \quad \mbox{for} \quad |t-s| <1 .
\end{equation*}
$\bullet$ Let divergence-free $V\in \mathcal{D}_{m,q}^{(q_1,q_2)}$ satisfying \eqref{T:weakSol:divfree:Vq} for $q\geq 1$. Because $\nabla \cdot V=0$, energy estimates \eqref{eq28 : Theorem : Toy Fokker-Plank} or \eqref{eq29 : Theorem : Toy Fokker-Plank}, 
and Lemma~\ref{L:interpolation} with $p=q$ implies that 
\[
\rho \in L_{x,t}^{r_1,r_2}  \ \text{ for }  \  \frac{d}{r_1} + \frac{2+q_{m,d}}{r_2} = \frac{d}{q},
\]
independent of $V$. For $q\geq 1$ and $ (1-\frac{2q}{d})_{+} \leq m < 1$, suppose that $V$ satisfies $\nabla \cdot V =0$ and \eqref{T:V:divfree_delta}.
Then the corresponding pairs $(\hat{r}_1, \hat{r}_2)$ satisfy
\[
\frac{d}{\hat{r}_1} + \frac{2+q_{m,d}}{\hat{r}_2} > \frac{d}{q}.
\]
Let us fix $r_1 = \hat{r}_1$. Then, we have $\frac{1}{\hat{r}_2} - \frac{1}{r_2} > 0$ and
 \begin{equation}\label{V_rho_est_divfree}
\int_{s}^{t} \int_{\Omega}  | V |\rho  \,dx\,d\tau \leq C\|V\|_{L_{x,t}^{q_1,q_2}} \,|t-s|^{\frac{1}{\hat{r}_2} - \frac{1}{r_2}},
\end{equation}
where $\frac{1}{\hat{r}_2} - \frac{1}{r_2}= \frac{2+\frac{d(q+m-2)}{q} - \left(\frac{d}{q_1} + \frac{2+q_{m,d}}{q_2}\right)}{2+q_{m,d}}>0$.
We plug \eqref{eq30 : Theorem : Toy Fokker-Plank} and \eqref{V_rho_est_divfree}  into \eqref{eq26 : Theorem : Toy Fokker-Plank} to get \eqref{regular:narrow-distance:2} with $a:=\min \{\frac{1}{\hat{r}_2} - \frac{1}{r_2}, \frac{1}{2} \}$.
This completes the proof. \qed

\section{Porous medium equation in case $\nabla \cdot V = 0$}\label{Appendix:PME}

Now we consider the porous medium equation (PME) with the divergence form of drifts in the form of 
\begin{equation}\label{E:PME}
\begin{cases}
\partial_t \rho =   \nabla \cdot( \nabla \rho^m  - V\rho), \ m>1,  & \text{ in } \  \Omega_{T}:= \Omega \times (0, T), \vspace{1 mm}\\
 \left( \nabla \rho^m - V\rho \right)\cdot \textbf{n} = 0 , &\text{ on } \ \partial \Omega \times (0, T) \vspace{1mm} \\
 \rho(\cdot,0)=\rho_0 \in L^q (\Omega), \ q\geq 1, & \text{ on } \ \Omega,
\end{cases}
\end{equation}
where $\Omega \Subset \bbr^d$ with smooth boundaries, $d\geq 2$, $0 <T < \infty$, and $V: \Omega_{T} \to \mathbb{R}^d$ is a given vector field to be specified later, and the vector $\textbf{n}$ is normal to the boundary $\partial \Omega$. 

In particular, assuming $\nabla \cdot V = 0$, we improve the existence results compare to results in \cite[Section~2.1.2]{HKK02} by adopting extended version of compactness arguments in Proposition~\ref{Proposition : compactness} and speed estimates in Lemma~\ref{L:speed} which includes super-critical regime of the scaling invariant classes. First, we rewrite Lemma~\ref{L:speed} for PME.   

\begin{lemma}\label{L:speed:PME} (Speed estimate of PME)
	Let $d\geq 2$, $m>1$ and $q \geq m$.  Suppose that $\rho:\Omega_{T} \mapsto \mathbb{R}$ is nonnegative measurable function satisfying 
	 \begin{equation}\label{FS:PME}
 \rho \in L^{\infty}(0, T; L^{q}(\Omega)) \quad \text{and} \quad
 \rho^{\frac{q+m-1}{2}} \in L^{2}(0, T; W^{1,2}(\Omega)).
 \end{equation}
 Furthermore, assume that 
 \begin{equation}\label{V:speed:PME}
 \begin{gathered}
  V \in L_{x,t}^{q_1, q_2}\cap \calC^{\infty} (\overline{\Omega}_{T})
  \ \text{ where }\ \frac{d}{q_1} + \frac{2 + q_{m,d}}{q_2} = 1+ \frac{d(q+m-2)}{2q} \vspace{1 mm}\\
   \ \text{ for }  \  
\begin{cases}
\frac{q-1}{2q}\leq \frac{1}{q_1} \leq \frac{q-1}{2q} + \frac{2+q_{m,d}}{2d(m+q-1)}, \ \frac{m+q-2}{2(m+q-1)}	\leq \frac{1}{q_2} \leq \frac{1}{2}, & \text{if } d>2, \vspace{1mm}\\
\frac{q-1}{2q}\leq \frac{1}{q_1} < \frac{q-1}{2q} + \frac{2+q_{m,d}}{2d(m+q-1)}, \ \frac{m+q-2}{2(m+q-1)}	< \frac{1}{q_2} \leq \frac{1}{2}, & \text{if } d=2.
\end{cases} 
 \end{gathered}
 \end{equation}
 Then the following holds
 \begin{equation*}
 	\iint_{\Omega_{T}} \{ \abs{\frac{\nabla \rho^m}{\rho}}^2 \rho + \abs{V}^2\rho \} \,dxdt \leq C
\end{equation*}
where the constant $C$ depends on $\Omega$, $\rho_0$, and $V$. 
\end{lemma}

\begin{proof}
First, by testing $m\rho^{m-1}$ to \eqref{E:PME}, we obtain the following:  
\[
 \frac{d}{dt}\int_{\Omega} \rho^m \,dx + (m-1)\int_{\Omega} \left|\frac{\nabla \rho^{m}}{\rho}\right|^2 \rho \, dx 
 = (m-1) \int_{\Omega}  \rho V \cdot \frac{\nabla \rho^m}{\rho} \,dx .
\]
Now the integration in terms of $t$ yields the following
\begin{equation}\label{speed01:PME}
 \int_{\Omega} \rho^m (\cdot, T) \,dx + (m-1)\iint_{\Omega_{T}} \left|\frac{\nabla \rho^{m}}{\rho}\right|^2 \rho \, dx dt
 =  \int_{\Omega} \rho_{0}^{m} \,dx + 
(m-1) \iint_{\Omega_T}  \rho V \cdot \frac{\nabla \rho^m}{\rho} \,dxdt. 
\end{equation}
Because $\rho$ is nonnegative, we ignore the first term on the left. Moreover, as long as $q\geq m$, it holds 
\[
\int_{\Omega} \rho_{0}^{m} \,dx \leq \abs{\Omega}^{1-\frac{m}{q}} \|\rho_0\|_{L^q_x (\Omega)}.
\]
The Young's inequality provides that 
\[
(m-1)\iint_{\Omega_T}  \rho V \cdot \frac{\nabla \rho^m}{\rho} \,dxdt 
\leq  \frac{m-1}{2}\iint_{\Omega_T} \abs{\frac{\nabla \rho^m}{\rho}}^2 \rho \,dxdt + 2(m-1)\iint_{\Omega_T} |V|^2 \rho \,dxdt,
\]
which the first term on the right is absorbed to the left of \eqref{speed01:PME}. Therefore, it leads that
\begin{equation*}\label{speed02:PME}
 \iint_{\Omega_{T}} \left|\frac{\nabla \rho^{m}}{\rho}\right|^2 \rho \, dx dt
 \leq C(m,|\Omega|, \|\rho_0\|_{L^{1}_{x}(\Omega)}) + 
4 \iint_{\Omega_T} \abs{V}^2 \rho  \,dxdt.
\end{equation*}

Now, based on the hypotheses \eqref{FS:PME}, we apply the interpolation inequality in \cite[Lemma~3.2]{HKK02} that  
\[
\iint_{\Omega_{T}} |V|^2 \rho \,dxdt \leq \|V\|^{2}_{L_{x,t}^{q_1, q_2}} \|\rho\|_{L_{x,t}^{r_1, r_2}}, \quad r_1 = \frac{q_1}{q_1 - 2}, \ r_2 = \frac{q_2}{q_2 - 2}
\]
by taking the H\"{o}lder inequalities $\frac{2}{q_1} + \frac{q_1 - 2}{q_1} = 1$ in the spatial variable and $\frac{2}{q_2} + \frac{q_2 - 2}{q_2} = 1$ in the temporal variable, respectively, for $q_1, q_2 > 2$. Then the suitable range of $(r_1, r_2)$ is satisfying \eqref{L:r1r2} with $p=q$ which corresponds to the range of $(q_1, q_2)$ satisfying \eqref{V:speed:PME}.
\end{proof}

Next is a general version of compactness arguments for PME (cf. Proposition~\ref{Proposition : compactness}).
\begin{proposition}\label{P:AL:PME} (Compactness argument of PME)
 Let $q\geq 1$ and $m>1$.
Suppose that $\rho$ is a regular solution of \eqref{E:PME} with $\rho_{0}\in  L^{q}(\Omega) \cap \calC^{\alpha}(\overline{\Omega})$ such that 
\begin{equation*}\label{FS01:PME}
\rho \in L^{\infty} \left(0, T; L^q (\Omega)\right) \quad \text{ and } \quad 
\rho^{\frac{q+m-1}{2}} \in L^2 \left(0, T; W^{1,2}(\Omega)\right).
\end{equation*}
Assume that 
\begin{equation}\label{V:compact:PME}
\begin{gathered}
	V \in L_{x,t}^{q_1, q_2}\cap \calC^{\infty} (\overline{\Omega}_{T})
	\ \text{ where } \  \frac{d}{q_1} + \frac{2 + q_{m,d}}{q_2} = 2+\frac{d(q+m-2)}{q} \vspace{1 mm}\\
	\text{ for } \ 
	\begin{cases}
		\frac{q-1}{q}\leq \frac{1}{q_1} \leq \frac{q-1}{2q} + \frac{2+q_{m,d}}{d(m+q-1)}, \ \frac{m+q-2}{q+m-1} \leq \frac{1}{q_2} \leq 1, & \text{if } d > 2, \vspace{1 mm}\\
		\frac{q-1}{q} \leq \frac{1}{q_1} < \frac{q-1}{2q} + \frac{2+q_{m,d}}{d(m+q-1)}, \ \frac{m+q-2}{q+m-1} < \frac{1}{q_2} \leq 1, & \text{if } d = 2.
	\end{cases} 
\end{gathered} 
\end{equation}
Then $\rho$ is relatively compact in the following space
\begin{equation*}\label{compact:PME}
\rho \in L^{\alpha, \beta}_{x,t} \quad \text{ where } \ \alpha < r_1 \ \text{ and } \ \beta <r_2 
\end{equation*}
for $(r_1, r_2)$ satisfying \eqref{L:r1r2} for $p=q$. 
\end{proposition}

\begin{proof}
The proof follows the compactness argument in Proposition~\ref{Proposition : compactness}; we only indicate the point where the PME case requires a different truncation as in \cite[Section~6.3.2]{BDG13}, \cite{HKKP}.

If $1\le q\le 3-m$, which can occur only when $1\le m\le 3-q$, then the same estimate as in \eqref{nabla_rho} gives $\nabla \rho \in L^\gamma_{x,t}(\Omega_T)$ 
for some $\gamma>1$. Hence the Aubin--Lions argument in Case 1 of Proposition~\ref{Proposition : compactness} applies.

It remains to consider $q>3-m$. In this range the exponent $(3-m-q)/2$ is negative. Therefore, instead of the truncation from above used in the fast-diffusion case, we use the truncation from below
\[
w_n^{(k)}:=T_k^{-}(\rho_n):=\max\{\rho_n,1/k\}.
\]
Since $w_n^{(k)}\ge 1/k$, we have
\[
|\nabla w_n^{(k)}|
=
\frac{2}{m+q-1}
\left(w_n^{(k)}\right)^{\frac{3-m-q}{2}}
\left|
\nabla \left(w_n^{(k)}\right)^{\frac{m+q-1}{2}}
\right|
\le
C(k,m,q)
\left|
\nabla \rho_n^{\frac{m+q-1}{2}}
\right|.
\]
Thus the localized compactness and diagonal argument used in Case 2 of Proposition~\ref{Proposition : compactness} applies to $w_n^{(k)}$. Passing $k\to\infty$, we obtain the desired relative compactness of $\rho_n$.
\end{proof}

Now we deliver updated existence results of PME in case $\nabla \cdot V = 0$, see \cite[Theorem~2.8 \& 2.10]{HKK02}. 

\begin{theorem}\label{T:weakSol:divfree:PME}
Let $d\geq 2$, $m>1$ and $q\geq 1$.  
Suppose that 
\begin{equation}\label{T:weakSol:divfree:V:PME}
	V \in \mathcal{D}_{m,q}^{(q_1, q_2)} 
	\ \text{ for } \
	\begin{cases}
		0 \leq \frac{1}{q_1} \leq \frac{q-1}{q} + \frac{2+q_{m,d}}{d(m+q-1)}, \ 0 \leq \frac{1}{q_2} \leq 1, & \text{if } d > 2, \vspace{1 mm}\\
		0 \leq \frac{1}{q_1} < \frac{q-1}{q} + \frac{2+q_{m,d}}{d(m+q-1)}, \ 0 \leq \frac{1}{q_2} \leq 1, & \text{if } d = 2.
	\end{cases}
\end{equation}

\begin{itemize}
	\item[(i)] For $q=1$, assume that $\rho_0 \in \mathcal{P}(\Omega)$ and $\int_{\Omega} \rho_0 \log \rho_0 \,dx < \infty$. 
Then, there exists a nonnegative $L^1$-weak solution of \eqref{E:PME} in Definition~\ref{D:weak-sol} that holds \eqref{T:weakSol:E_1} with $C=C (\int_{\Omega} \rho_0 \log \rho_0 \,dx )$.

	\item[(ii)] For $q>1$, assume that $\rho_0 \in  \mathcal{P}(\Omega) \cap L^{q}(\Omega)$. Then, there exists a nonnegative $L^q$-weak solution of \eqref{E:PME} in Definition~\ref{D:weak-sol} that holds \eqref{T:weakSol:E_q} with $C=C (\|\rho_{0}\|_{L^{q} (\Omega)})$.
	
	\item[(iii)] Furthermore, suppose that 
	\begin{equation*}\label{T:weakSol:divfree:V:delta:PME}
	V \in \mathcal{D}_{m,q,+}^{(q_1, q_2)} 
	\ \text{ for } \
		0 \leq \frac{1}{q_1} < \frac{q-1}{q} + \frac{2+q_{m,d}}{d(m+q-1)}, \ 0 \leq \frac{1}{q_2} < 1.
	\end{equation*}
    Then, we have 
\begin{equation*}\label{T:weakSol:Narrow_2:PME}
\delta(\rho(t),\rho(s))\leq C(t-s)^{a},   \quad  \forall ~ 0\leq s<t\leq T 
\end{equation*}
with $a= \min \left\{ \frac{2q+d(q-1)}{2[2q+d(m+q-1)]}, \frac{2+\frac{d(q+m-2)}{q} - \left(\frac{d}{q_1} + \frac{2+q_{m,d}}{q_2}\right)}{2+q_{m,d}} \right\}$ 
and $C= \begin{cases}
	C ( \|V\|_{\mathcal{D}_{m,1,+}^{(q_1,q_2)}}, \int_{\Omega} \rho_0 \log \rho_0 \,dx ), & \text{if } q=1 , \\
	C (\|V\|_{\mathcal{D}_{m,q,+}^{(q_1,q_2)}}, \|\rho_{0}\|_{L^{q} (\Omega)} ), &\text{if } q>1.
\end{cases}$	
	
\end{itemize}

\end{theorem}

\begin{remark}
In \cite[Theorem~2.4 \& 2.12]{HKK02}, \( L^q \)-weak solutions for the PME that satisfy energy estimates are introduced when \( V \) or \( \nabla V \) belongs to scaling-invariant classes (non-divergence-free case). When \( m > 1 \) and \( V \in \mathcal{S}_{m,q}^{(q_1,q_2)} \), an \( L^q \)-weak solution \( \rho \) satisfies \( \rho \in L^{q,\infty}_{x,t} \) and \( \nabla \rho^{\frac{m+q-1}{2}} \in L^{2}_{x,t} \), which implies that \( \rho \in L_{x,t}^{m+q-1+\frac{2q}{d}} \) by parabolic embedding. Therefore, it follows that 
    \[
     \int_{s}^{t} \int_{\Omega}   | \nabla  \rho^m |  \,dx\,d\tau \leq C |t-s|^{\frac{2q+d(q-1)}{2[2q+d(m+q-1)]}},
    \]
    by following \eqref{eq30 : Theorem : Toy Fokker-Plank} with $\rho \in L_{x,t}^{m+q-1+\frac{2q}{d}}$ instead of $\rho \in L^{1,\infty}_{x,t}$.
    Since $V\in \mathcal{S}_{m,q}^{(q_1,q_2)}$ in \eqref{Scaling_invariant_class}, this corresponds to 
    \[
    \rho \in L_{x,t}^{r_1,r_2}  \ \text{ for }  \  \frac{d}{r_1} + \frac{2+q_{m,d}}{r_2} = \frac{d}{q}, 
    \ \text{ and } \ 
    \frac{d}{\hat{r}_1} + \frac{2+q_{m,d}}{\hat{r}_2} = 1+d.
    \]
    Therefore, as in \eqref{eq31 : Theorem : Toy Fokker-Plank} and \eqref{V_rho_est}, this yields that 
    \[
     \int_{s}^{t} \int_{\Omega}  | V |\rho  \,dx\,d\tau \leq C\|V\|_{\mathcal{S}_{m,q}^{(q_1,q_2)}}|t-s|^{\frac{q+d(q-1)}{2q+d(m-1)}}.
    \]
    Hence, we expect the $L^q$-weak solution to hold 
	 \begin{equation*}\label{T:weakSol:Narrow_1:PME}
\delta(\rho(t),\rho(s))\leq C(t-s)^{a},   \quad  \forall ~ 0\leq s<t\leq T, 
\end{equation*}
where 
    \[
    a=\min\left\{\frac{2q+d(q-1)}{2[2q+d(m+q-1)]}, \ \frac{q+d(q-1)}{2q+d(m-1)}\right\}.
    \]
\end{remark}

The following theorem is for an absolutely continuous $L^q$-weak solution when $q\geq m$, see \cite[Theorem~2.12]{HKK02}. 
\begin{theorem}\label{T:ACweakSol:divfree:PME}
Let $d\geq 2$, $m>1$, and $q\geq m$.  
Suppose that  $\rho_0 \in  \mathcal{P}(\Omega) \cap L^{q}(\Omega)$ and 
\begin{equation}\label{T:ACweakSol:divfree:V:PME}
\begin{aligned}
		 V \in 
			\mathcal{D}_{m,q,s}^{(q_1,q_2)}  
		\  \text{ for } \ 
	\begin{cases}
0 \leq \frac{1}{q_1} \leq \frac{q-1}{2q} + \frac{2+q_{m,d}}{2d(m+q-1)},  \ 0 \leq \frac{1}{q_2} \leq \frac{1}{2}, & \text{if } d>2, \vspace{1mm}\\
0 \leq \frac{1}{q_1} < \frac{q-1}{2q} + \frac{2+q_{m,d}}{2d(m+q-1)}, \ 0 \leq \frac{1}{q_2} \leq \frac{1}{2}, & \text{if } d=2.
\end{cases} 
\end{aligned}
\end{equation}
Then, there exists a nonnegative $L^q$-weak solution of \eqref{E:PME} in Definition~\ref{D:weak-sol} such that  $\rho \in AC(0,T; \mathcal{P} (\Omega))$ with $\rho(\cdot, 0)=\rho_0$.
Furthermore, $\rho$ satisfies 
\[
 \sup_{0\leq t \leq T} \int_{\mathbb{R}^d}  \rho^q  \,dx
 + \iint_{\Omega_T} \{ \left| \nabla \rho^{\frac{q+m-1}{2}}\right |^2 + \left(\abs{\frac{\nabla \rho^m}{\rho}}^{2}+|V|^{2}\right ) \rho \} \,dx\,dt \leq C,
\]
and 
\[
 W_{2}(\rho(t),\rho(s))\leq C (t-s)^{\frac{1}{2}},\qquad
\forall ~~0\leq s\leq t\leq T,
\]
with $C = C (\|V\|_{\mathcal{D}_{m,q,s}^{(q_1,q_2)}}, \,\|\rho_{0}\|_{L^{q} (\Omega)})$.
\end{theorem}


\begin{bibdiv}
\begin{biblist}

\bib{ags:book}{book}{
    AUTHOR = {Ambrosio, L.}, 
    author = {Gigli, N.},
    author = {Savar\'{e}, G.},
     TITLE = {Gradient flows in metric spaces and in the space of
              probability measures},
    SERIES = {Lectures in Mathematics ETH Z\"{u}rich},
 PUBLISHER = {Birkh\"{a}user Verlag, Basel},
      YEAR = {2005},
     PAGES = {viii+333},
      ISBN = {978-3-7643-2428-5; 3-7643-2428-7},
   MRCLASS = {49-02 (28A33 35K55 35K90 49Q20 60B10)},
  MRNUMBER = {2129498},
}

\bib{BDG13}{article}{
      author={B\"{o}gelein, V.},
      author={Duzaar, F.},
      author={Gianazza, U.},
       title={Porous medium type equations with measure data and potential
  estimates},
        date={2013},
        ISSN={0036-1410},
     journal={SIAM J. Math. Anal.},
      volume={45},
      number={6},
       pages={3283\ndash 3330},
         url={https://doi.org/10.1137/130925323},
      review={\MR{3124895}},
}


\bib{BDG15}{article}{
      author={B\"{o}gelein, V.},
      author={Duzaar, F.},
      author={Gianazza, U.},
       title={Very weak solutions of singular porous medium equations with
  measure data},
        date={2015},
        ISSN={1534-0392},
     journal={Commun. Pure Appl. Anal.},
      volume={14},
      number={1},
       pages={23\ndash 49},
         url={https://doi.org/10.3934/cpaa.2015.14.23},
      review={\MR{3299023}},
}



\bib{BonFig2024}{article}{
    AUTHOR = {Bonforte, M.}, 
    author = {Figalli, A.},
     TITLE = {The {C}auchy-{D}irichlet problem for the fast diffusion
              equation on bounded domains},
   JOURNAL = {Nonlinear Anal.},
  FJOURNAL = {Nonlinear Analysis. Theory, Methods \& Applications. An
              International Multidisciplinary Journal},
    VOLUME = {239},
      YEAR = {2024},
     PAGES = {Paper No. 113394, 55},
      ISSN = {0362-546X,1873-5215},
   MRCLASS = {35K55 (35B40 35J20 35K67 35P30)},
  MRNUMBER = {4658530},
MRREVIEWER = {Si\ Ning\ Zheng},
       DOI = {10.1016/j.na.2023.113394},
       URL = {https://doi.org/10.1016/j.na.2023.113394},
}

\bib{Bonforte_Vazquez_2006}{article}{
    AUTHOR = {Bonforte, M.}, 
    author = {Vazquez, J. L.},
     TITLE = {Global positivity estimates and {H}arnack inequalities for the
              fast diffusion equation},
   JOURNAL = {J. Funct. Anal.},
  FJOURNAL = {Journal of Functional Analysis},
    VOLUME = {240},
      YEAR = {2006},
    NUMBER = {2},
     PAGES = {399--428},
      ISSN = {0022-1236,1096-0783},
   MRCLASS = {35K55 (35B45 35B65)},
  MRNUMBER = {2261689},
MRREVIEWER = {Yoshio\ Yamada},
       DOI = {10.1016/j.jfa.2006.07.009},
       URL = {https://doi.org/10.1016/j.jfa.2006.07.009},
}

\bib{Bonforte_Vazquez_2007}{article}{
    AUTHOR = {Bonforte, M.}, 
    author = {Vazquez, J. L.},
     TITLE = {Reverse smoothing effects, fine asymptotics, and {H}arnack
              inequalities for fast diffusion equations},
   JOURNAL = {Bound. Value Probl.},
  FJOURNAL = {Boundary Value Problems},
      YEAR = {2007},
     PAGES = {Art. ID 21425, 31},
      ISSN = {1687-2762,1687-2770},
   MRCLASS = {35K55 (35B45 35K15)},
  MRNUMBER = {2291926},
       DOI = {10.1155/2007/21425},
       URL = {https://doi.org/10.1155/2007/21425},
}



\bib{Cannon_DiBenedetto}{article}{
    AUTHOR = {Cannon, J. R.},
    author = {DiBenedetto, E.},
     TITLE = {The initial value problem for the {B}oussinesq equations with
              data in {$L\sp{p}$}},
 BOOKTITLE = {Approximation methods for {N}avier-{S}tokes problems ({P}roc.
              {S}ympos., {U}niv. {P}aderborn, {P}aderborn, 1979)},
    SERIES = {Lecture Notes in Math.},
    VOLUME = {771},
     PAGES = {129--144},
 PUBLISHER = {Springer, Berlin},
      YEAR = {1980},
      ISBN = {3-540-09734-1},
   MRCLASS = {35Q20 (76D05 76R05)},
  MRNUMBER = {565993},
MRREVIEWER = {Alberto\ Valli},
}

\bib{CCY2019}{article}{
    AUTHOR = {Carrillo, J. A.},
    author = {Craig, K.},
    author = {Yao, Y.},
     TITLE = {Aggregation-diffusion equations: dynamics, asymptotics, and
              singular limits},
 BOOKTITLE = {Active particles. {V}ol. 2. {A}dvances in theory, models, and
              applications},
    SERIES = {Model. Simul. Sci. Eng. Technol.},
     PAGES = {65--108},
 PUBLISHER = {Birkh\"auser/Springer, Cham},
      YEAR = {2019},
      ISBN = {978-3-030-20296-5; 978-3-030-20297-2},
   MRCLASS = {82C24 (35K57 35Q82 82C22)},
  MRNUMBER = {3932458},
}		

\bib{CarDFL2022}{article}{
    AUTHOR = {Carrillo, J. A.}, 
    author = {Delgadino, M. G.},
    author = {Frank, R. L.},
    author = {Lewin, M.},
     TITLE = {Fast diffusion leads to partial mass concentration in
              {K}eller-{S}egel type stationary solutions},
   JOURNAL = {Math. Models Methods Appl. Sci.},
  FJOURNAL = {Mathematical Models and Methods in Applied Sciences},
    VOLUME = {32},
      YEAR = {2022},
    NUMBER = {4},
     PAGES = {831--850},
      ISSN = {0218-2025,1793-6314},
   MRCLASS = {35A23 (26D15 35K55 46E35 49J40)},
  MRNUMBER = {4421218},
       DOI = {10.1142/S021820252250018X},
       URL = {https://doi.org/10.1142/S021820252250018X},
}

\bib{CarGG2024}{article}{
   author={Carrillo, J. A.},
   author={Fern\'andez-Jim\'enez, A.},
   author={G\'omez-Castro, D.},
   title={Partial mass concentration for fast-diffusions with non-local
   aggregation terms},
   journal={J. Differential Equations},
   volume={409},
   date={2024},
   pages={700--773},
   issn={0022-0396},
   review={\MR{4784796}},
   doi={10.1016/j.jde.2024.08.013},
}

\bib{CarGVaz2022}{article}{
    AUTHOR = {Carrillo, J. A.}, 
    author = {G\'omez-Castro, D.},
    author = {V\'azquez, J. L.},
     TITLE = {Infinite-time concentration in aggregation-diffusion equations
              with a given potential},
   JOURNAL = {J. Math. Pures Appl. (9)},
  FJOURNAL = {Journal de Math\'ematiques Pures et Appliqu\'ees. Neuvi\`eme
              S\'erie},
    VOLUME = {157},
      YEAR = {2022},
     PAGES = {346--398},
      ISSN = {0021-7824,1776-3371},
   MRCLASS = {35K55 (35B40 35D40 35K65 35Q84)},
  MRNUMBER = {4351080},
       DOI = {10.1016/j.matpur.2021.11.002},
       URL = {https://doi.org/10.1016/j.matpur.2021.11.002},
}

\bib{Chen_DiBenedetto}{article}{
    AUTHOR = {Chen, Y. Z.}, 
    author = {DiBenedetto, E.},
     TITLE = {H\"older estimates of solutions of singular parabolic
              equations with measurable coefficients},
   JOURNAL = {Arch. Rational Mech. Anal.},
  FJOURNAL = {Archive for Rational Mechanics and Analysis},
    VOLUME = {118},
      YEAR = {1992},
    NUMBER = {3},
     PAGES = {257--271},
      ISSN = {0003-9527},
   MRCLASS = {35K65 (35D10 35K55 35K60)},
  MRNUMBER = {1158938},
MRREVIEWER = {Jes\'us\ Hern\'andez},
       DOI = {10.1007/BF00387898},
       URL = {https://doi.org/10.1007/BF00387898},
}


\bib{DB93}{book}{
    AUTHOR = {DiBenedetto, E.},
     TITLE = {Degenerate parabolic equations},
    SERIES = {Universitext},
 PUBLISHER = {Springer-Verlag, New York},
      YEAR = {1993},
     PAGES = {xvi+387},
      ISBN = {0-387-94020-0},
   MRCLASS = {35K65 (35-02)},
  MRNUMBER = {1230384},
MRREVIEWER = {Ya Zhe Chen},
       DOI = {10.1007/978-1-4612-0895-2},
       URL = {https://doi.org/10.1007/978-1-4612-0895-2},
}

\bib{DGV12}{book}{
    AUTHOR = {DiBenedetto, E.},
    author = {Gianazza, U.}, 
    author = {Vespri, V.},
     TITLE = {Harnack's inequality for degenerate and singular parabolic
              equations},
    SERIES = {Springer Monographs in Mathematics},
 PUBLISHER = {Springer, New York},
      YEAR = {2012},
     PAGES = {xiv+278},
      ISBN = {978-1-4614-1583-1},
   MRCLASS = {35-02 (35B45 35B65 35K59 35K65 35K67)},
  MRNUMBER = {2865434},
MRREVIEWER = {Alain Brillard},
       DOI = {10.1007/978-1-4614-1584-8},
       URL = {https://doi.org/10.1007/978-1-4614-1584-8},
}

\bib{DiBenedetto_Kwong}{article}{
    AUTHOR = {DiBenedetto, E.}, 
    author = {Kwong, Y. C.},
     TITLE = {Harnack estimates and extinction profile for weak solutions of
              certain singular parabolic equations},
   JOURNAL = {Trans. Amer. Math. Soc.},
  FJOURNAL = {Transactions of the American Mathematical Society},
    VOLUME = {330},
      YEAR = {1992},
    NUMBER = {2},
     PAGES = {783--811},
      ISSN = {0002-9947,1088-6850},
   MRCLASS = {35B45 (35B05 35K55 35K65)},
  MRNUMBER = {1076615},
MRREVIEWER = {H.\ J.\ Kuiper},
       DOI = {10.2307/2153935},
       URL = {https://doi.org/10.2307/2153935},
}

\bib{DiBenedetto_Kwong_Vespri}{article}{
    AUTHOR = {DiBenedetto, E.},
    author = {Kwong, Y.}, 
    author = {Vespri, V.},
     TITLE = {Local space-analyticity of solutions of certain singular
              parabolic equations},
   JOURNAL = {Indiana Univ. Math. J.},
  FJOURNAL = {Indiana University Mathematics Journal},
    VOLUME = {40},
      YEAR = {1991},
    NUMBER = {2},
     PAGES = {741--765},
      ISSN = {0022-2518,1943-5258},
   MRCLASS = {35B65 (35B50 35K55)},
  MRNUMBER = {1119195},
MRREVIEWER = {R.\ Schumann},
       DOI = {10.1512/iumj.1991.40.40033},
       URL = {https://doi.org/10.1512/iumj.1991.40.40033},
}


\bib{Freitag}{article}{
    AUTHOR = {Freitag, M.},
     TITLE = {Global existence and boundedness in a chemorepulsion system
              with superlinear diffusion},
   JOURNAL = {Discrete Contin. Dyn. Syst.},
  FJOURNAL = {Discrete and Continuous Dynamical Systems. Series A},
    VOLUME = {38},
      YEAR = {2018},
    NUMBER = {11},
     PAGES = {5943--5961},
      ISSN = {1078-0947},
   MRCLASS = {35K51 (35D30 92C17)},
  MRNUMBER = {3917794},
       DOI = {10.3934/dcds.2018258},
       URL = {https://doi.org/10.3934/dcds.2018258},
}

\bib{HKK01}{article}{
    AUTHOR = {Hwang, S.}, 
    author = {Kang, K.}, 
    author = {Kim, H. K.},
     TITLE = {Existence of weak solutions for porous medium equation with a
              divergence type of drift term},
   JOURNAL = {Calc. Var. Partial Differential Equations},
  FJOURNAL = {Calculus of Variations and Partial Differential Equations},
    VOLUME = {62},
      YEAR = {2023},
    NUMBER = {4},
     PAGES = {Paper No. 126},
      ISSN = {0944-2669},
   MRCLASS = {35A01 (35K55 35Q84 92B05)},
  MRNUMBER = {4568176},
       DOI = {10.1007/s00526-023-02451-4},
       URL = {https://doi.org/10.1007/s00526-023-02451-4},
}

\bib{HKK02}{article}{
    AUTHOR = {Hwang, S.}, 
    author = {Kang, K.}, 
    author = {Kim, H. K.},
     TITLE = {Existence of weak solutions for porous medium equation with a
              divergence type of drift term in a bounded domain},
   JOURNAL = {J. Differential Equations},
  FJOURNAL = {Journal of Differential Equations},
    VOLUME = {389},
      YEAR = {2024},
     PAGES = {361--414},
      ISSN = {0022-0396,1090-2732},
   MRCLASS = {35D30 (35K55 35Q84 92C17)},
  MRNUMBER = {4697984},
       DOI = {10.1016/j.jde.2024.01.028},
       URL = {https://doi.org/10.1016/j.jde.2024.01.028},
}

\bib{HKKP}{article}{
    AUTHOR = {Hwang, S.}, 
    author = {Kang, K.}, 
    author = {Kim, H. K.},
	AUTHOR = {Park, J. T.},
     TITLE = {Existence of weak solutions for nonlinear drift-diffusion equations with measure data},
   JOURNAL = {arXiv:2501.07847 [math.AP]},
       URL = {https://doi.org/10.48550/arXiv.2501.07847},
}

\bib{HZ21}{article}{
    AUTHOR = {Hwang, S.}, 
    author = {Zhang, Y. P.},
     TITLE = {Continuity results for degenerate diffusion equations with
              {$L_t^p L_x^q$} drifts},
   JOURNAL = {Nonlinear Anal.},
  FJOURNAL = {Nonlinear Analysis. Theory, Methods \& Applications. An
              International Multidisciplinary Journal},
    VOLUME = {211},
      YEAR = {2021},
     PAGES = {Paper No. 112413, 37},
      ISSN = {0362-546X},
   MRCLASS = {35B65 (35K55 35K65)},
  MRNUMBER = {4265723},
MRREVIEWER = {Eurica Henriques},
       DOI = {10.1016/j.na.2021.112413},
       URL = {https://doi.org/10.1016/j.na.2021.112413},
}


\bib{KK-SIMA}{article}{
    AUTHOR = {Kang, K.}, 
    author = {Kim, H. K.},
     TITLE = {Existence of weak solutions in {W}asserstein space for a
              chemotaxis model coupled to fluid equations},
   JOURNAL = {SIAM J. Math. Anal.},
  FJOURNAL = {SIAM Journal on Mathematical Analysis},
    VOLUME = {49},
      YEAR = {2017},
    NUMBER = {4},
     PAGES = {2965--3004},
      ISSN = {0036-1410},
   MRCLASS = {92B05 (35K55 35Q35 35Q84)},
  MRNUMBER = {3682182},
       DOI = {10.1137/16M1083232},
       URL = {https://doi.org/10.1137/16M1083232},
}

\bib{KZ18}{article}{
    AUTHOR = {Kim, I.}, 
    author = {Zhang, Y. P.},
     TITLE = {Regularity properties of degenerate diffusion equations with
              drifts},
   JOURNAL = {SIAM J. Math. Anal.},
  FJOURNAL = {SIAM Journal on Mathematical Analysis},
    VOLUME = {50},
      YEAR = {2018},
    NUMBER = {4},
     PAGES = {4371--4406},
      ISSN = {0036-1410},
   MRCLASS = {35K59 (35B65 35K15 35K65)},
  MRNUMBER = {3842921},
MRREVIEWER = {Patrick Guidotti},
       DOI = {10.1137/17M1159749},
       URL = {https://doi.org/10.1137/17M1159749},
}


\bib{Santambrosio15}{book}{
    AUTHOR = {Santambrogio, F.},
     TITLE = {Optimal transport for applied mathematicians},
    SERIES = {Progress in Nonlinear Differential Equations and their
              Applications},
    VOLUME = {87},
      NOTE = {Calculus of variations, PDEs, and modeling},
 PUBLISHER = {Birkh\"{a}user/Springer, Cham},
      YEAR = {2015},
     PAGES = {xxvii+353},
      ISBN = {978-3-319-20827-5; 978-3-319-20828-2},
   MRCLASS = {49-02 (35J96 49J45 49M29 58E50 90C05 90C48 91B02)},
  MRNUMBER = {3409718},
MRREVIEWER = {Luigi De Pascale},
       DOI = {10.1007/978-3-319-20828-2},
       URL = {https://doi.org/10.1007/978-3-319-20828-2},
}
		
\bib{SSSZ}{article}{
    AUTHOR = {Seregin, G.},
    author = {Silvestre, L.},
    author = {\v{S}ver\'{a}k, V.}, 
    author = {Zlato\v{s}, A.},
     TITLE = {On divergence-free drifts},
   JOURNAL = {J. Differential Equations},
  FJOURNAL = {Journal of Differential Equations},
    VOLUME = {252},
      YEAR = {2012},
    NUMBER = {1},
     PAGES = {505--540},
      ISSN = {0022-0396},
   MRCLASS = {35Q35 (35A09 35B45 35B53 35B65 35J15 35K10)},
  MRNUMBER = {2852216},
MRREVIEWER = {Alain Brillard},
       DOI = {10.1016/j.jde.2011.08.039},
       URL = {https://doi.org/10.1016/j.jde.2011.08.039},
}



\bib{Sim87}{article}{
    AUTHOR = {Simon, J.},
     TITLE = {Compact sets in the space {$L^p(0,T;B)$}},
   JOURNAL = {Ann. Mat. Pura Appl. (4)},
  FJOURNAL = {Annali di Matematica Pura ed Applicata. Serie Quarta},
    VOLUME = {146},
      YEAR = {1987},
     PAGES = {65--96},
      ISSN = {0003-4622},
   MRCLASS = {46E40 (46E30)},
  MRNUMBER = {916688},
MRREVIEWER = {James Bell Cooper},
       DOI = {10.1007/BF01762360},
       URL = {https://doi.org/10.1007/BF01762360},
}

\bib{Vaz06}{book}{
    AUTHOR = {V\'{a}zquez, J. L.},
     TITLE = {Smoothing and decay estimates for nonlinear diffusion
              equations},
    SERIES = {Oxford Lecture Series in Mathematics and its Applications},
    VOLUME = {33},
      NOTE = {Equations of porous medium type},
 PUBLISHER = {Oxford University Press, Oxford},
      YEAR = {2006},
     PAGES = {xiv+234},
      ISBN = {978-0-19-920297-3; 0-19-920297-4},
   MRCLASS = {35-02 (35B40 35B65 35K57 47H20)},
  MRNUMBER = {2282669},
MRREVIEWER = {Vicen\c{t}iu D. R\u{a}dulescu},
       DOI = {10.1093/acprof:oso/9780199202973.001.0001},
       URL = {https://doi.org/10.1093/acprof:oso/9780199202973.001.0001},
}


\bib{Vaz07}{book}{
    AUTHOR = {V\'{a}zquez, J. L.},
     TITLE = {The porous medium equation},
    SERIES = {Oxford Mathematical Monographs},
      NOTE = {Mathematical theory},
 PUBLISHER = {The Clarendon Press, Oxford University Press, Oxford},
      YEAR = {2007},
     PAGES = {xxii+624},
      ISBN = {978-0-19-856903-9; 0-19-856903-3},
   MRCLASS = {35-02 (35B05 35B40 35K57 76S05)},
  MRNUMBER = {2286292},
}



\bib{V}{book}{
    AUTHOR = {Villani, C.},
     TITLE = {Optimal transport},
    SERIES = {Grundlehren der mathematischen Wissenschaften [Fundamental
              Principles of Mathematical Sciences]},
    VOLUME = {338},
      NOTE = {Old and new},
 PUBLISHER = {Springer-Verlag, Berlin},
      YEAR = {2009},
     PAGES = {xxii+973},
      ISBN = {978-3-540-71049-3},
   MRCLASS = {49-02 (28A75 37J50 49Q20 53C23 58E30)},
  MRNUMBER = {2459454},
MRREVIEWER = {Dario Cordero-Erausquin},
       DOI = {10.1007/978-3-540-71050-9},
       URL = {https://doi.org/10.1007/978-3-540-71050-9},
}

\bib{W18}{article}{
    AUTHOR = {Winkler, M.},
     TITLE = {Global existence and stabilization in a degenerate
              chemotaxis-{S}tokes system with mildly strong diffusion
              enhancement},
   JOURNAL = {J. Differential Equations},
  FJOURNAL = {Journal of Differential Equations},
    VOLUME = {264},
      YEAR = {2018},
    NUMBER = {10},
     PAGES = {6109--6151},
      ISSN = {0022-0396},
   MRCLASS = {92C17 (35B40 35K55 35Q35 35Q92)},
  MRNUMBER = {3770046},
MRREVIEWER = {Tomasz Cie\'{s}lak},
       DOI = {10.1016/j.jde.2018.01.027},
       URL = {https://doi.org/10.1016/j.jde.2018.01.027},
}
		
\bib{Z11}{article}{
    AUTHOR = {Zlato\v{s}, A.},
     TITLE = {Reaction-diffusion front speed enhancement by flows},
   JOURNAL = {Ann. Inst. H. Poincar\'{e} C Anal. Non Lin\'{e}aire},
  FJOURNAL = {Annales de l'Institut Henri Poincar\'{e} C. Analyse Non Lin\'{e}aire},
    VOLUME = {28},
      YEAR = {2011},
    NUMBER = {5},
     PAGES = {711--726},
      ISSN = {0294-1449},
   MRCLASS = {35K57 (35B40 35C07 35R60 80A25)},
  MRNUMBER = {2838397},
MRREVIEWER = {Violaine Roussier-Michon},
       DOI = {10.1016/j.anihpc.2011.05.004},
       URL = {https://doi.org/10.1016/j.anihpc.2011.05.004},
}
		
\end{biblist}
\end{bibdiv}
\end{document}